\documentclass[11pt,oneside,reqno]{amsart}

\usepackage{graphicx}
\usepackage{pict2e}
\usepackage{amssymb}
\usepackage{amsthm}                
\usepackage[margin=1in]{geometry}  
\usepackage{graphics}
\usepackage{epstopdf}
\usepackage{amsmath}
\usepackage{graphicx}
\usepackage{subfigure}
\usepackage{latexsym}
\usepackage{amsmath}
\usepackage{amsfonts}
\usepackage{amssymb}
\usepackage{tikz}
\usepackage{mathdots}
\usepackage{comment}
\usepackage{float}
\usepackage[mathscr]{euscript}
\usepackage{enumerate}
\usepackage{epsfig}
\usepackage { graphicx }
\usepackage { hyperref }
\usepackage { graphics }
\usepackage { graphicx }
\usepackage{xparse}
\usepackage{epstopdf}
\usepackage {eufrak}
\usepackage[utf8]{inputenc}
\usepackage{longtable}

\usepackage[lite]{amsrefs}

\theoremstyle{definition}
\newtheorem{theorem}{Theorem}[section]
\newtheorem{lemma}[theorem]{Lemma}

\theoremstyle{definition}
\newtheorem{definition}[theorem]{Definition}
\newtheorem{example}[theorem]{Example}
\theoremstyle{remark}
\newtheorem{remark}[theorem]{Remark}
\numberwithin{equation}{section}

\numberwithin{equation}{section}

\begin{document}

\title{ The colored Jones polynomial of singular knots}

\author{Khaled Bataineh}
\address{Department of Mathematics and Statistics, Jordan University of Science and Technology, Irbid 22110 Jordan 
}
\email{khaledb@just.edu.jo }

\author{Mohamed Elhamdadi}
\address{Department of Mathematics, University of South Florida, 
Tampa, FL 33647 USA}
\email{emohamed@math.usf.edu }

\author{Mustafa Hajij}
\address{Department of Mathematics, University of South Florida, 
Tampa, FL 33647 USA}
\email{mhajij@usf.edu}





\begin{abstract}
We generalize the colored Jones polynomial to $4$-valent graphs. This generalization is given as a sequence of invariants in which the first term is a one variable specialization of the Kauffman-Vogel polynomial.  We use the invariant we construct to give a sequence of singular braid group representations.  
\end{abstract}

 \maketitle

 \tableofcontents
\section{Introduction}

The study of singular knots, or equivalently rigid $4$-valent graphs, and their invariants was generated largely by the theory of Vassiliev invariants. Many existing knot invariants have been extended to singular knot invariants. In \cite{Birman}, Birman introduced braids in the theory of Vassiliev via the singular braids and conjectured that the monoid of singular braids maps injectively into the group algebra of the braid group.  A proof of this conjecture was given by Paris in \cite{Paris}. Fiedler extended the Kauffman state models of the Jones and Alexander polynomials to the context of singular knots \cite{Fiedler}.  In \cite{Gemein } Gemein investigated extensions of the Artin representation and the Burau representation to the singular braid monoid and the relations between them. Juyumaya and Lambropoulou constructed a Jones-type invariant for singular links using a Markov trace on a variation of the Hecke algebra \cite{JL}. In \cite{kaufvog} Kauffman and Vogel defined a polynomial invariant of embedded $4$-valent graph in $\mathbb{R}^3$ extending an invariant for links in $\mathbb{R}^3$ called  the Kauffman polynomial \cite{kaufgraph}. The latter is a two variable polynomial that takes value in $\mathbb{Z}[a,a^{-1},z]$ and is an invariant of regular isotopy for links. The Kauffman polynomial of a link $L$ is denoted by $[L]$ and is defined via the following axioms :
\begin{enumerate}
\item 
$  \left[
   \begin{minipage}[h]{0.08\linewidth}
        \vspace{0pt}
        \scalebox{0.15}{\includegraphics{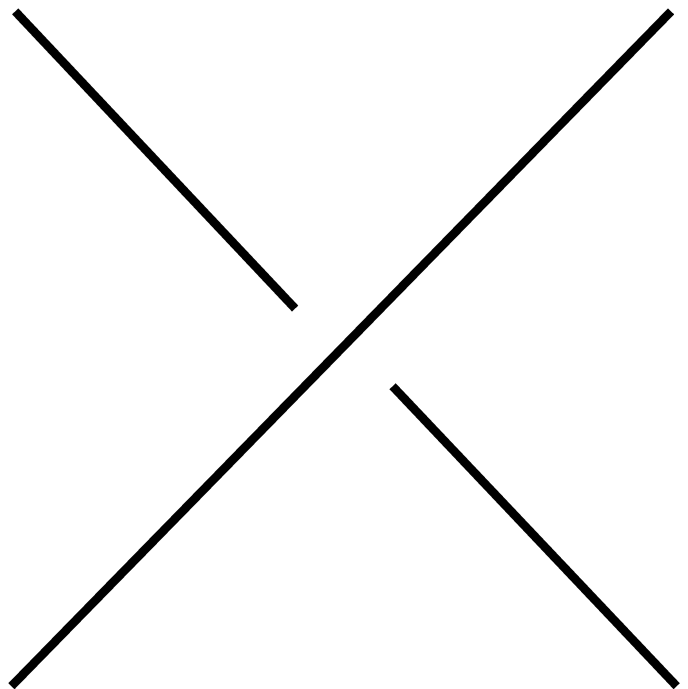}}
   \end{minipage}
  \right] - \left[
   \begin{minipage}[h]{0.08\linewidth}
        \vspace{0pt}
        \scalebox{0.15}{\includegraphics{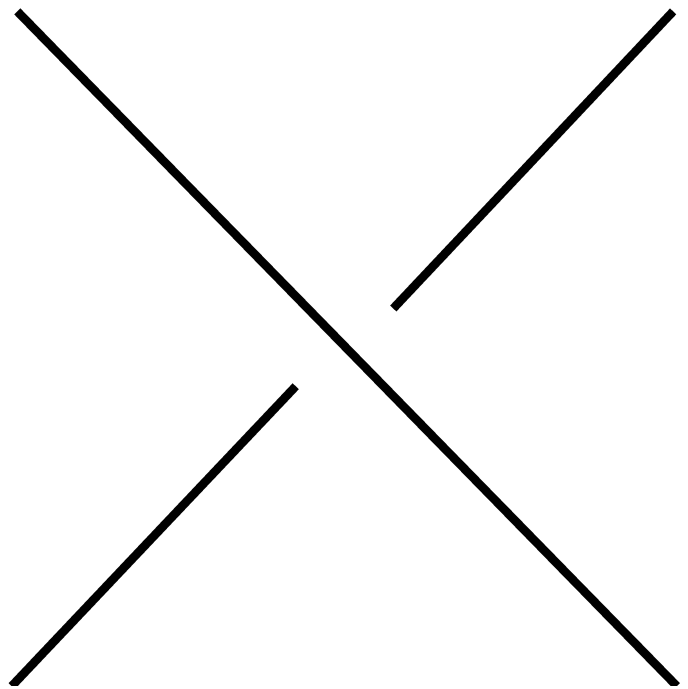}}
   \end{minipage}
  \right]=z \left(
  \left[
   \begin{minipage}[h]{0.08\linewidth}
        \vspace{0pt}
        \scalebox{0.15}{\includegraphics{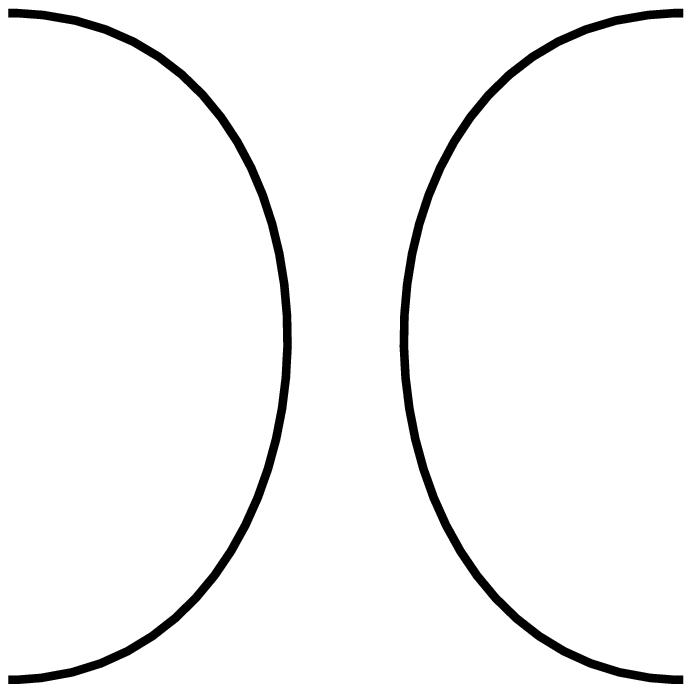}}
   \end{minipage}\right]
  -\left[
   \begin{minipage}[h]{0.08\linewidth}
        \vspace{0pt}
        \scalebox{0.15}{\includegraphics{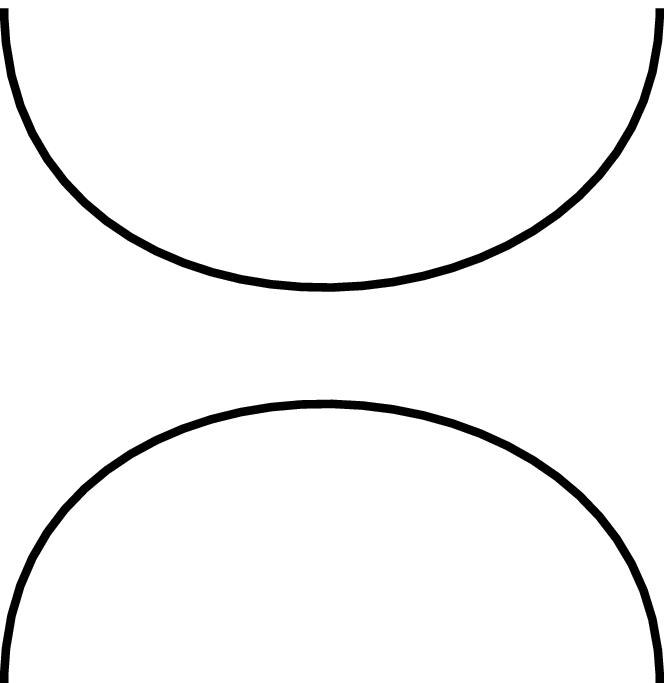}}
   \end{minipage}
   \right] \right). $
\vspace{3pt}
\item
$ \left[ \begin{minipage}[h]{0.035\linewidth}
        \vspace{0pt}
        \scalebox{0.25}{\includegraphics{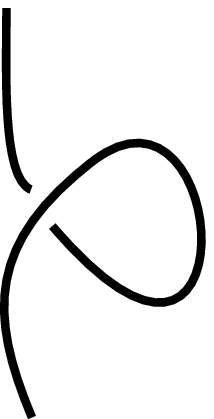}}
   \end{minipage}  \right]= a  \left[ \begin{minipage}[h]{0.035\linewidth}
        \vspace{0pt}
        \scalebox{0.25}{\includegraphics{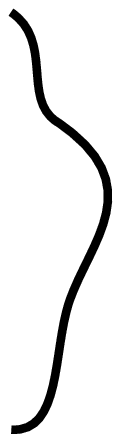}}
   \end{minipage} \right] $ and $\left[ \begin{minipage}[h]{0.035\linewidth}
        \vspace{0pt}
        \scalebox{0.25}{\includegraphics{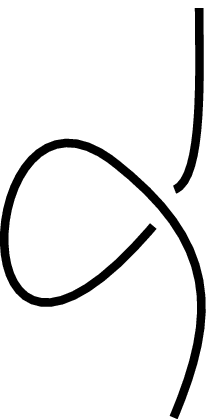}}
   \end{minipage}  \right]= a^{-1}  \left[ \begin{minipage}[h]{0.035\linewidth}
        \vspace{0pt}
        \scalebox{0.25}{\includegraphics{R_1P}}
   \end{minipage} \right] $. 
\item
\vspace{3pt}
$ \left[ \begin{minipage}[h]{0.05\linewidth}
        \vspace{0pt}
        \scalebox{0.02}{\includegraphics{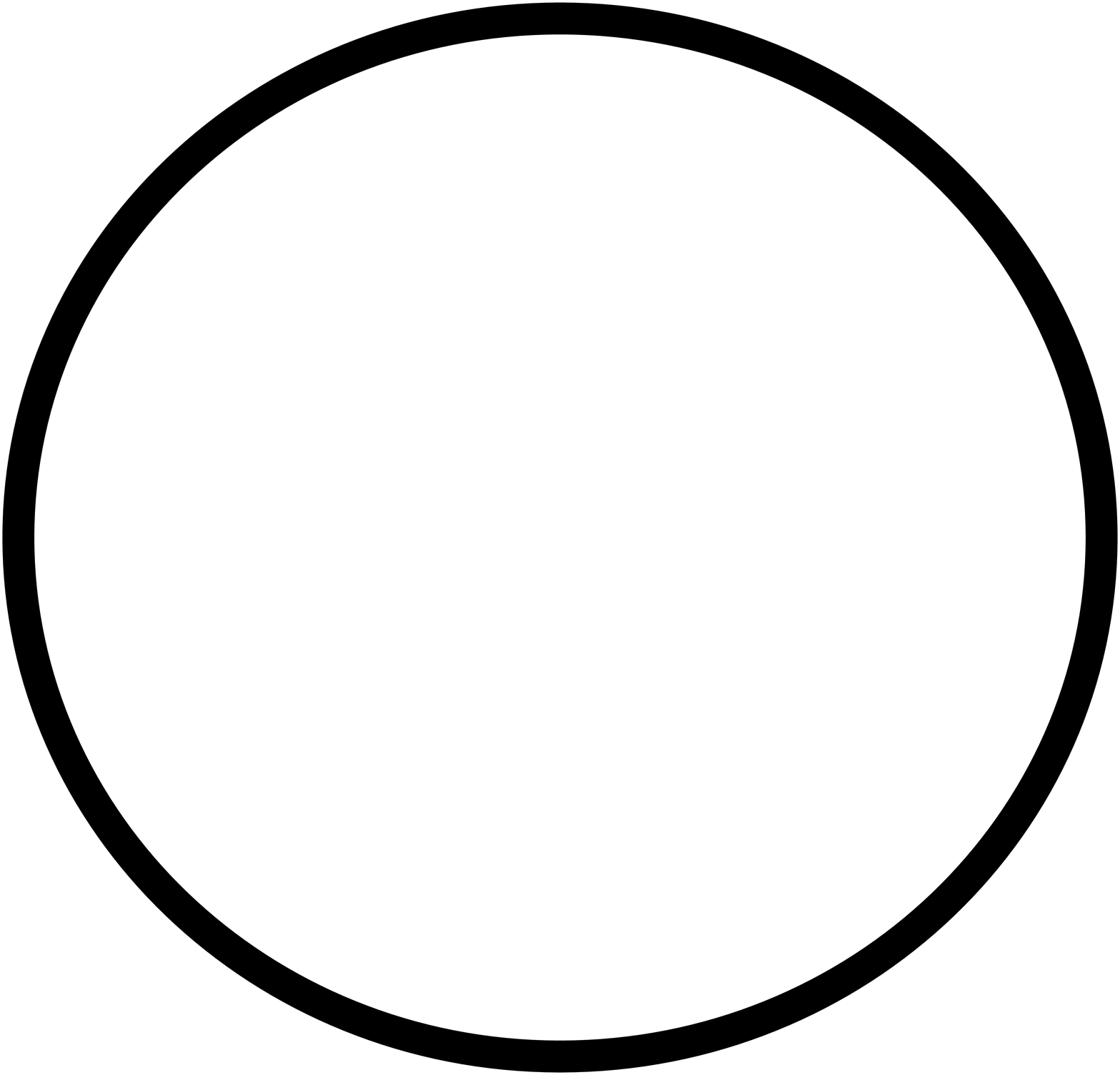}}
   \end{minipage}  \right]= 1.$

\end{enumerate}
The Kauffman polynomial is also called \textit{the Dubrovnik polynomial}. This invariant was extended to a $3$-variable function for embedded $4$-valent graphs in $\mathbb{R}^3$ to the \textit{Kauffman-Vogel polynomial} \cite{kaufvog} by adding the following axiom :

$$ \left[
   \begin{minipage}[h]{0.08\linewidth}
        \vspace{0pt}
        \scalebox{0.08}{\includegraphics{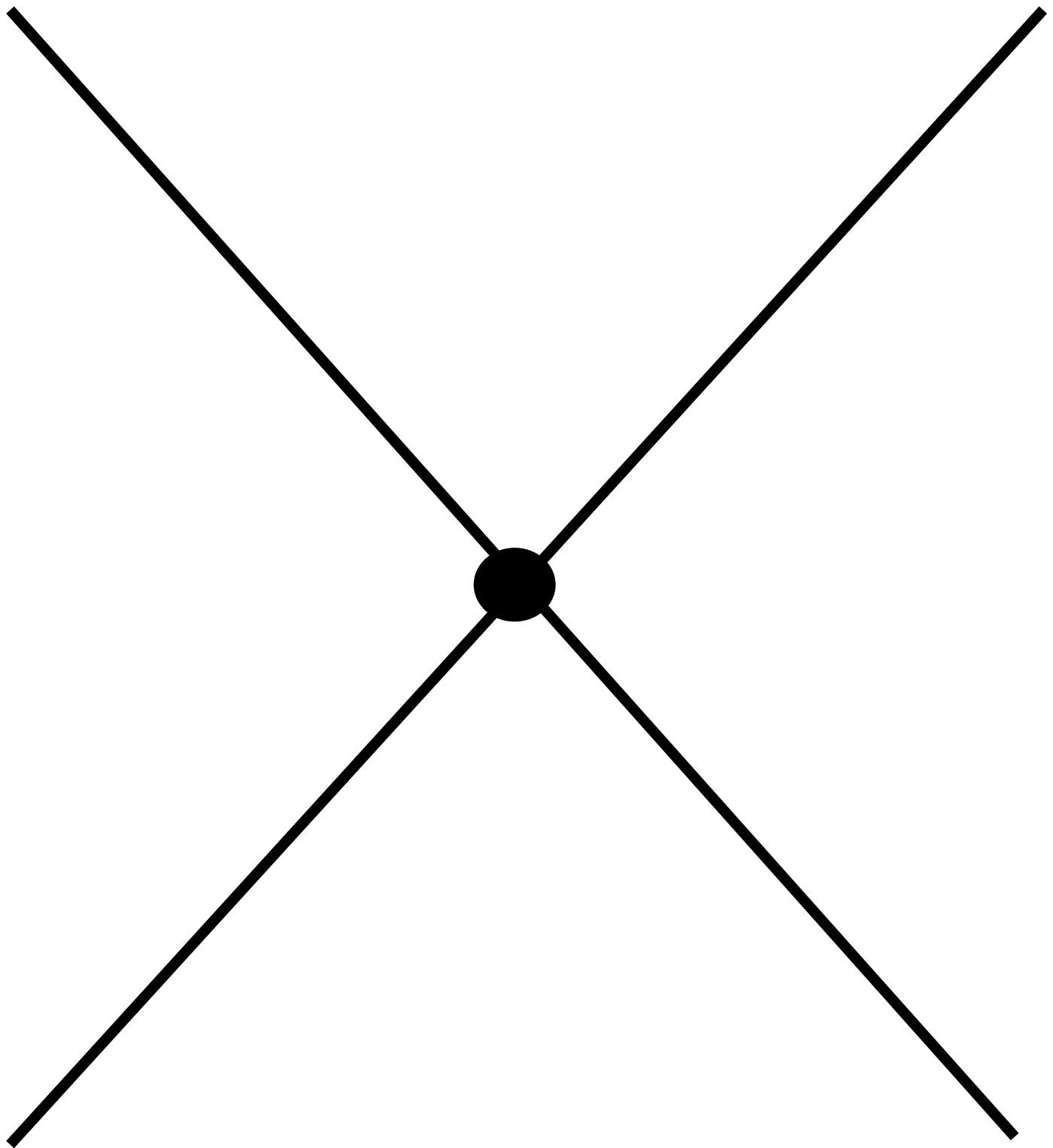}}
   \end{minipage}
  \right]=  \left[
   \begin{minipage}[h]{0.08\linewidth}
        \vspace{0pt}
        \scalebox{0.15}{\includegraphics{pos_crossing}}
   \end{minipage}
  \right]  - A  \left[
   \begin{minipage}[h]{0.08\linewidth}
        \vspace{0pt}
        \scalebox{0.15}{\includegraphics{other_smoothing}}
   \end{minipage}\right]
  - B\left[
   \begin{minipage}[h]{0.07\linewidth}
        \vspace{0pt}
        \scalebox{0.15}{\includegraphics{one_smoothing}}
   \end{minipage}
   \right]$$
where $A$ and $B$ are commuting variables and $A -B = z$. A specialization of the Kauffman-Vogel polynomial invariant can be obtained using the skein theory associated with the Kauffman bracket \cite{KauffLin}. This version is a one variable specialization of the Kauffman-Vogel polynomial and it is defined by using \textit{Jones-Wenzl} projector \cite{J2,Wenzl}. The purpose of this paper is to give a generalization of this version of the  Kauffman-Vogel polynomial. Our generalization is given in the form of a sequence of invariants whose first term is the one variable specialization of the Kauffman-Vogel polynomial. The sequence of invariants gives us naturally a sequence of singular braid representations.

The organization of the paper is as follows. In section \ref{sec2} we give the necessary background needed in this paper. In section \ref{sec3} the one variable specialization of the Kauffman-Vogel polynomial is defined. In section \ref{sec4} we introduce our generalization of this polynomial. In section \ref{sec5} we show how to use this invariant to give a sequence of singular braid representations.   
 



\section{The Kauffman Bracket Skein Module}\label{sec2}
%
%
%
%
%
%
%
%
%
%
%
%
%
In this section we review the definition of the Kauffman bracket skein module of a $3$-manifold $M$ over a commutative ring $\mathcal{R}$. 
A framed link in $M$ is an oriented embedding of a disjoint union of oriented annuli in $M$.  A framed point in the boundary $\partial M$ of $M$ is a closed interval in $\partial M$. Let $x$ and $y$ be framed points in $\partial M$. A \textit{band} in $M$ is an oriented embedding of $I \times I$ into $M$ that meets $\partial M$ orthogonally at $x$ and $y$.

\begin{definition} \cite{Przytycki}
Let $M$ be a $3$-oriented manifold and $\mathcal{R}$ be a commutative ring with a unit and an invertible element $A$. Let $\mathcal{L}_M$ denotes the set of all isotopy classes of unoriented framed links in $M$. Here we consider the empty link to be an element of $\mathcal{L}_M$. Let $\mathcal{R}\mathcal{L}_M$ be the free $\mathcal{R}$-module generated by $\mathcal{L}_M$. The Kauffman bracket skein module of the $3$-manifold $M$ and the ring $\mathcal{R}$ is the quotient given by:
\begin{equation}
\mathcal{S}(M,\mathcal{R},A)=\mathcal{R}\mathcal{L}_M/R(M),
\end{equation}
where $R(M)$ is the submodule of $\mathcal{R}\mathcal{L}_M$ generated by all expressions of the form
\begin{eqnarray*}(1)\hspace{3 mm}
  \begin{minipage}[h]{0.06\linewidth}
        \vspace{0pt}
        \scalebox{0.04}{\includegraphics{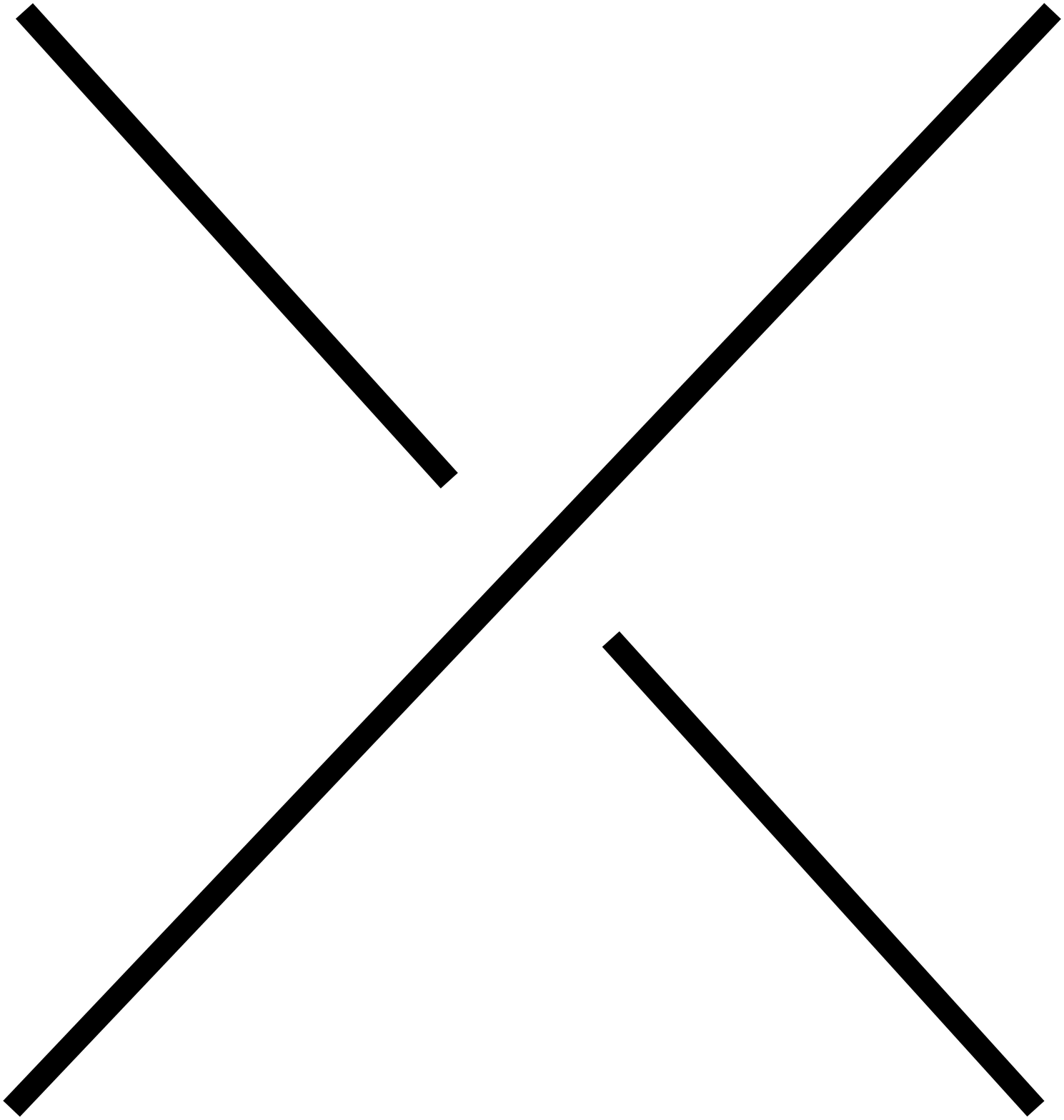}}
   \end{minipage}
   -
     A 
  \begin{minipage}[h]{0.06\linewidth}
        \vspace{0pt}
        \scalebox{0.04}{\includegraphics{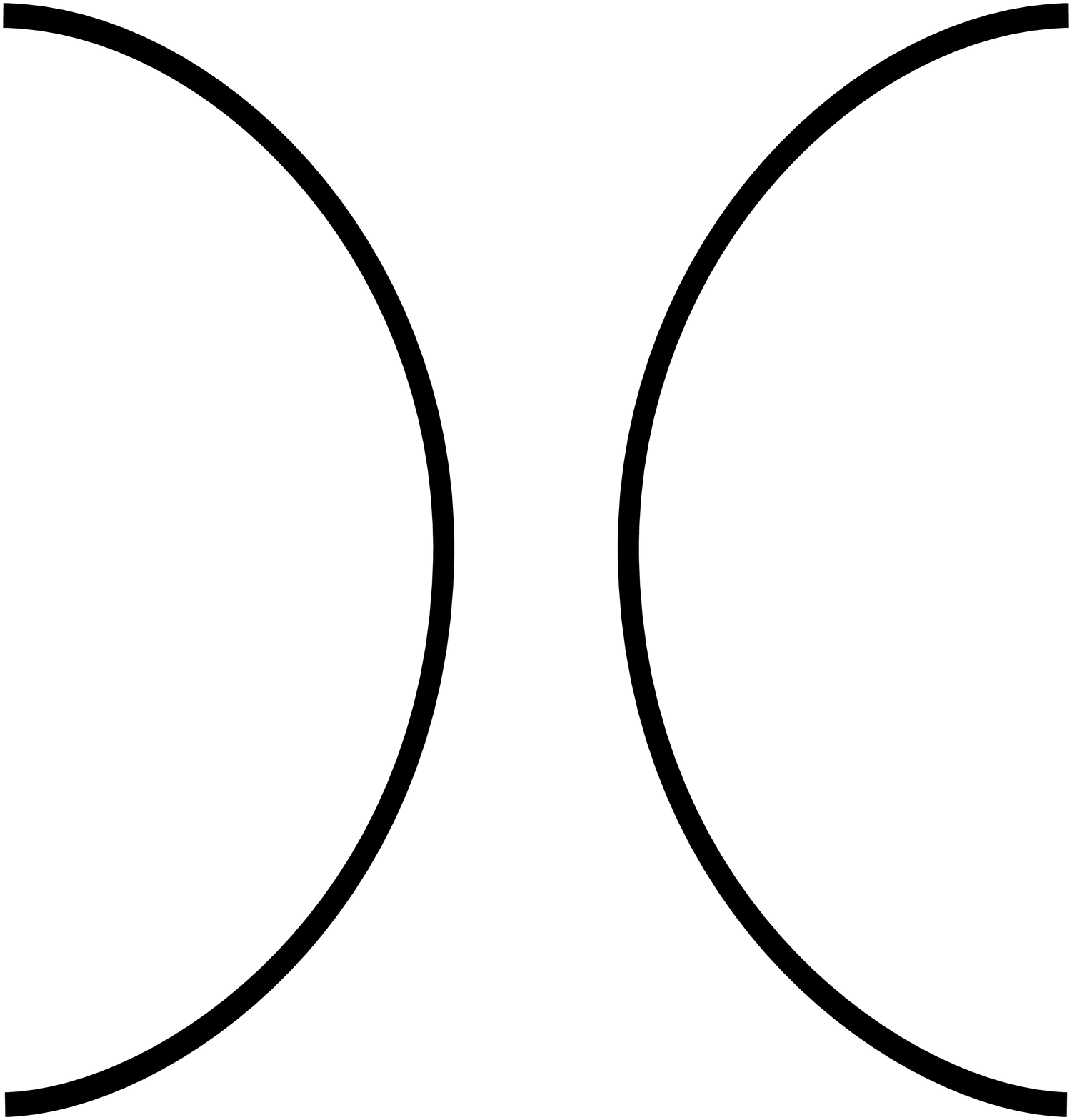}}
   \end{minipage}
   -
  A^{-1} 
  \begin{minipage}[h]{0.06\linewidth}
        \vspace{0pt}
        \scalebox{0.04}{\includegraphics{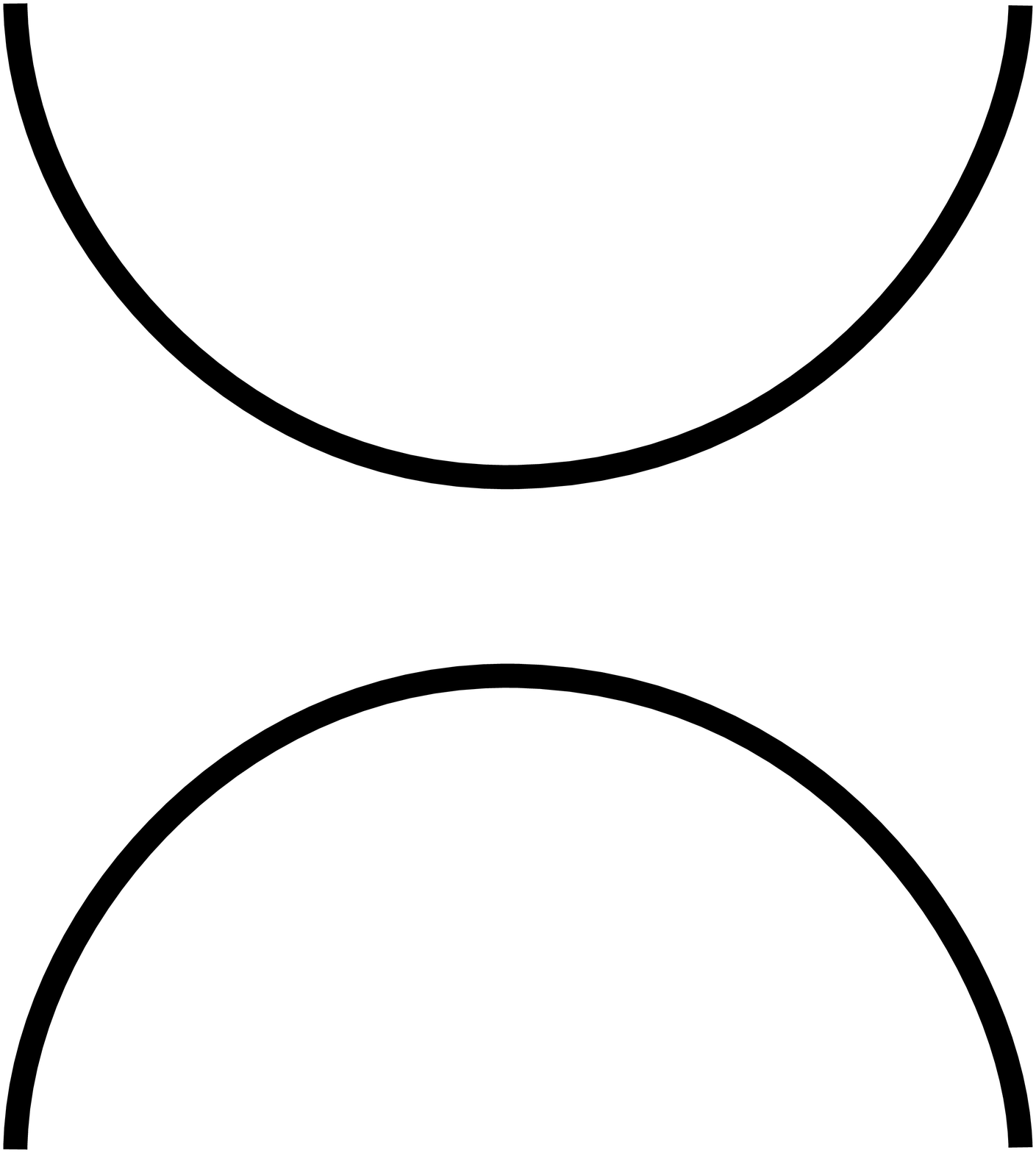}}
   \end{minipage}
, \hspace{20 mm}
  (2)\hspace{3 mm} L\sqcup
   \begin{minipage}[h]{0.05\linewidth}
        \vspace{0pt}
        \scalebox{0.02}{\includegraphics{simple-circle}}
   \end{minipage}
  +
  (A^{2}+A^{-2})L, 
  \end{eqnarray*}
where $L\sqcup$ \begin{minipage}[h]{0.05\linewidth}
        \vspace{0pt}
        \scalebox{0.02}{\includegraphics{simple-circle}}
   \end{minipage}  consists of a framed link $L$ in $M$ and the trivial framed knot 
   \begin{minipage}[h]{0.05\linewidth}
        \vspace{0pt}
        \scalebox{0.02}{\includegraphics{simple-circle}}
   \end{minipage}.

\end{definition}
We will sometimes drop the ring $\mathcal{R}$ from the notation and refer to the Kauffman bracket skein module of the manifold $M$ and the ring $\mathcal{R}$ simply by $\mathcal{S}(M)$ when the context is clear. The definition of the Kauffman bracket skein module can be extended to $3$-manifolds with boundaries. Let $x_1, \cdots , x_{2n}$ be a set, possibly empty, of designated framed points on $\partial M$. Let $\mathcal{L}_M$ be the set of all surfaces in $M$ decomposed into a union of finite number of framed links and bands joining the points $\{ x_i \}_{i=1}^{2n}$. The relative Kauffman bracket skein module is defined to be \begin{equation}
\mathcal{S}(M,\mathcal{R},A,\{ x_i \}_{i=1}^{2n})=\mathcal{R}\mathcal{L}_M/R(M).
\end{equation}
It can be shown that the definition of the relative Kauffman bracket skein module is independent of the choice of the position of the points $\{ x_i \}_{i=1}^{2n}$.  Furthermore, the construction of the relative Kauffman bracket skein module is functorial in the sense that an embedding of oriented $3$-manifolds with $2n$ (framed) points on the boundaries 
\begin{equation}
j: (M,\{ x_i \}_{i=1}^{2n}) \hookrightarrow (N,\{ y_i \}_{i=1}^{2n})
\end{equation}
 induces a homomorphism of $\mathcal{R}$-modules  \begin{equation}
\mathcal{S}(M,\mathcal{R},A,\{ x_i \}_{i=1}^{2n}) \rightarrow \mathcal{S}(N,\mathcal{R},A,\{ y_i \}_{i=1}^{2n}).  
\end{equation}
When the $3$-manifold $M$ is homeomorphic to $F\times I$ where $S$ an oriented surface with a finite set of points (possibly empty) in its boundary $\partial  F$ and $I$ is an interval, then one can project framed links in $M$ to link diagrams in $F$.\\

 The first example of the Kauffman bracket skein module that we will consider in this paper is the Kauffman bracket skein module of the $3$-sphere $S^3$. It can be easily shown that this module is free on the empty link, meaning $\mathcal{S}(S^3)= \mathcal{R}$. The second one is the relative Kauffman bracket skein module of $D^3=I\times I \times I$ with $2n$ marked points on its boundary $\partial D^3$.  The first $n$ points are placed on the top edge $D^3$ and the other $n$ points on the bottom edge. Recall that the relative skein module does not depend on the exact position of the points $\{x_i\}_{i=1}^{2n}$. However, we need to specify the position here in order to define an algebra structure on  $\mathcal{S}(D^3,\mathcal{R},A,\{ x_i \}_{i=1}^{2n})$. Let $S_1$ and $S_2$ be two elements in $\mathcal{L}_M$ such that $\partial S_j$, where $j=1,2$, consists of the points $\{x_i\}_{i=1}^{2n}$ that we specified above. Define $S_1 \times S_2$ to be the surface in $D^3$ obtained by attaching $S_1$ on the top of $S_2$ and then compress the result to $D^3$. This multiplication extends to a well-defined multiplication on $\mathcal{S}(D^3,\mathcal{R},A,\{ x_i \}_{i=1}^{2n})$. With this multiplication the module $\mathcal{S}(D^3,\mathcal{R},A,\{ x_i \}_{i=1}^{2n})$ becomes an associative algebra over $\mathcal{R}$ known as the \textit{$n^{th}$ Temperley-Lieb algebra} $TL_n$. For more details see \cite{Przytycki}. Historically, The Temperley–Lieb algebra first arose in the form of some graph-theoretic problems studied in the context of Potts models in statistical mechanics \cite{TemperleyLieb}.  The Temperley–Lieb algebra was independently rediscovered by Jones \cite{JonesPlanarAlgebras} in his work on von Neumann algebras.\\

For the rest of the paper we will fix $\mathcal{R}$ to be $\mathbb{Q}(A)$ the field generated by the indeterminate $A$ over the rational numbers.

\subsection{The Jones-Wenzl Idempotents}
 The Jones-Wenzl idempotent $f^{(n)}\in TL_n$ has proven to be central to understand the Temperley-Lieb algebra and its applications. This idempotent plays a central role in the Witten-Reshetikhin-Turaev Invariants for $SU(2)$ \cite{KauffLin,Lic92,RT}, the colored Jones polynomial and its applications \cite{BHMV,RT,Hajij2,Hajij1,TuraevViro92}, and quantum spin networks \cite{MV}. The Jones-Wenzl idempotent was defined in \cite{J2} and it enjoys a recursive formula due to Wenzl \cite{Wenzl} : 
\begin{align}
\label{recursive}
  \begin{minipage}[h]{0.05\linewidth}
        \vspace{0pt}
        \scalebox{0.12}{\includegraphics{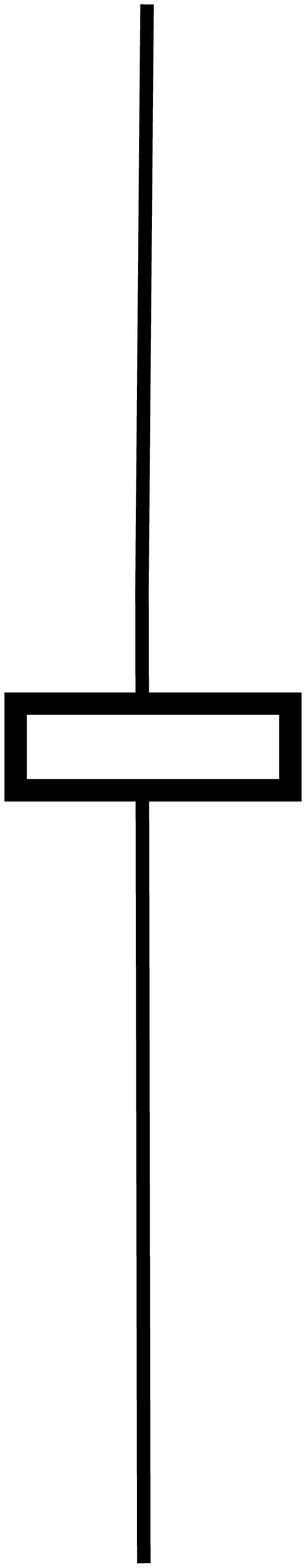}}
         \put(-20,+70){\footnotesize{$n$}}
   \end{minipage}
   =
  \begin{minipage}[h]{0.08\linewidth}
        \hspace{8pt}
        \scalebox{0.12}{\includegraphics{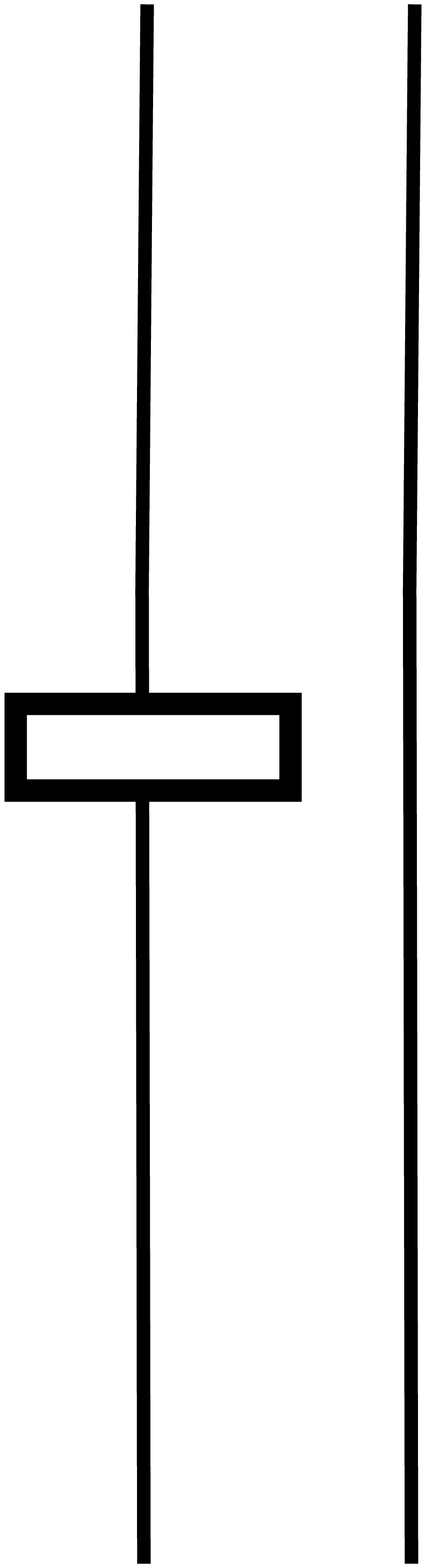}}
        \put(-42,+70){\footnotesize{$n-1$}}
        \put(-8,+70){\footnotesize{$1$}}
   \end{minipage}
   \hspace{9pt}
   -
 \Big( \frac{\Delta_{n-2}}{\Delta_{n-1}}\Big)
  \hspace{9pt}
  \begin{minipage}[h]{0.10\linewidth}
        \vspace{0pt}
        \scalebox{0.12}{\includegraphics{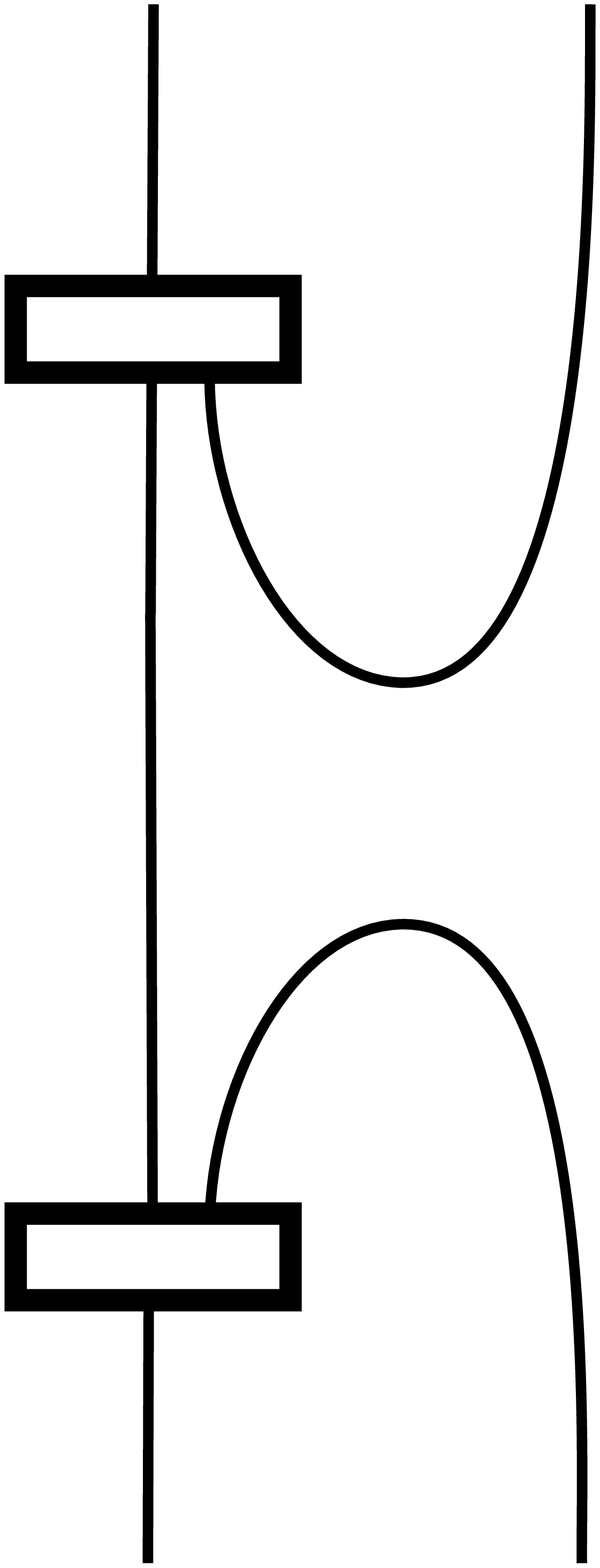}}
         \put(2,+85){\footnotesize{$1$}}
         \put(-52,+87){\footnotesize{$n-1$}}
         \put(-25,+47){\footnotesize{$n-2$}}
         \put(2,+10){\footnotesize{$1$}}
         \put(-52,+5){\footnotesize{$n-1$}}
   \end{minipage}
  , \hspace{20 mm}
    \begin{minipage}[h]{0.05\linewidth}
        \vspace{0pt}
        \scalebox{0.12}{\includegraphics{nth-jones-wenzl-projector}}
        \put(-20,+70){\footnotesize{$1$}}
   \end{minipage}
  =
  \begin{minipage}[h]{0.05\linewidth}
        \vspace{0pt}
        \scalebox{0.12}{\includegraphics{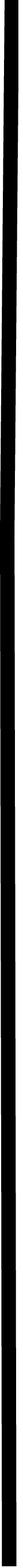}}
   \end{minipage}   
  \end{align}
  where 
\begin{equation*}
 \Delta_{n}=(-1)^{n}\frac{A^{2(n+1)}-A^{-2(n+1)}}{A^{2}-A^{-2}}.
\end{equation*} 

The graphical notation of $f^{(n)}$ is due to Lickorish \cite{Lic92}. The idempotent satisfies the following properties: 

\begin{eqnarray}
\label{properties}
\hspace{0 mm}
    \begin{minipage}[h]{0.21\linewidth}
        \vspace{0pt}
        \scalebox{0.115}{\includegraphics{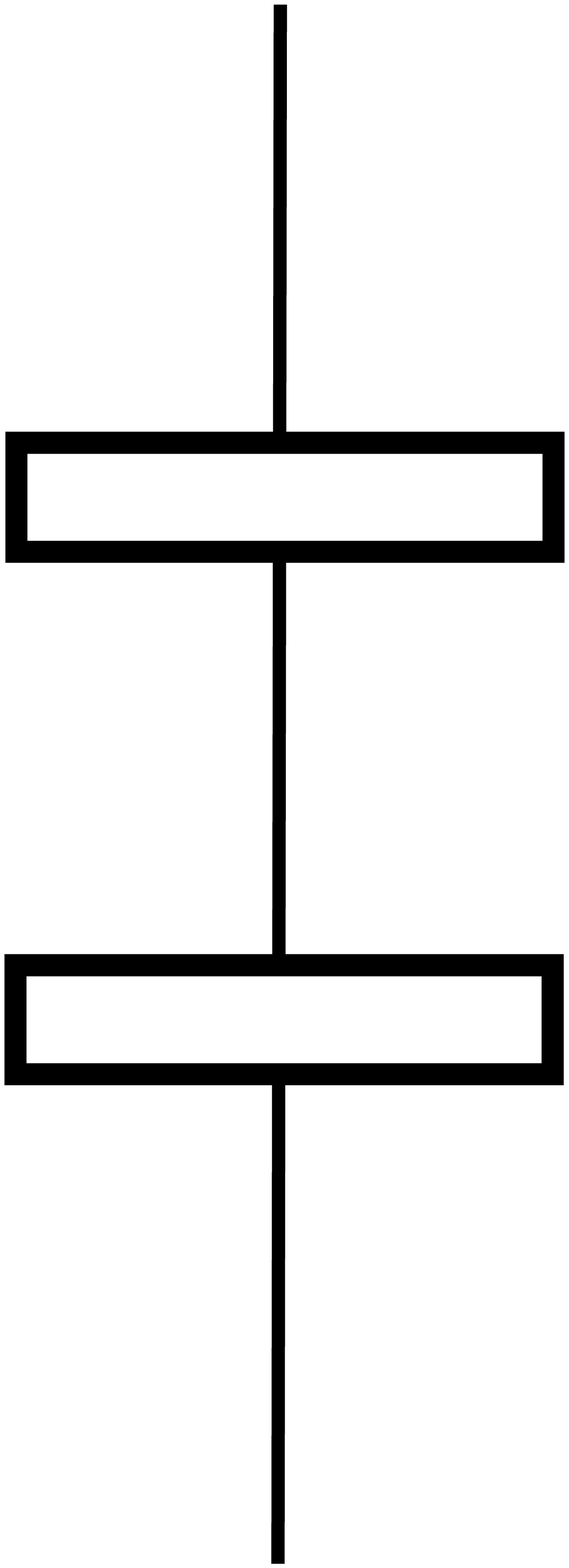}}
        \put(0,+80){\footnotesize{$n$}}
       
   \end{minipage}
  = \hspace{5pt}
     \begin{minipage}[h]{0.1\linewidth}
        \vspace{0pt}
         \hspace{50pt}
        \scalebox{0.115}{\includegraphics{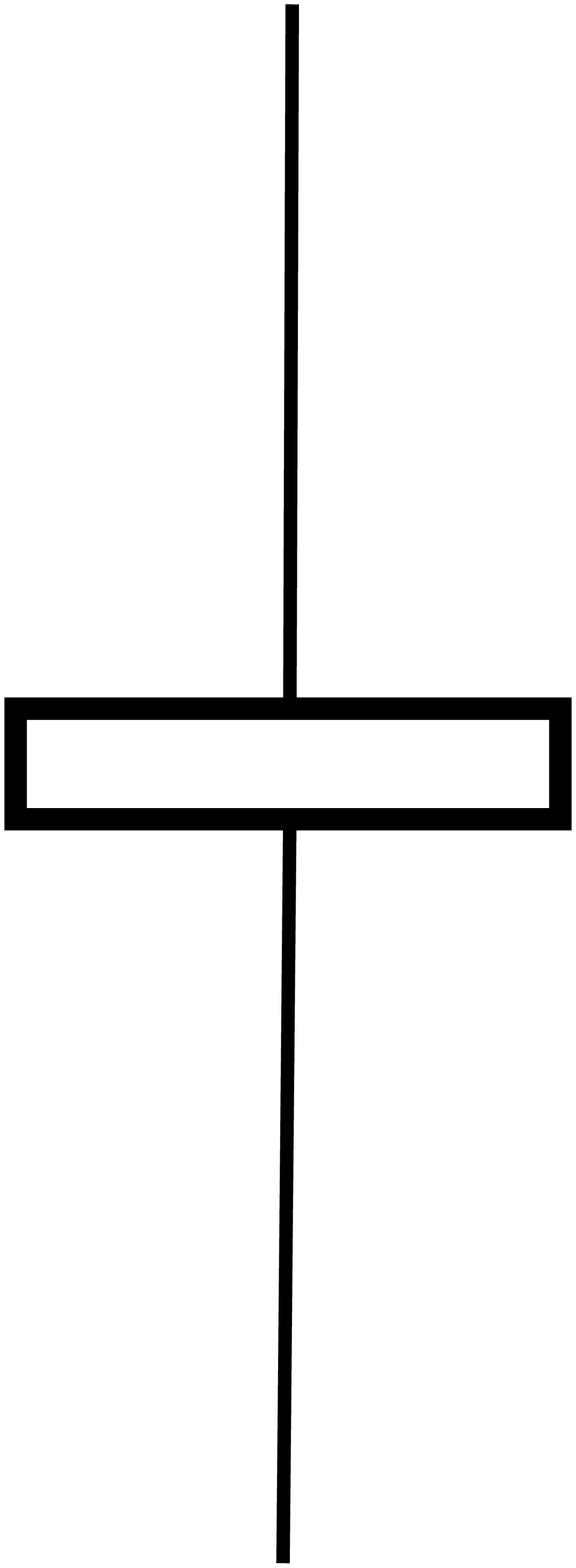}}
        \put(-60,80){\footnotesize{$n$}}
   \end{minipage}
    , \hspace{15 mm}
    \begin{minipage}[h]{0.09\linewidth}
        \vspace{0pt}
        \scalebox{0.115}{\includegraphics{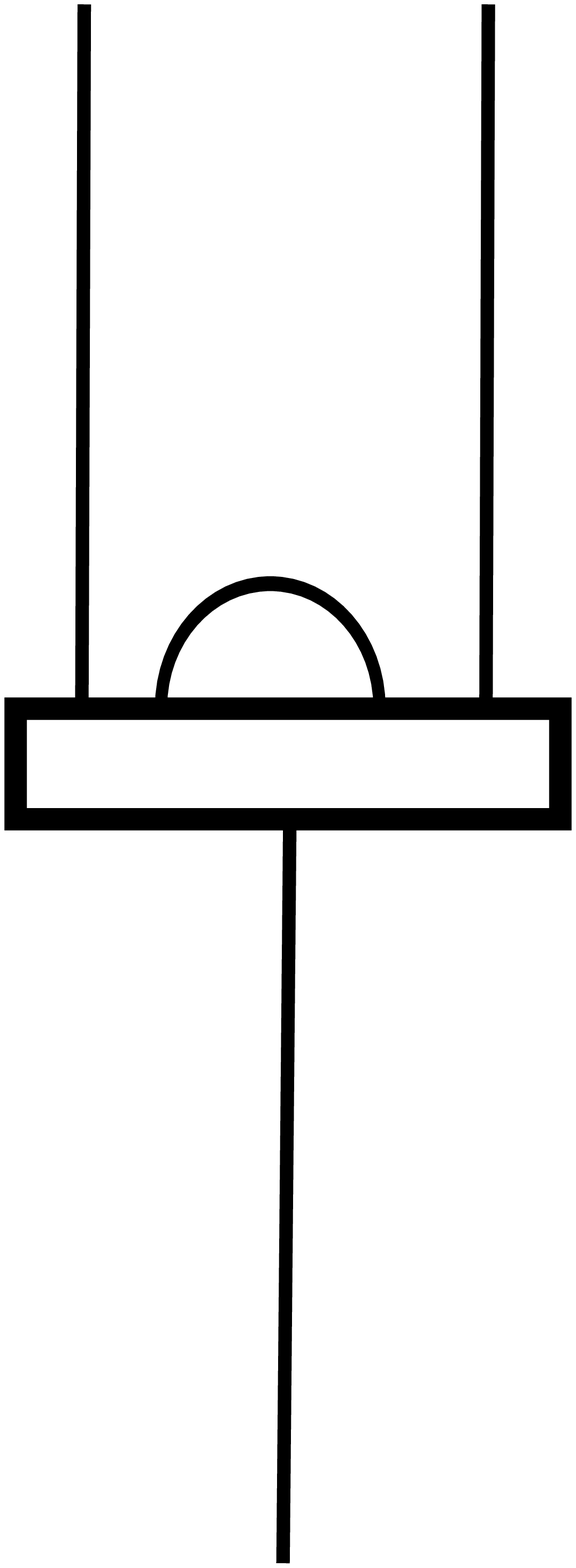}}
         \put(-70,+82){\footnotesize{$n-i-2$}}
         \put(-20,+64){\footnotesize{$1$}}
        \put(-2,+82){\footnotesize{$i$}}
        \put(-28,20){\footnotesize{$n$}}
   \end{minipage}
   =0,
   \label{AX}
  \end{eqnarray}
\begin{eqnarray}
\label{properties}
\hspace{0 mm}
\Delta_{n}=
  \begin{minipage}[h]{0.1\linewidth}
        \vspace{0pt}
        \scalebox{0.12}{\includegraphics{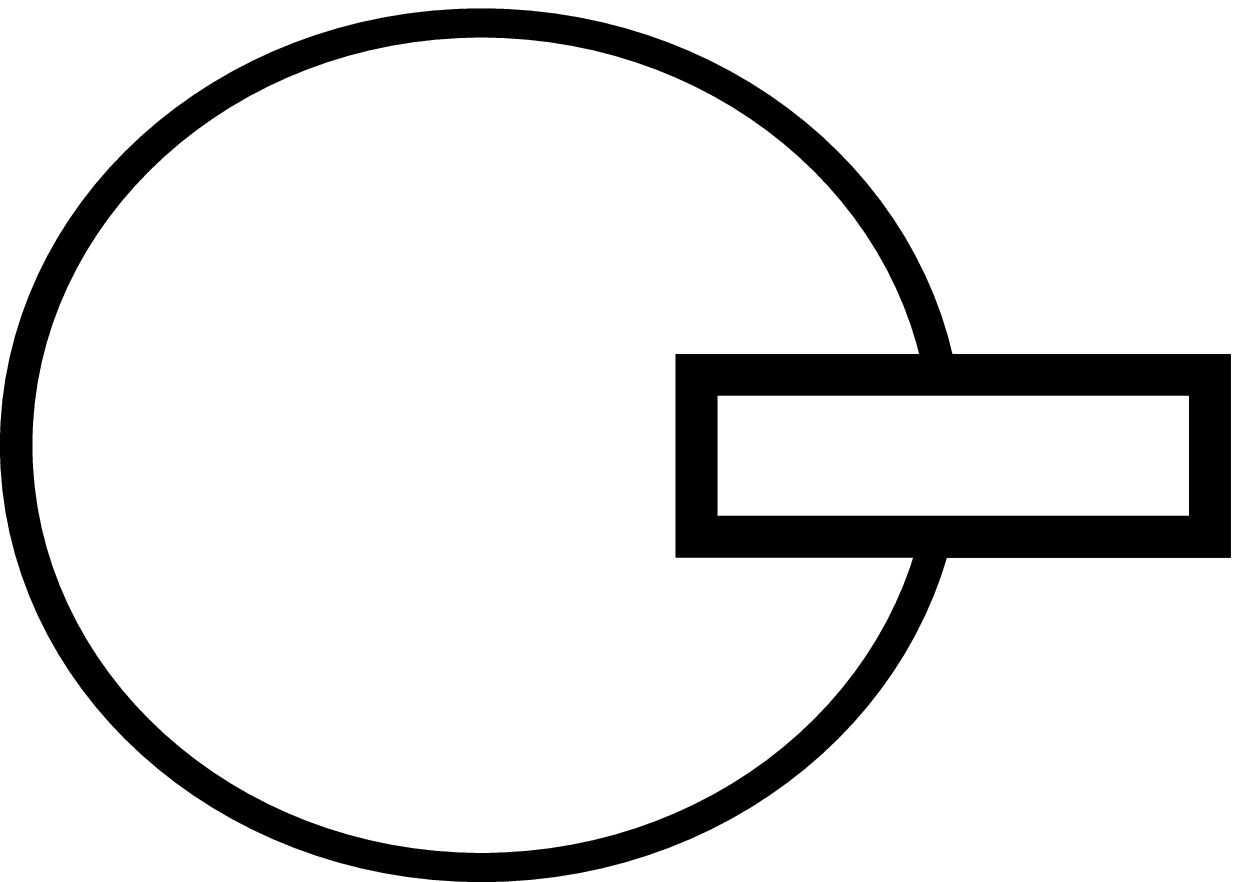}}
        \put(-29,+34){\footnotesize{$n$}}
   \end{minipage}
   , \hspace{14 mm}
     \begin{minipage}[h]{0.08\linewidth}
        \vspace{0pt}
        \scalebox{0.115}{\includegraphics{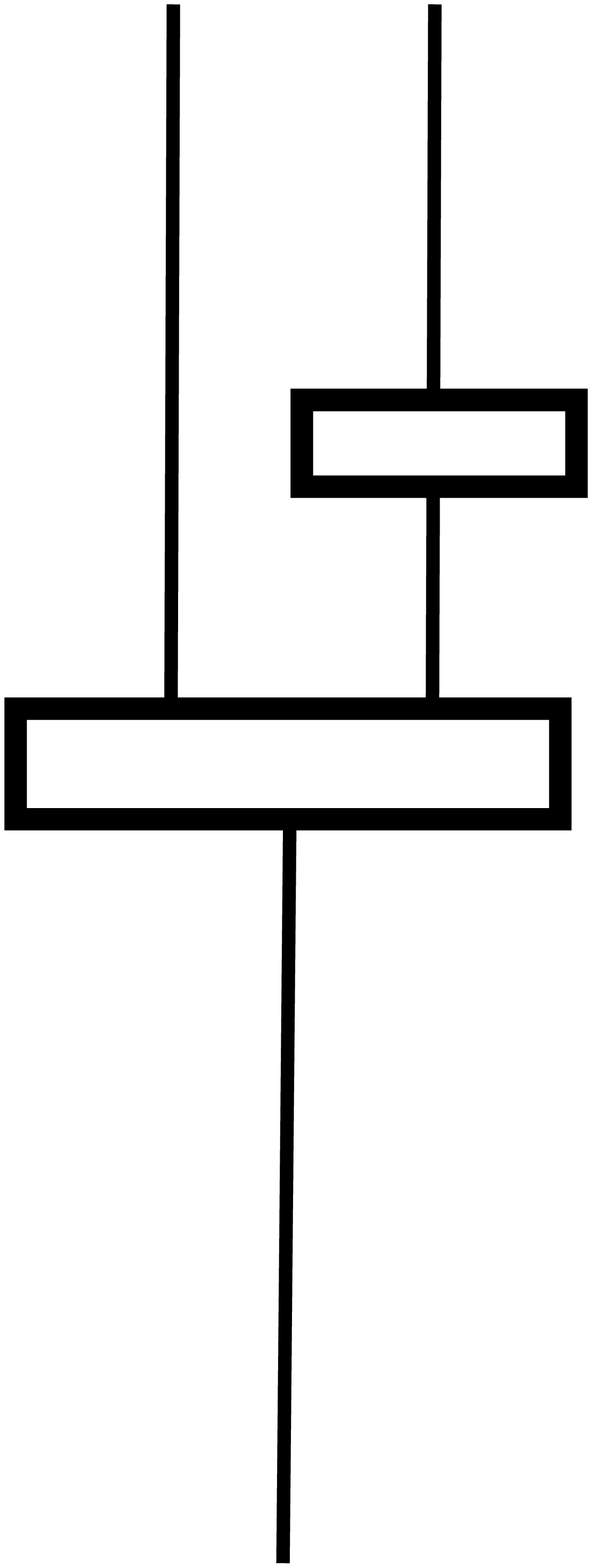}}
        \put(-34,+82){\footnotesize{$n$}}
        \put(-19,+82){\footnotesize{$m$}}
        \put(-46,20){\footnotesize{$m+n$}}
   \end{minipage}
  =
     \begin{minipage}[h]{0.09\linewidth}
        \vspace{0pt}
        \scalebox{0.115}{\includegraphics{idempotent2}}
        \put(-46,20){\footnotesize{$m+n$}}
   \end{minipage}
  \end{eqnarray}
and     
  \begin{eqnarray}
    \label{properties_2}
     \hspace{8 mm}
   \begin{minipage}[h]{0.08\linewidth}
        \vspace{0 pt}
        \scalebox{.15}{\includegraphics{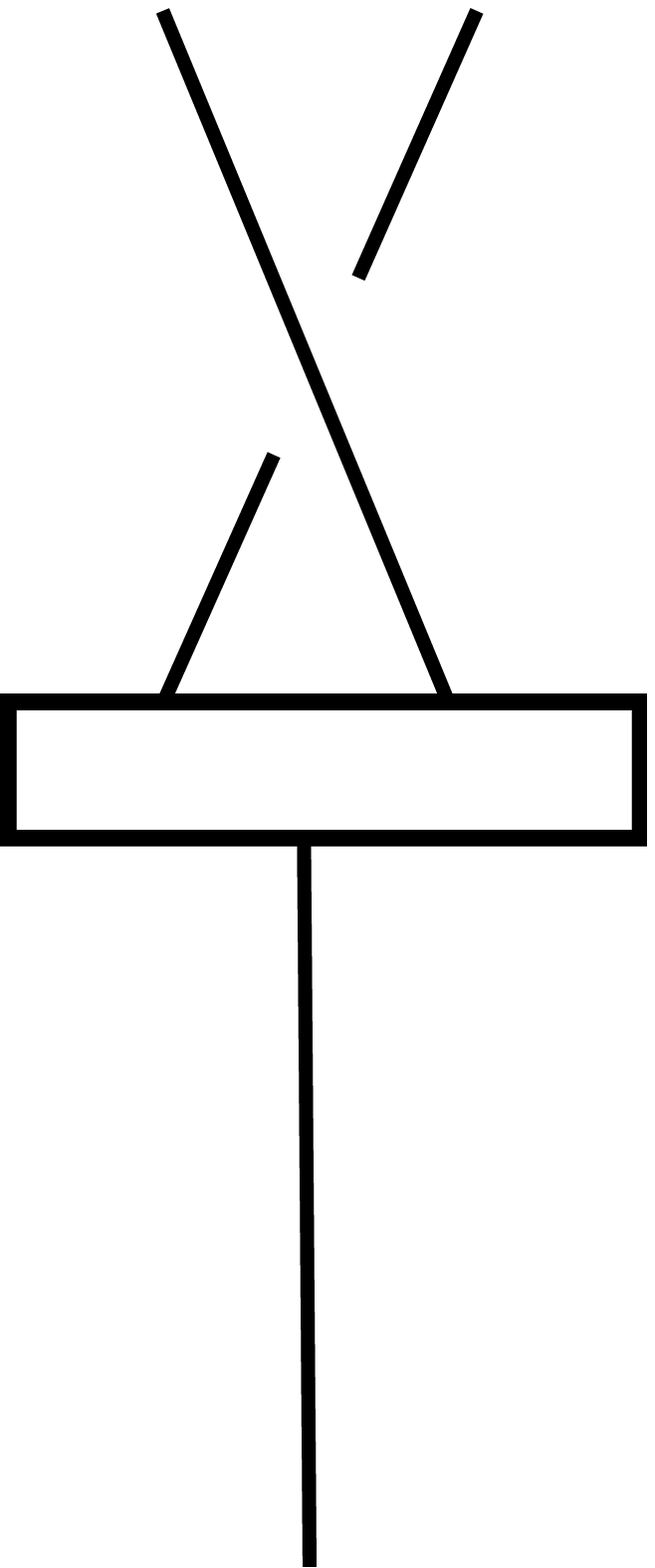}}
        \scriptsize{
        \put(-25,+72){$i$}
        \put(-5,+72){$j$}
        \put(-10,+10){$i+j$}
        }

   \end{minipage}   
   =
     A^{-ij}
  \begin{minipage}[h]{0.1\linewidth}
        \vspace{0pt}
        \scalebox{.15}{\includegraphics{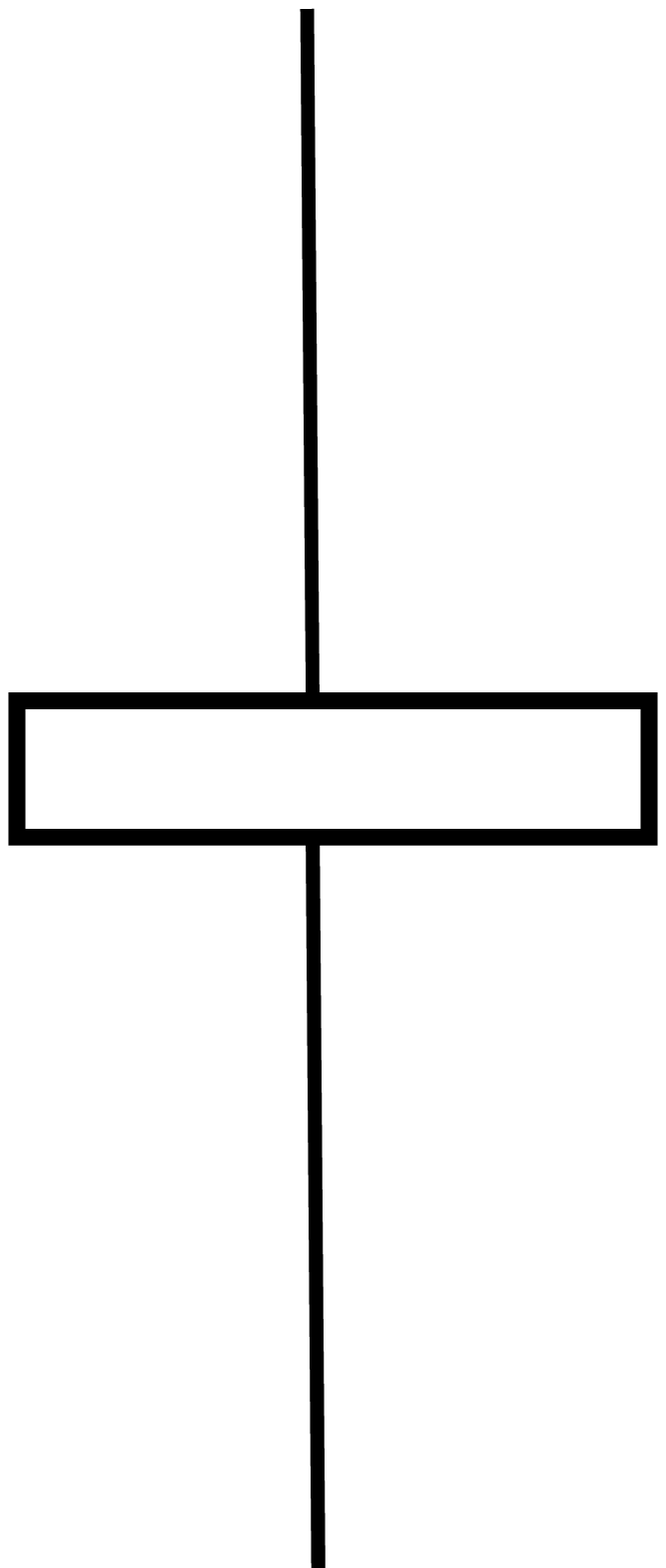}}
        \scriptsize{
          \put(-10,+10){$i+j$}
                  }
          \end{minipage}
           , \hspace{14 mm}
     \begin{minipage}[h]{0.08\linewidth}
        \vspace{-5pt}
        \scalebox{0.19}{\includegraphics{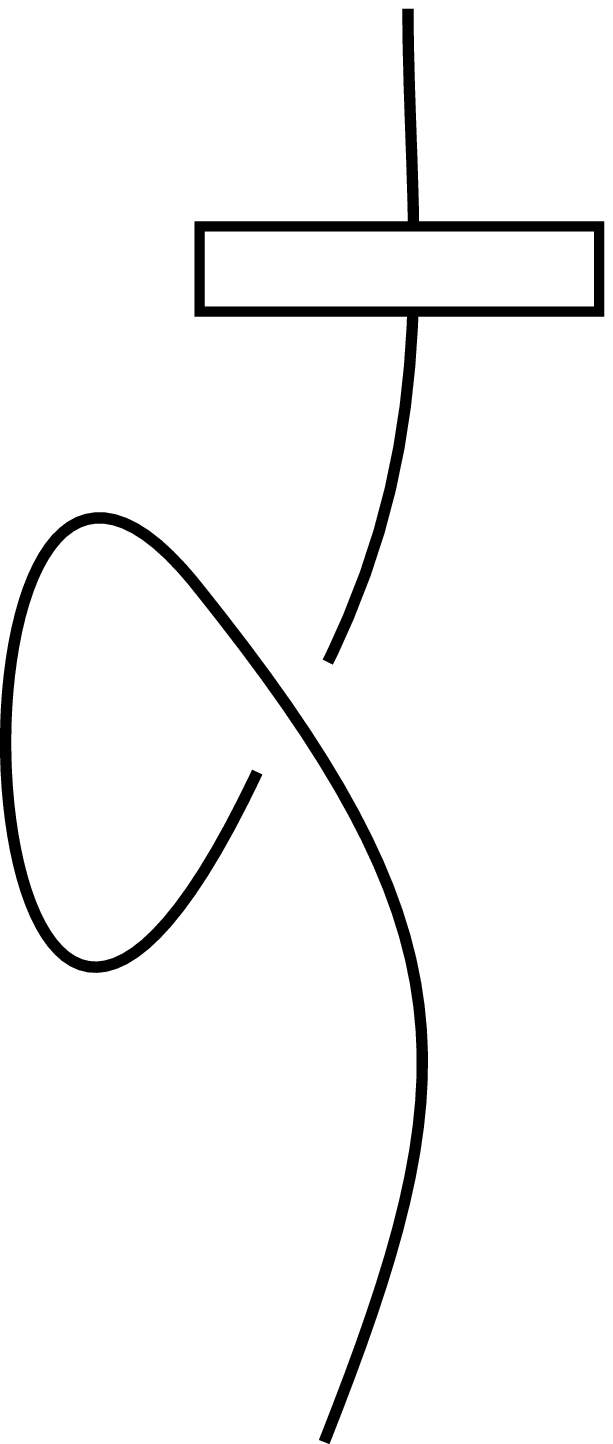}}
        \put(-22,+75){\footnotesize{$n$}}
   \end{minipage}
  =(-1)^nA^{-n^2-2n}
     \begin{minipage}[h]{0.09\linewidth}
        \hspace{5pt}
        \scalebox{0.19}{\includegraphics{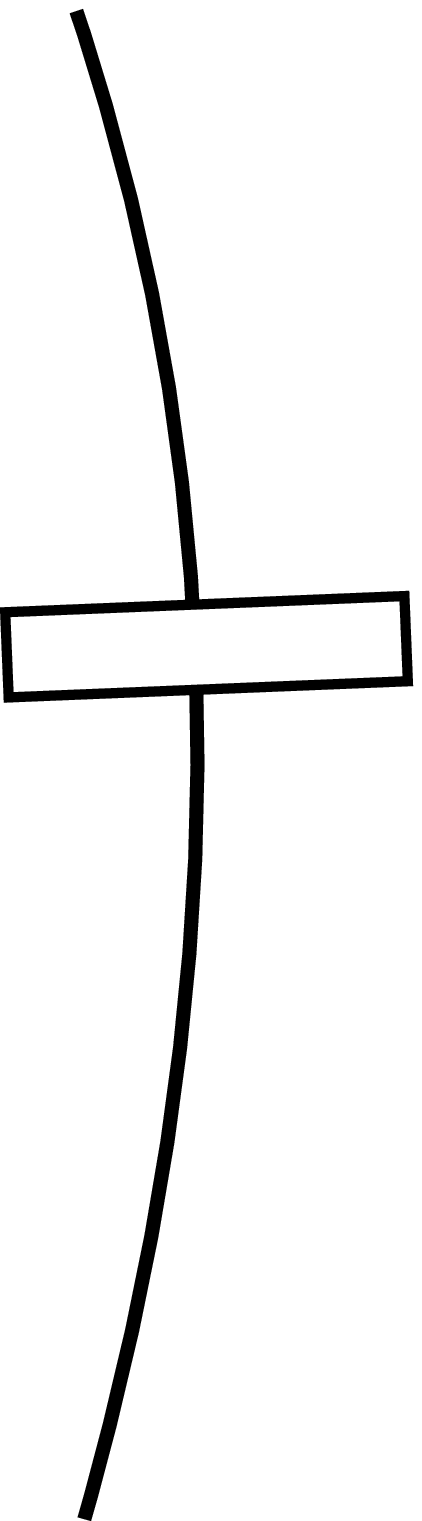}}
        \put(-6,55){\footnotesize{$n$}}
   \end{minipage}
   \end{eqnarray}
The definiton of the Jones-Wenzl projector is the main tool for our construction of the new singular knot invariants that we will introduce in section \ref{sec4}.
\subsection{The Colored Temperley-Lieb Algebra}
Let $m,n$ be two positive integers. Consider the skein module of $I\times I \times I$ with $2mn$ specified points on the boundary. More specifically, we put $mn$ marked points on the top and $mn$ points on the bottom.  Partition the set
of the $2 mn$ points on the boundary of the disk into $2m$ sets each one of them has $n$ points. At each cluster of $n$ points we place a Jones-Wenzl idempotent $f^{(n)}$. The skein module of $I \times I \times I$ with $2mn$ specified points on the boundary can be made into a unital associative algebra in a similar way as in the case of the Templely-Lieb algebra. In other words, if $A$ and $B$ are two diagrams in this algebra then $A \times B$ is defined as illustrated in Figure \ref{colored multiplication}.
\begin{figure}[H]
  \centering
   {\includegraphics[scale=0.5]{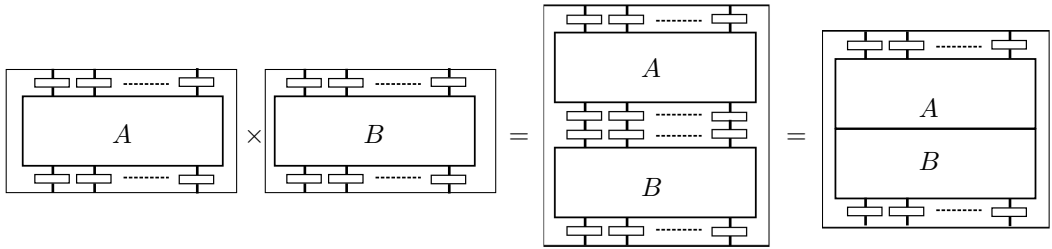}
   \small{
    \put(-355,40){$A$}
    \put(-305,40){$\times$}
    \put(-260,40){$B$}
    \put(-205,40){$=$}
    \put(-155,65){$A$}
    \put(-155,20){$B$}
    \put(-100,40){$=$}
      \put(-50,50){$A$}
    \put(-50,30){$B$}
         }
     \caption{Multiplication in the colored Templely-Lieb algebra.  }
  \label{colored multiplication}}
\end{figure}

 We will denote this algebra by $TL_n^m$.  The algebra $TL_{n}^m$ can be seen as the subalgebra of $TL_{mn}$ generated by all elements of the form $(f^{(m)})^{\otimes n} \otimes D\otimes (f^{(m)})^{\otimes n}$ where $D$ is a diagram that generates $TL_{mn}$. Using the properties of the Jones Wenzl idempotent, the skein module $TL_{n}^n$ is one dimensional generated by $f^{(n)}$. On the other hand, $TL_n^1$ is just the standard Templely-Lieb algebra $TL_n$.
\subsubsection{Braid group representations into the colored Temperley-Lieb algebra }
For every integer $m,n \geq 1$, the following map gives a representation of $B_n$ inside $TL_{n}^m$ :

\begin{eqnarray}
\label{rep}
 \hspace{-230 pt}\sigma_i=\begin{minipage}[h]{0.0\linewidth}
        \scalebox{0.5}{\includegraphics{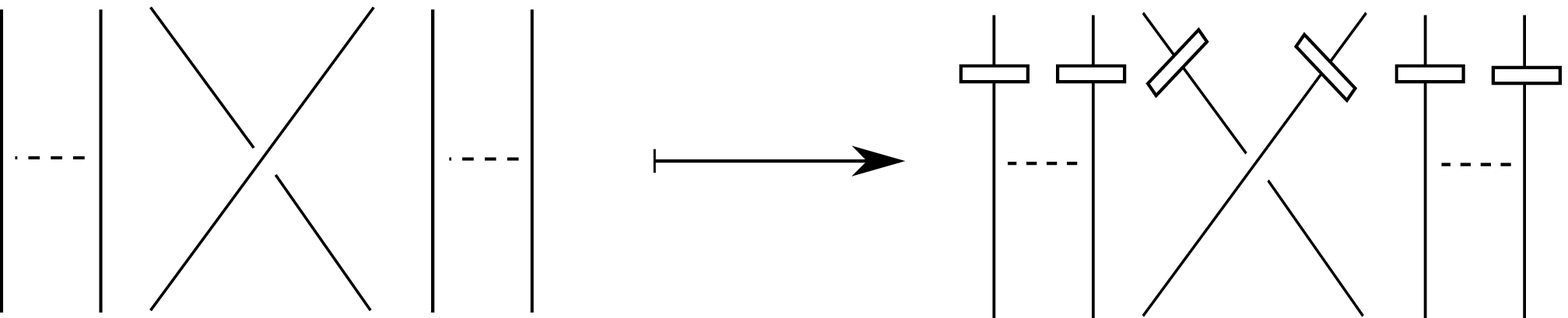}}
         \put(-160,40){$\rho_{m,n}$}
         \tiny{
         \put(-114,50){$m$}
         \put(-3,50){$m$}
         }
   \end{minipage}  
\end{eqnarray} 
The fact that Reidemeister moves $II$ and $III$ hold in the Kauffman bracket skein module implies that the map $\rho_{m,n}$ is indeed a representation. More precisely, the moves shown in Figure \ref{CRmoves} are basically a finite sequence of the usual Reidemeister moves $II$ and $III$ applied on each single strand and summand of the idempotents.
\begin{figure}[H]
  \centering
   {\includegraphics[scale=0.15]{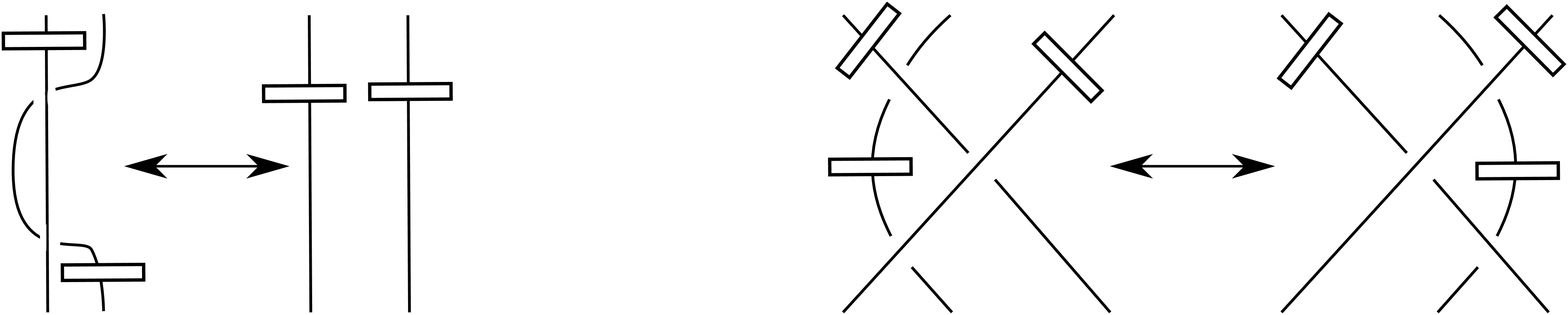}
     \caption{Reidemeister moves hold on strands colored with the Jones-Wenzl projector.}
  \label{CRmoves}}
\end{figure}
In section \ref{sec5} we will extend the representation $\rho_{m,n}$ to a representation $\hat{\rho}_{m,n}$ of singular braid monoid into the colored Temperley-Lieb algebra. 
\section{The Kauffman-Vogel Polynomial for Rigid 4-Valent Graphs}
\label{sec3}
A \textit{rigid $4$-valent graph} on $n$ components is the image of a smooth immersion of $n$ circles in $S^3$ that has finitely many double points, called vertices. Rigid $4$-valent graphs are also called sometimes singular knots. Similarly, the vertices are sometimes called singularities. Two rigid $4$-valent graphs are ambient isotopic if there is an orientation preserving self-homeomorphism of $S^3$ that takes one graph to the other and preserves a small rigid disk around each vertex. We will deal with graph diagrams, which are projections of the graph in the plane such that the information at each crossing is preserved by leaving a little break in the lower strand. Two rigid $4$-valent graphs $G_1$ and $G_2$ are ambient isotopic if and only if one can obtain a diagram of $G_2$ from a diagram of $G_1$ by a finite sequence of classical and singular Reidemeister moves as in Figure \ref{Rmoves}. See \cite{kaufgraph} for more details. 

\begin{figure}[H]
  \centering
   {\includegraphics[scale=0.1]{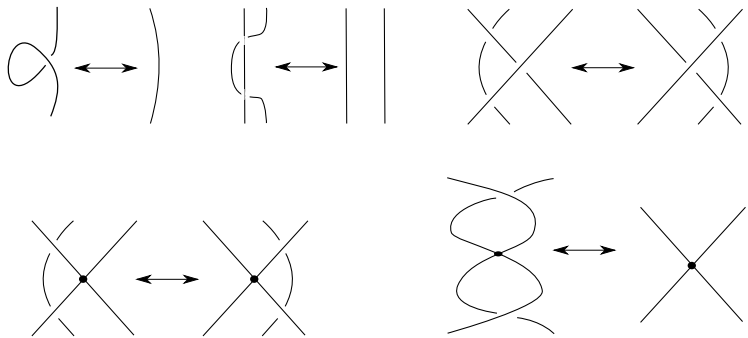}
     \caption{Classical Reidemeister moves $RI$, $RII$ and $RIII$ on the top and singular Reidemeister $RIV$ and $RV$ on the bottom.}
  \label{Rmoves}}
\end{figure}
If one does not allow the move on the top left of Figure $\ref{Rmoves}$ in the sequence, then we obtain what is called  {\it regular isotopy} of rigid $4$-valent graphs.

As we mentioned in the introduction, a one variable version of the Kauffman-Vogel polynomial invariant can be obtained using the Jones-Wenzl projector \cite{KauffLin}. We recall this version here. For a $4$-valent rigid vertex embedded graph $G$, we will refer to this polynomial by $[G]_2$.   

\begin{definition}\label{defG2}
The polynomial $[G]_2$ is defined recursively via the following five axioms:
\begin{enumerate}
\item

$   \left[
   \begin{minipage}[h]{0.08\linewidth}
        \vspace{0pt}
        \scalebox{0.15}{\includegraphics{pos_crossing}}
   \end{minipage}
  \right]_2  = A^4  \left[
   \begin{minipage}[h]{0.08\linewidth}
        \vspace{0pt}
        \scalebox{0.15}{\includegraphics{other_smoothing}}
   \end{minipage}\right]_2
  + A^{-4}\left[
   \begin{minipage}[h]{0.07\linewidth}
        \vspace{0pt}
        \scalebox{0.15}{\includegraphics{one_smoothing}}
   \end{minipage}
   \right]_2+(A^2+A^{-2})  \left[
   \begin{minipage}[h]{0.08\linewidth}
        \vspace{0pt}
        \scalebox{0.08}{\includegraphics{singular}}
   \end{minipage}
  \right]_2$
\item
$ \left[ \begin{minipage}[h]{0.035\linewidth}
        \vspace{0pt}
        \scalebox{0.25}{\includegraphics{R_1}}
   \end{minipage}  \right]_2= A^{8}  \left[ \begin{minipage}[h]{0.035\linewidth}
        \vspace{0pt}
        \scalebox{0.25}{\includegraphics{R_1P}}
   \end{minipage} \right]_2 $ and $\left[ \begin{minipage}[h]{0.035\linewidth}
        \vspace{0pt}
        \scalebox{0.25}{\includegraphics{R_1M}}
   \end{minipage}  \right]_2= A^{-8}  \left[ \begin{minipage}[h]{0.035\linewidth}
        \vspace{0pt}
        \scalebox{0.25}{\includegraphics{R_1P}}
   \end{minipage} \right]_2 $. 
\item
\vspace{3pt}
$ \left[ \begin{minipage}[h]{0.05\linewidth}
        \vspace{0pt}
        \scalebox{0.02}{\includegraphics{simple-circle}}
   \end{minipage}  \right]_2= 2 +A^{-4} + A^4$
 \item
\vspace{3pt}  
 $ \left[
   \begin{minipage}[h]{0.08\linewidth}
        \vspace{0pt}
        \scalebox{0.08}{\includegraphics{singular}}
   \end{minipage}
  \right]_2=\begin{minipage}[h]{0.08\linewidth}
        \vspace{0pt}
        \scalebox{0.08}{\includegraphics{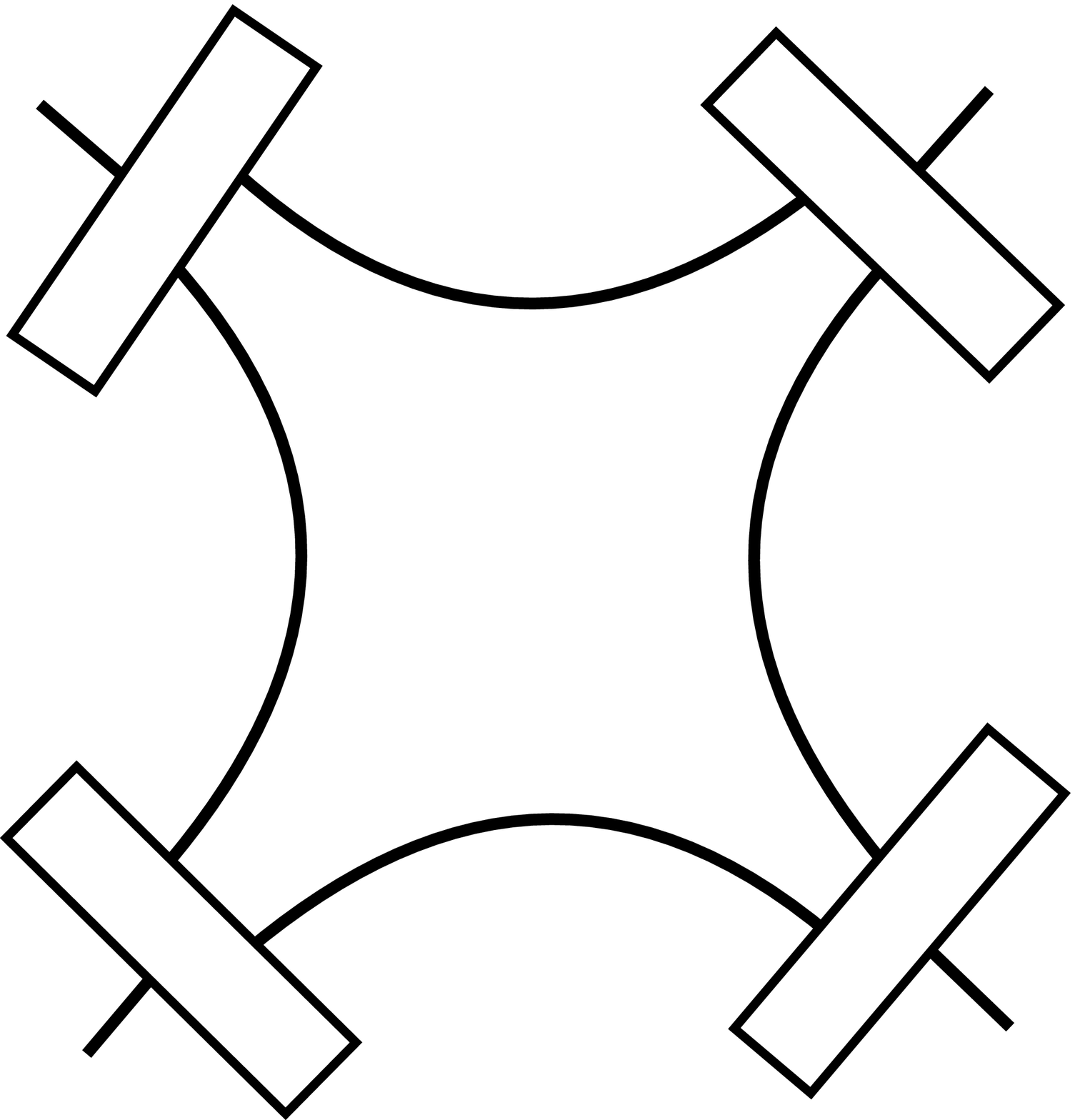}}
        \tiny{
        \put(-28,40){$1$}
        \put(-28,5){$1$}
        \put(-12,20){$1$}
        \put(-40,20){$1$}
        \put(-1,45){$2$}
        \put(-48,45){$2$}
        \put(-1,0){$2$}
        \put(-48,0){$2$}}\label{identity4}
   \end{minipage}$ 
 \item
\vspace{3pt}  
 $ \left[
   \begin{minipage}[h]{0.08\linewidth}
        \vspace{0pt}
        \scalebox{0.08}{\includegraphics{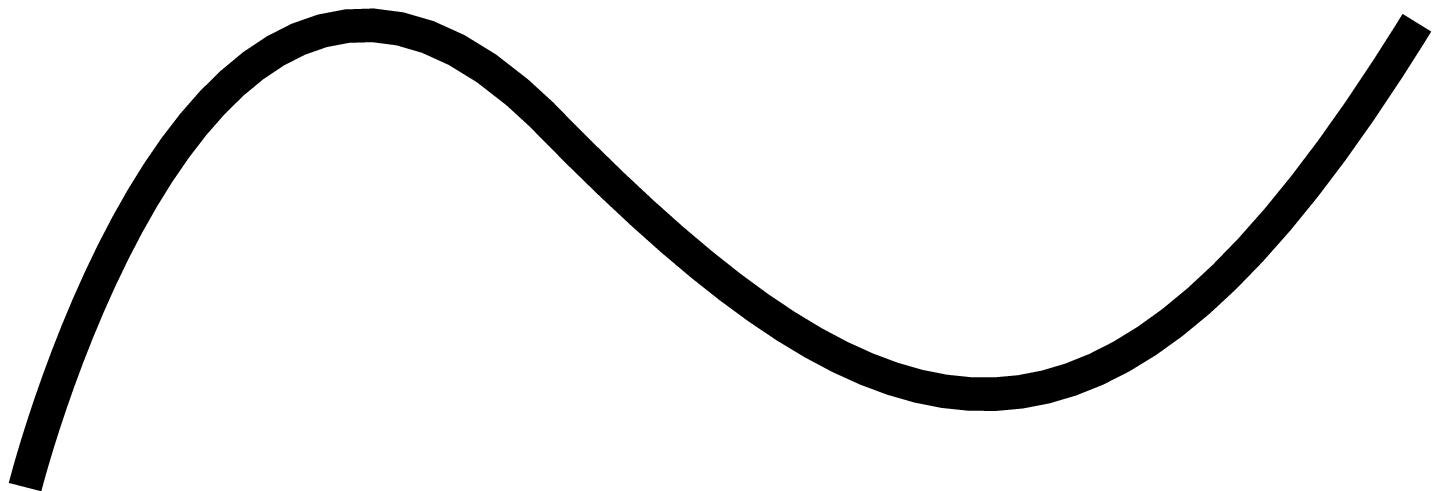}}
   \end{minipage}
  \right]_2=\begin{minipage}[h]{0.08\linewidth}
        \vspace{0pt}
        \scalebox{0.1}{\includegraphics{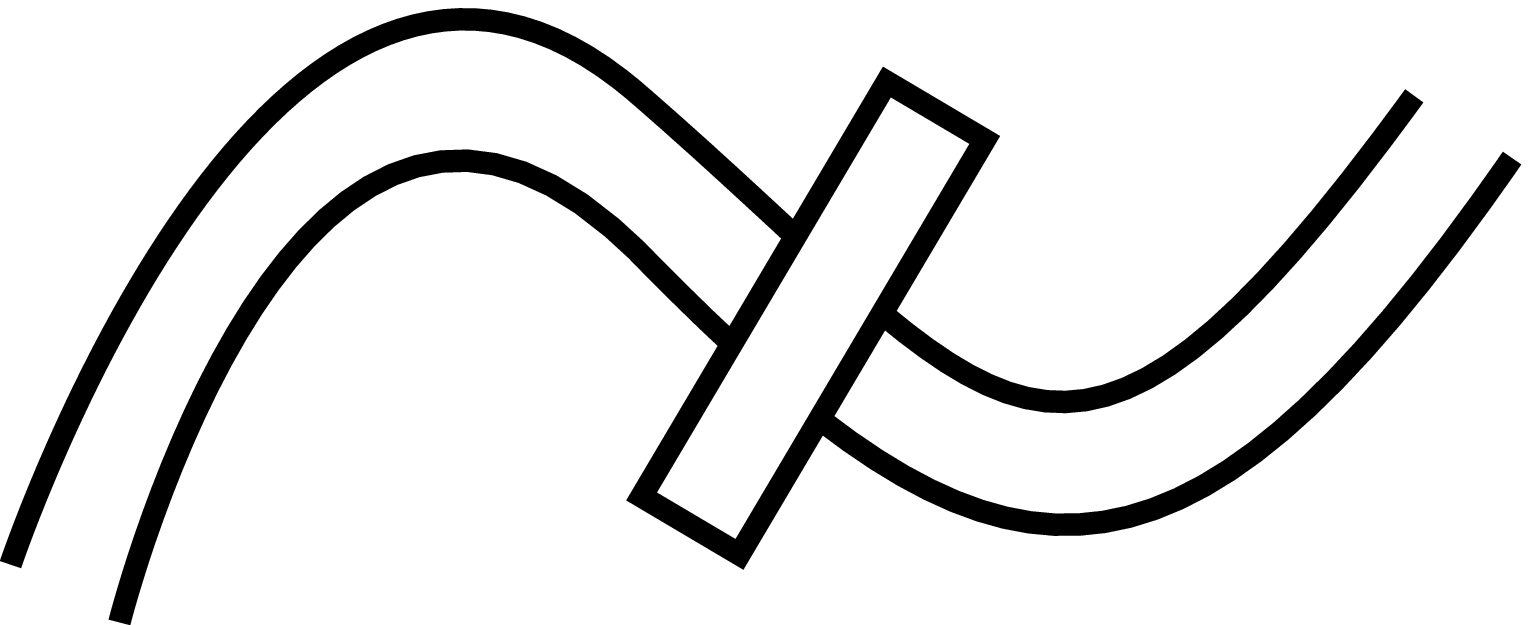}}
   \end{minipage}$     
  
\end{enumerate}	
\end{definition}
 In the next section we rewrite the first axiom in a slightly different way which helps us in our generalization of this invariant.	
\section{Colored Kauffman-Vogel Polynomial for Rigid 4-Valent Graphs}
\label{sec4}
In this section we give a generalization for the one-variable specialization of the Kauffman-Vogel polynomial given in the previous section. This invariant can also be seen as an extension for the colored Jones polynomial to $4$-valent graph. The one variable specialization of the Kauffman-Vogel polynomial that we gave in the previous section can be defined via the following rules: 
\begin{enumerate}
\item

$   \left[
   \begin{minipage}[h]{0.08\linewidth}
        \vspace{0pt}
        \scalebox{0.15}{\includegraphics{pos_crossing}}
   \end{minipage}
  \right]_2  = \begin{minipage}[h]{0.08\linewidth}
        \vspace{0pt}
        \scalebox{0.08}{\includegraphics{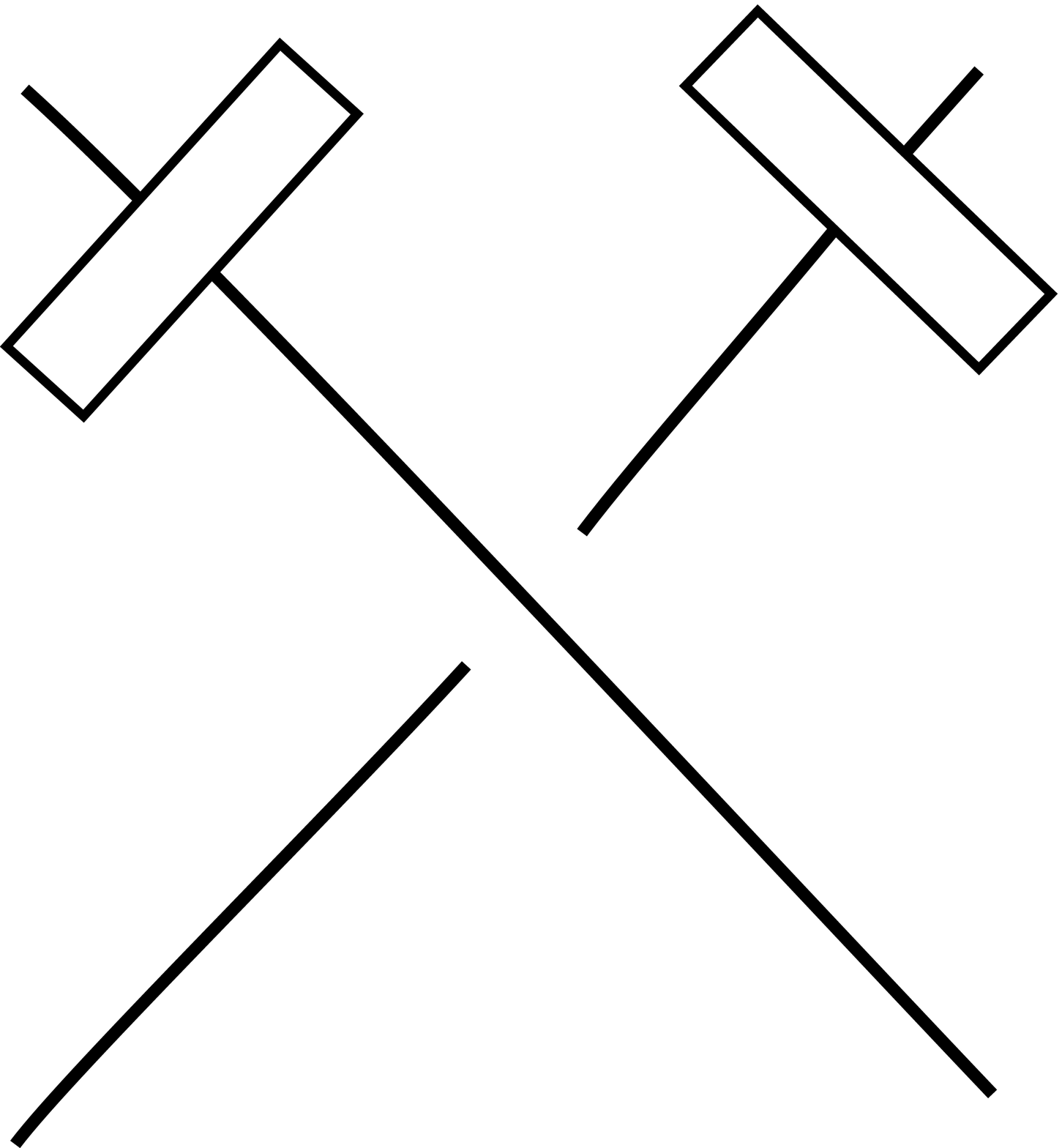}}
        \tiny{
        \put(-1,45){$2$}
        \put(-48,45){$2$}}
   \end{minipage}$
 \item
\vspace{3pt}  
 $ \left[
   \begin{minipage}[h]{0.08\linewidth}
        \vspace{0pt}
        \scalebox{0.08}{\includegraphics{singular}}
   \end{minipage}
  \right]_2=\begin{minipage}[h]{0.08\linewidth}
        \vspace{0pt}
        \scalebox{0.08}{\includegraphics{singular_map}}
        \tiny{
        \put(-28,40){$1$}
        \put(-28,5){$1$}
        \put(-12,20){$1$}
        \put(-40,20){$1$}
        \put(-1,45){$2$}
        \put(-48,45){$2$}
        \put(-1,0){$2$}
        \put(-48,0){$2$}}
   \end{minipage}$ 
 \item
\vspace{3pt}  
 $ \left[
   \begin{minipage}[h]{0.08\linewidth}
        \vspace{0pt}
        \scalebox{0.08}{\includegraphics{strand}}
   \end{minipage}
  \right]_2=\begin{minipage}[h]{0.08\linewidth}
        \vspace{0pt}
        \scalebox{0.1}{\includegraphics{strand_idm}}
   \end{minipage}$     
  
\end{enumerate}		
Replacing the five axioms in the Definition \ref{defG2} by the three axioms given above follows from the following facts

\begin{equation}\label{n=1}  \begin{minipage}[h]{0.08\linewidth}
        \vspace{0pt}
        \scalebox{0.08}{\includegraphics{colored_corssing}}
        \tiny{
        \put(-1,45){$2$}
        \put(-48,45){$2$}}
   \end{minipage}  
    = A^4 
   \begin{minipage}[h]{0.1\linewidth}
        \vspace{0pt}
        \scalebox{0.19}{\includegraphics{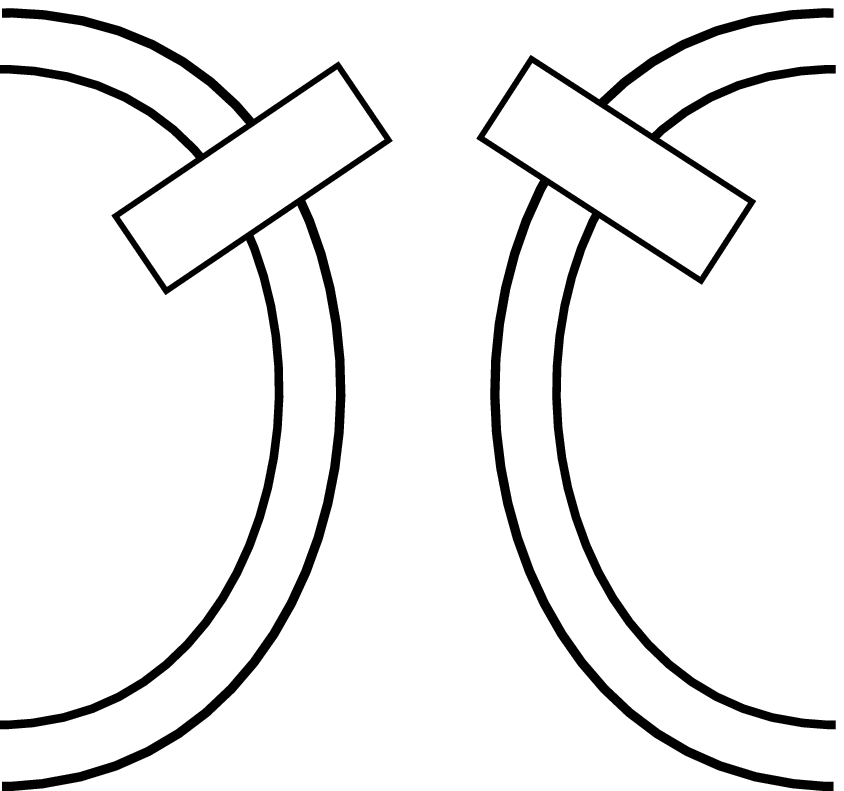}}
   \end{minipage}
  + A^{-4}
   \begin{minipage}[h]{0.1\linewidth}
        \vspace{0pt}
        \scalebox{0.19}{\includegraphics{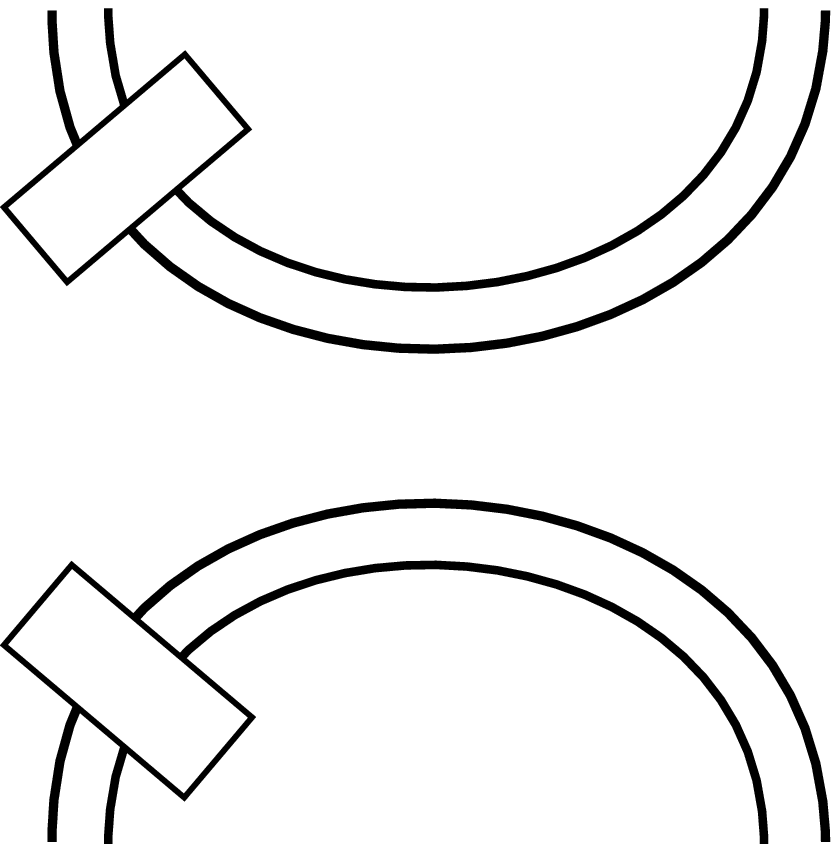}}
   \end{minipage}
   +(A^2+A^{-2})  
   \begin{minipage}[h]{0.08\linewidth}
        \vspace{0pt}
        \scalebox{0.08}{\includegraphics{singular_map}}
        \tiny{
        \put(-28,40){$1$}
        \put(-28,5){$1$}
        \put(-12,20){$1$}
        \put(-40,20){$1$}
        \put(-1,45){$2$}
        \put(-48,45){$2$}
        \put(-1,0){$2$}
        \put(-48,0){$2$}}
   \end{minipage}
\end{equation}
and
$$
 \begin{minipage}[h]{0.1\linewidth}
        \vspace{0pt}
        \scalebox{0.12}{\includegraphics{deltan}}
        \put(-29,+34){\footnotesize{$2$}}
   \end{minipage}=1+A^{-4}+A^4.
$$
Before we give our generalization for $[G]_2$, we prove the following two lemmas.

\begin{lemma}
\label{technical lemma}
\begin{enumerate}
For $n \geq 1$ the following identities holds :
\item 
\begin{eqnarray}
   \begin{minipage}[h]{0.15\linewidth}
        \vspace{0 pt}
        \scalebox{.15}{\includegraphics{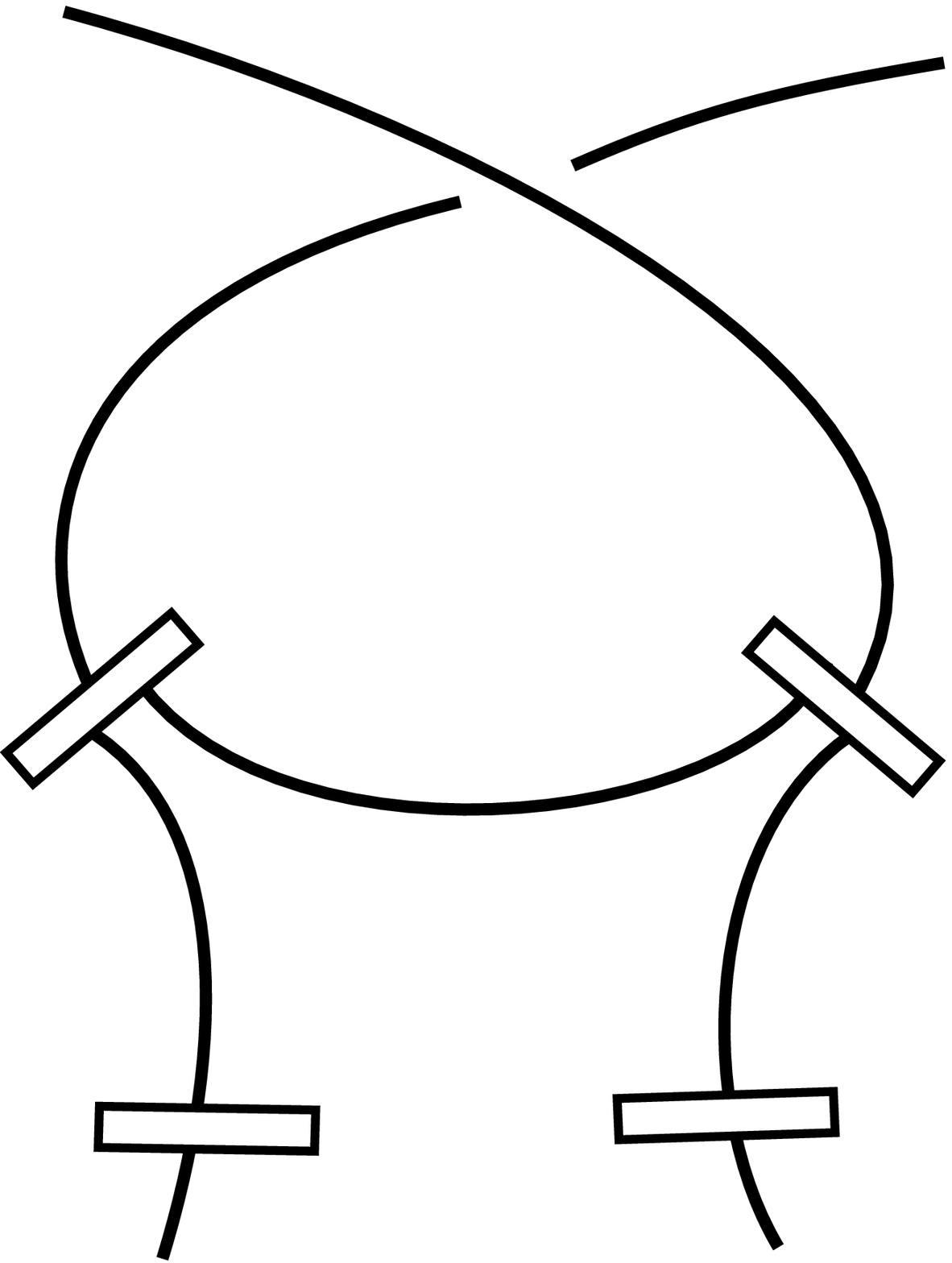}}
        \scriptsize{
         \put(-65,+52){$2n$}
          \put(-1,+52){$2n$}}
   \end{minipage}
   =(-1)^n A^{3n^2+2n}
  \begin{minipage}[h]{0.15\linewidth}
        \vspace{0pt}
        \scalebox{.18}{\includegraphics{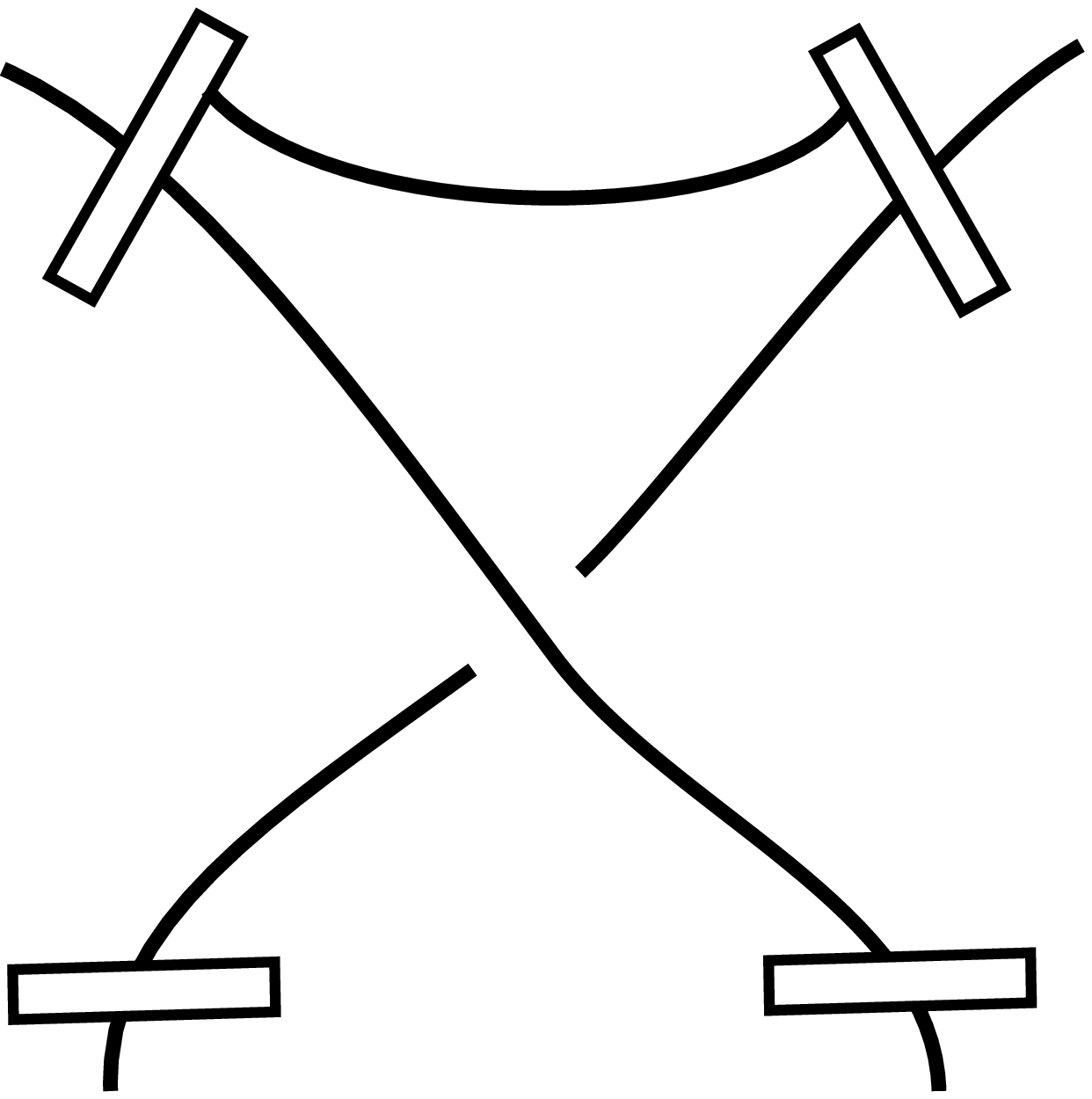}}
        \scriptsize{
        \put(-68,+76){$2n$}
          \put(0,+75){$2n$}
          \put(-58,26){$n$}
          \put(-22,26){$n$}
         \put(-43,+56){$n$}
           }
          \end{minipage}
  \end{eqnarray}
\item 
\begin{eqnarray}
\begin{minipage}[h]{0.15\linewidth}
\vspace{0 pt}
\scalebox{.15}{\includegraphics{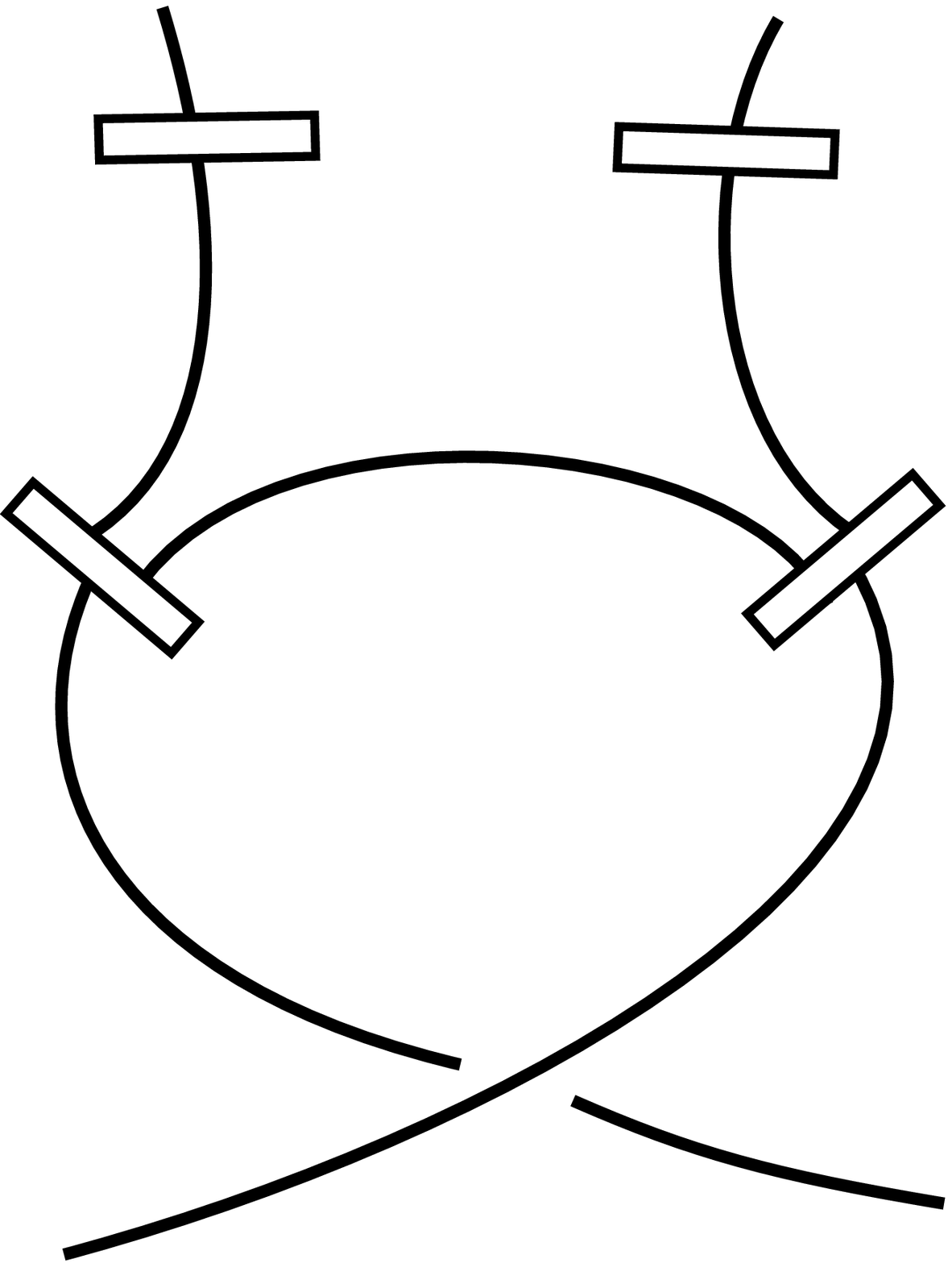}}
\scriptsize{
	\put(-65,+52){$2n$}
	\put(-1,+52){$2n$}}
\end{minipage}
=(-1)^n A^{-3n^2-2n}
\begin{minipage}[h]{0.15\linewidth}
\vspace{0pt}
\scalebox{.18}{\includegraphics{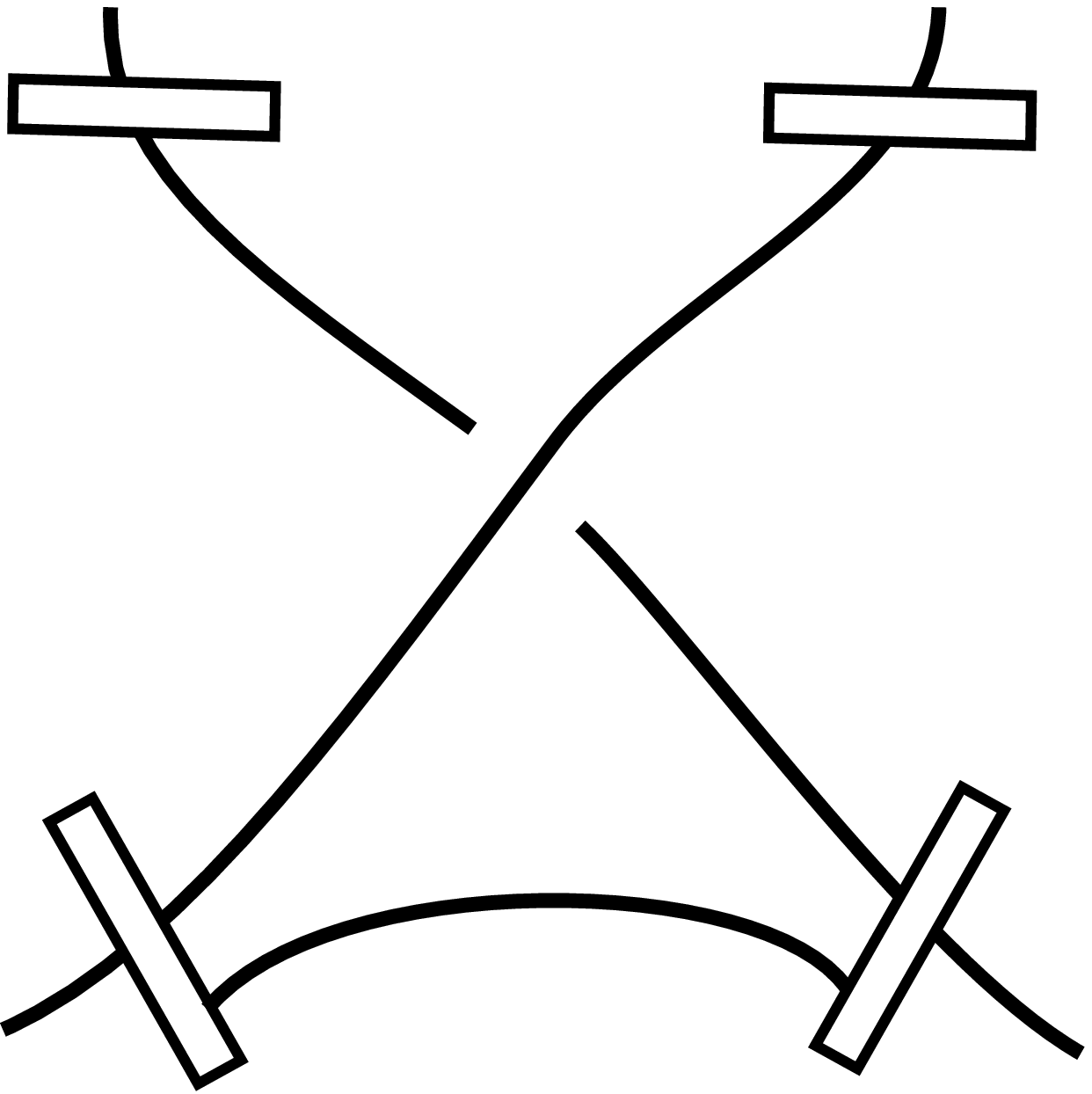}}
\scriptsize{
	\put(-68,-3){$2n$}
	\put(0,-3){$2n$}
	\put(-58,26){$n$}
	\put(-22,26){$n$}
	\put(-43,+56){$n$}
}
\end{minipage}
\end{eqnarray}
\end{enumerate}
\end{lemma}
\begin{proof}
By isotopy we have
\vspace{10 pt}
\begin{eqnarray}
\label{GOO}
   \begin{minipage}[h]{0.15\linewidth}
        \hspace{5 pt}
        \scalebox{.15}{\includegraphics{proof_1.eps}}
        \scriptsize{
         \put(-62,+52){$2n$}
          \put(-7,+52){$2n$}}
   \end{minipage}
   =
  \begin{minipage}[h]{0.15\linewidth}
        \vspace{0pt}
        \scalebox{.18}{\includegraphics{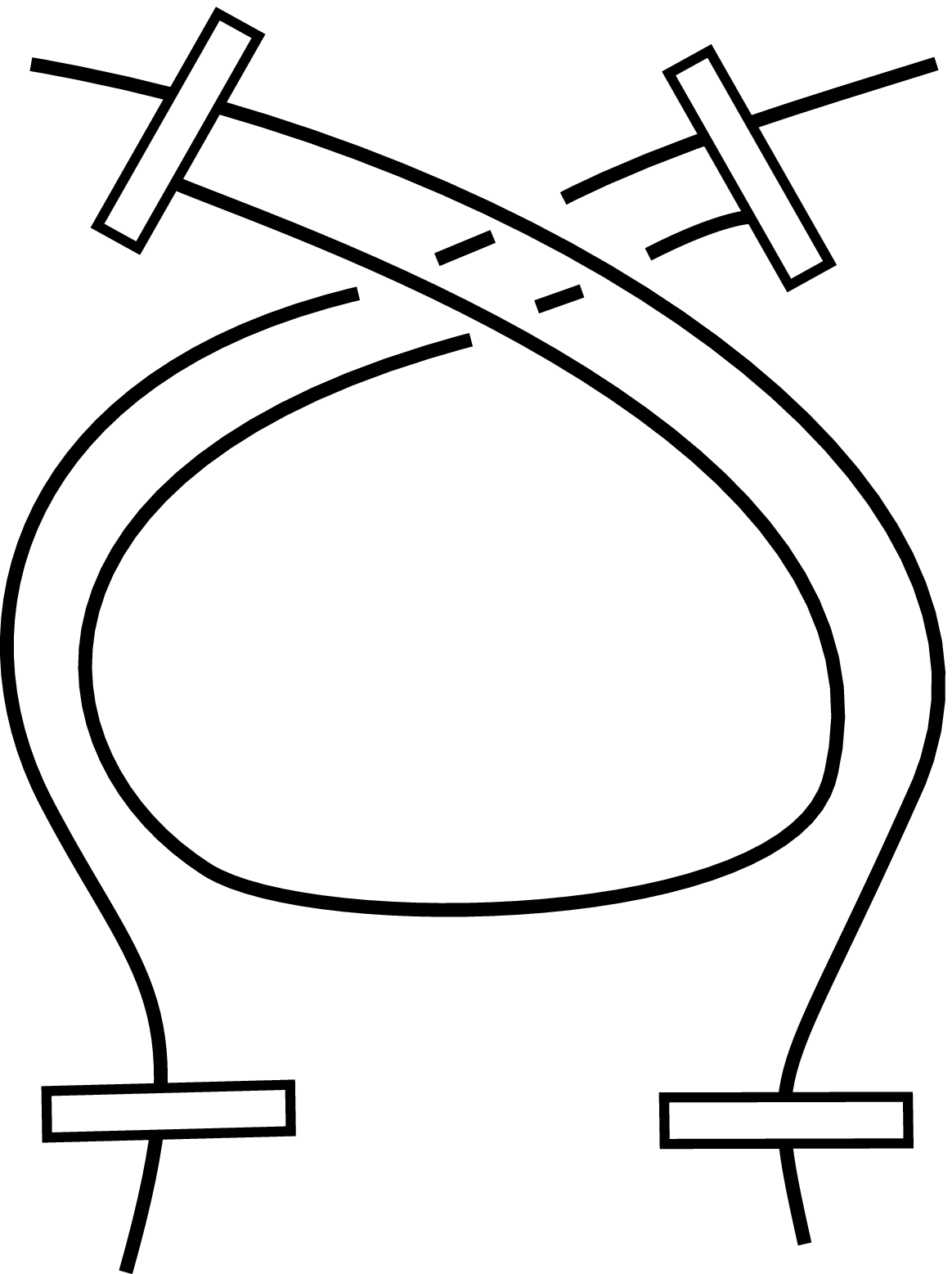}}
        \scriptsize{
          \put(-60,13){$n$}
          \put(-5,13){$n$}
           \put(-35,+17){$n$}
           }
          \end{minipage}
  \end{eqnarray}

Using property $(2)$ in \ref{properties_2}  we obtain: 
\vspace{38pt}
\begin{eqnarray*}
\vspace{10 pt}
   \begin{minipage}[h]{0.15\linewidth}
        \vspace{-18pt}
        \scalebox{.18}{\includegraphics{proof_2.eps}}
        \scriptsize{
        	\put(-67,66){$2n$}
        	\put(0,66){$2n$}
        \put(-60,13){$n$}
          \put(-5,13){$n$}
           }
          \end{minipage}
   =(-1)^n A^{n^2+2n}
  \begin{minipage}[h]{0.15\linewidth}
        \vspace{0pt}
        \scalebox{.16}{\includegraphics{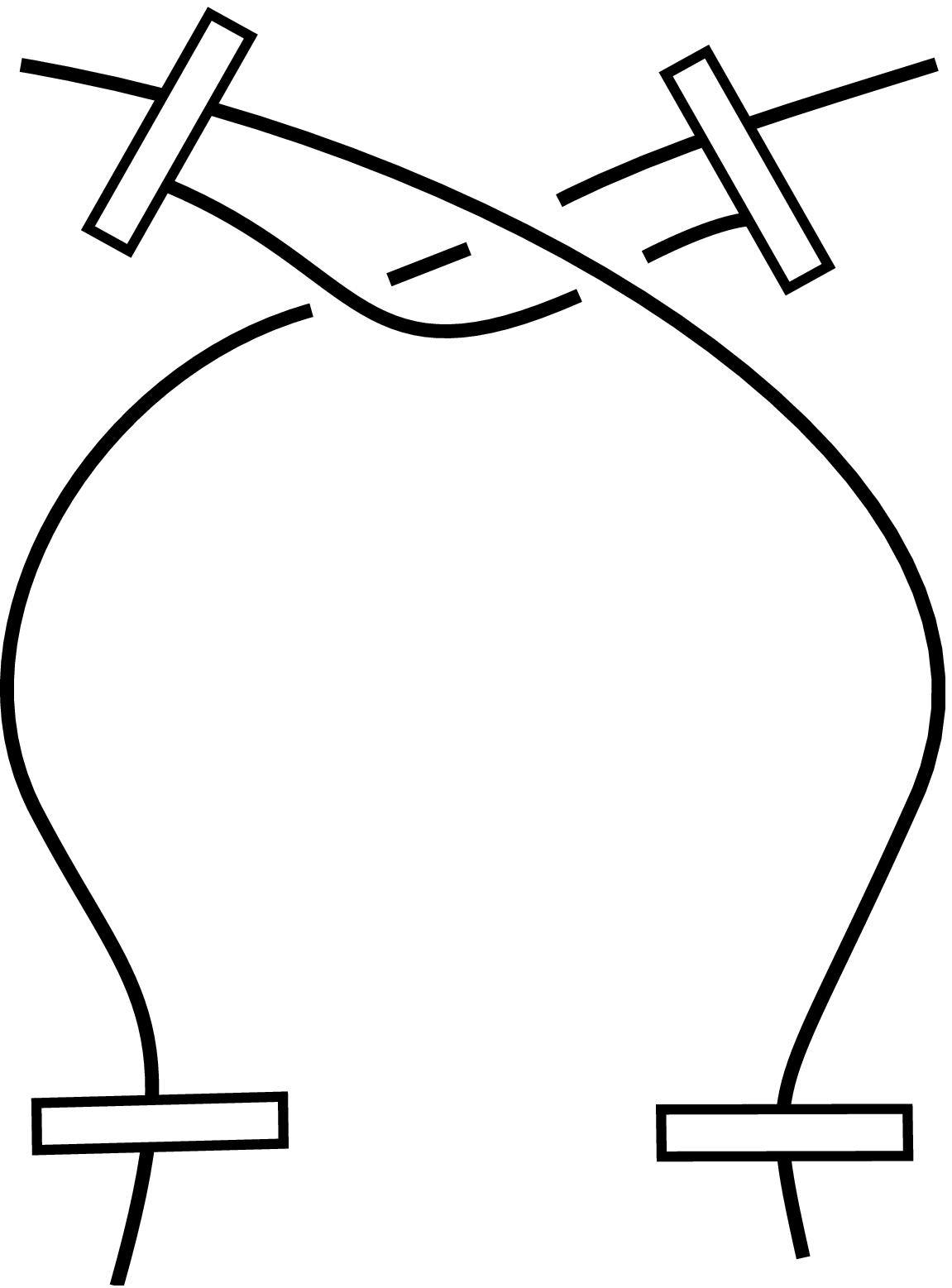}}
        \scriptsize{
        	\put(-60,63){$2n$}
        	\put(0,63){$2n$}
          \put(-60,13){$n$}
          \put(-5,13){$n$}
           }
          \end{minipage}
  \end{eqnarray*}  
The fact that one can do Reidemeister moves $II$ and $III$ for strands colored by the Jones-Wenzl projector implies : 
\begin{eqnarray*}
   \begin{minipage}[h]{0.15\linewidth}
        \vspace{-19 pt}
        \scalebox{.15}{\includegraphics{proof_3_.eps}}
        \scriptsize{
        	\put(-60,63){$2n$}
        	\put(0,63){$2n$}
           \put(-60,13){$n$}
          \put(-5,13){$n$}
          }
   \end{minipage}
   =
  \begin{minipage}[h]{0.15\linewidth}
        \vspace{0pt}
        \scalebox{.16}{\includegraphics{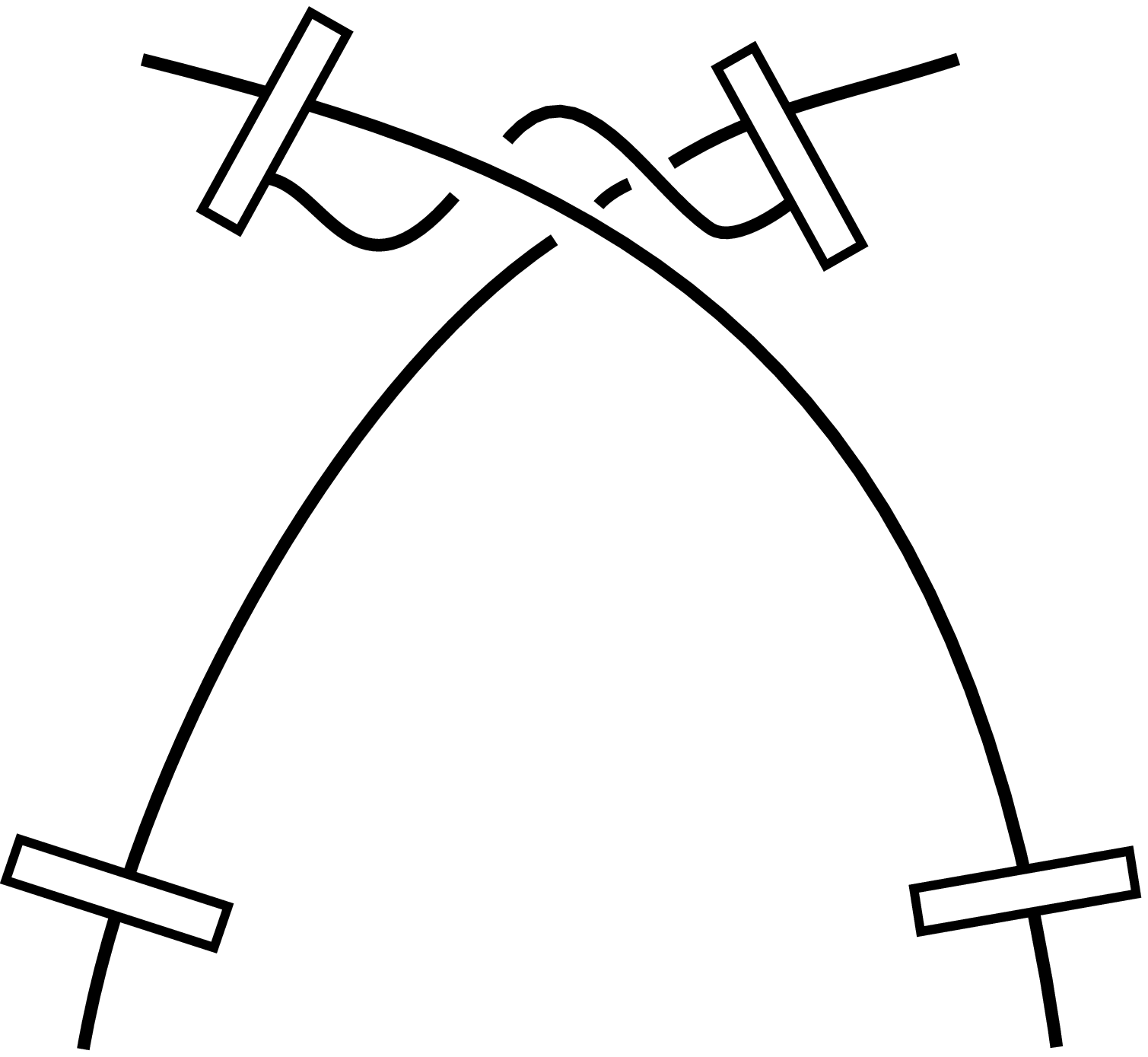}}
        \scriptsize{
        	\put(-68,83){$2n$}
        	\put(-10,83){$2n$}
           \put(-70,45){$n$}
          \put(-5,45){$n$}
           }
          \end{minipage}
  \end{eqnarray*}  
Finally, using property $(1)$ in \ref{properties_2}, one has:

\begin{eqnarray*}
   \begin{minipage}[h]{0.15\linewidth}
        \vspace{18pt}
        \scalebox{.16}{\includegraphics{proof_3.eps}}
        \scriptsize{
          \put(-70,45){$n$}
          \put(-5,45){$n$}
           }
          \end{minipage}
   =A^{2n^2}
  \begin{minipage}[h]{0.15\linewidth}
        \vspace{0pt}
        \scalebox{.18}{\includegraphics{proof_5.eps}}
        \scriptsize{
        \put(-68,+76){$2n$}
          \put(0,+75){$2n$}
          \put(-58,30){$n$}
          \put(-22,30){$n$}
         \put(-43,+56){$n$}
           }
          \end{minipage}
  \end{eqnarray*}  
 The result follows. 
\end{proof}

\begin{lemma}
\label{move V}
Let $n \geq 1 $. The following identity holds in the Temperley-Lieb algebra $TL_{4n}$:
\begin{eqnarray*}
   \begin{minipage}[h]{0.12\linewidth}
        \vspace{0pt}
        \scalebox{.13}{\includegraphics{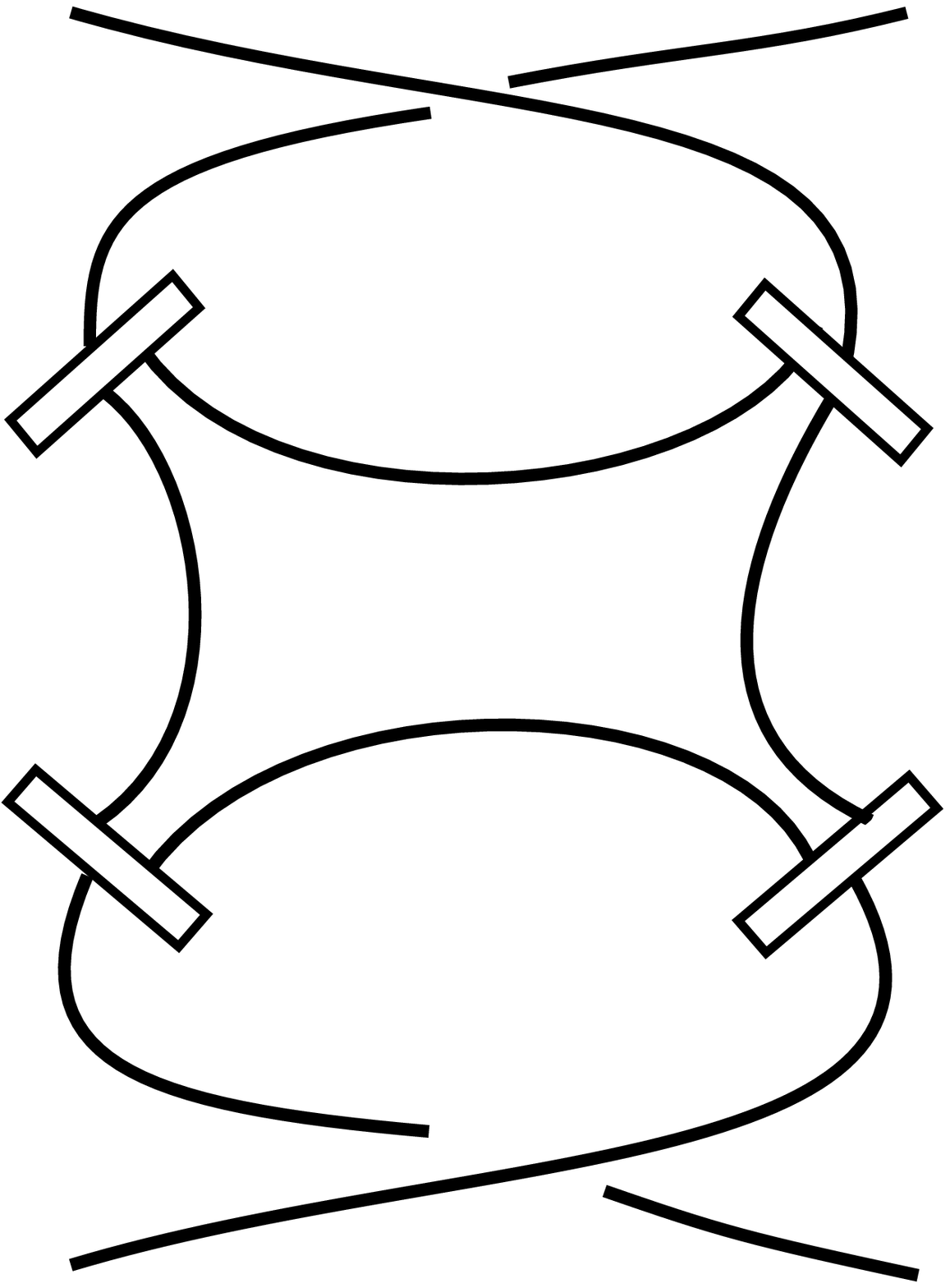}}
        \scriptsize{
          \put(-56,65){$2n$}
          \put(-5,65){$2n$}
          \put(-56,0){$2n$}
          \put(0,0){$2n$}
           }
          \end{minipage}=
  \begin{minipage}[h]{0.08\linewidth}
        \vspace{0pt}
        \scalebox{0.08}{\includegraphics{singular_map}}
        \tiny{
        \put(-28,40){$n$}
        \put(-28,5){$n$}
        \put(-12,20){$n$}
        \put(-40,20){$n$}
        \put(-1,45){$2n$}
        \put(-48,45){$2n$}
        \put(-1,-3){$2n$}
        \put(-48,-3){$2n$}}
   \end{minipage}
  \end{eqnarray*}  
\end{lemma}
\begin{proof}
The previous lemma implies:
\begin{eqnarray*}
    \begin{minipage}[h]{0.12\linewidth}
        \vspace{0pt}
        \scalebox{.13}{\includegraphics{proof_1f.eps}}
        \scriptsize{
          \put(-56,65){$2n$}
          \put(-5,65){$2n$}
          \put(-56,0){$2n$}
          \put(0,0){$2n$}
           }
          \end{minipage}
  &=& (-1)^{-n} A^{-3n^2-2n}\quad\begin{minipage}[h]{0.12\linewidth}
        \vspace{0pt}
        \scalebox{.13}{\includegraphics{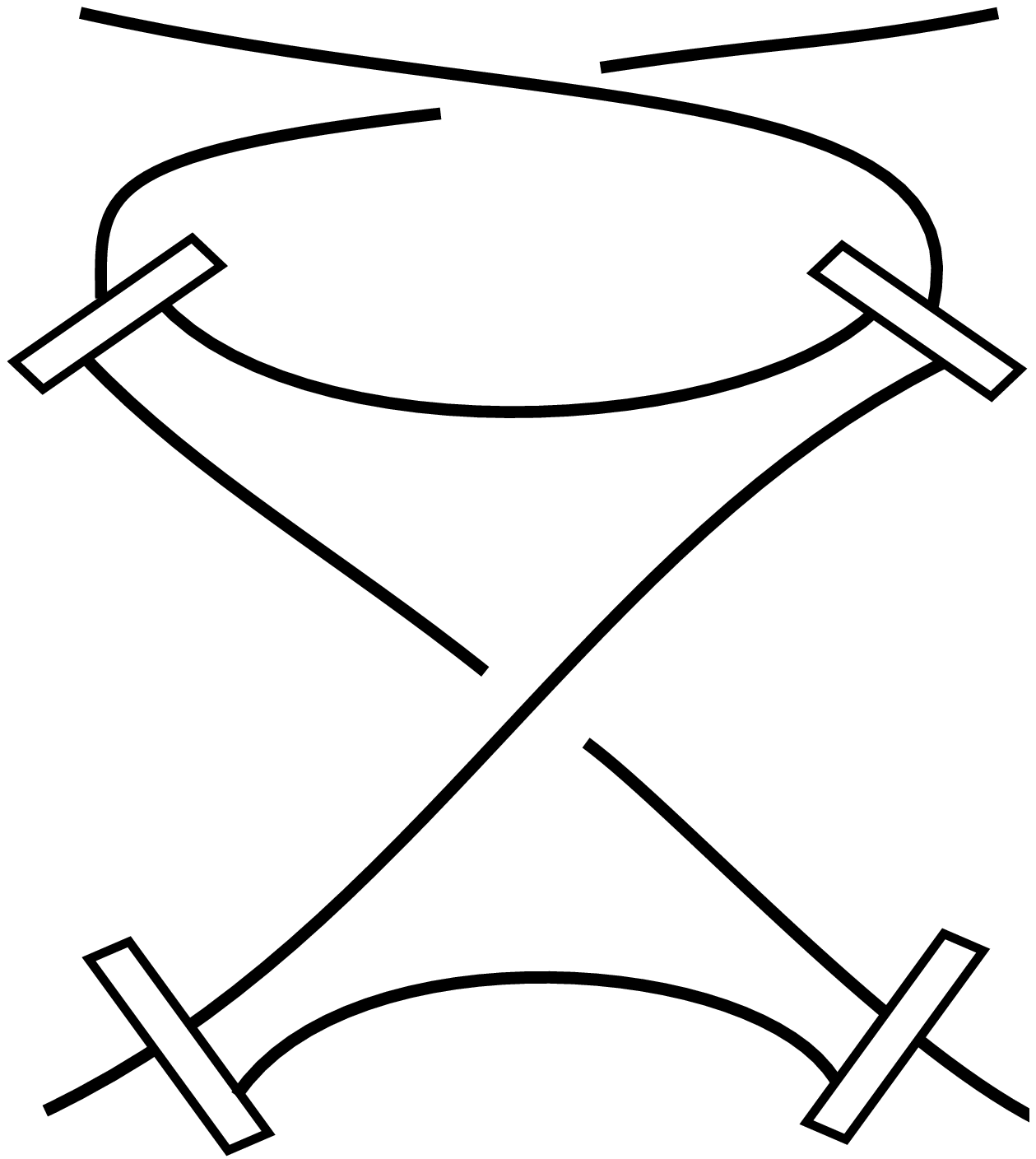}}
        \scriptsize{
          \put(-56,45){$2n$}
          \put(0,45){$2n$}
           \put(-35,15){$n$}
          \put(-15,15){$n$}
           }
          \end{minipage}
          \\&=&(-1)^{n} A^{3n^2+2n} (-1)^{-n} A^{-3n^2-2n}\begin{minipage}[h]{0.15\linewidth}
        \vspace{0pt}
        \scalebox{.18}{\includegraphics{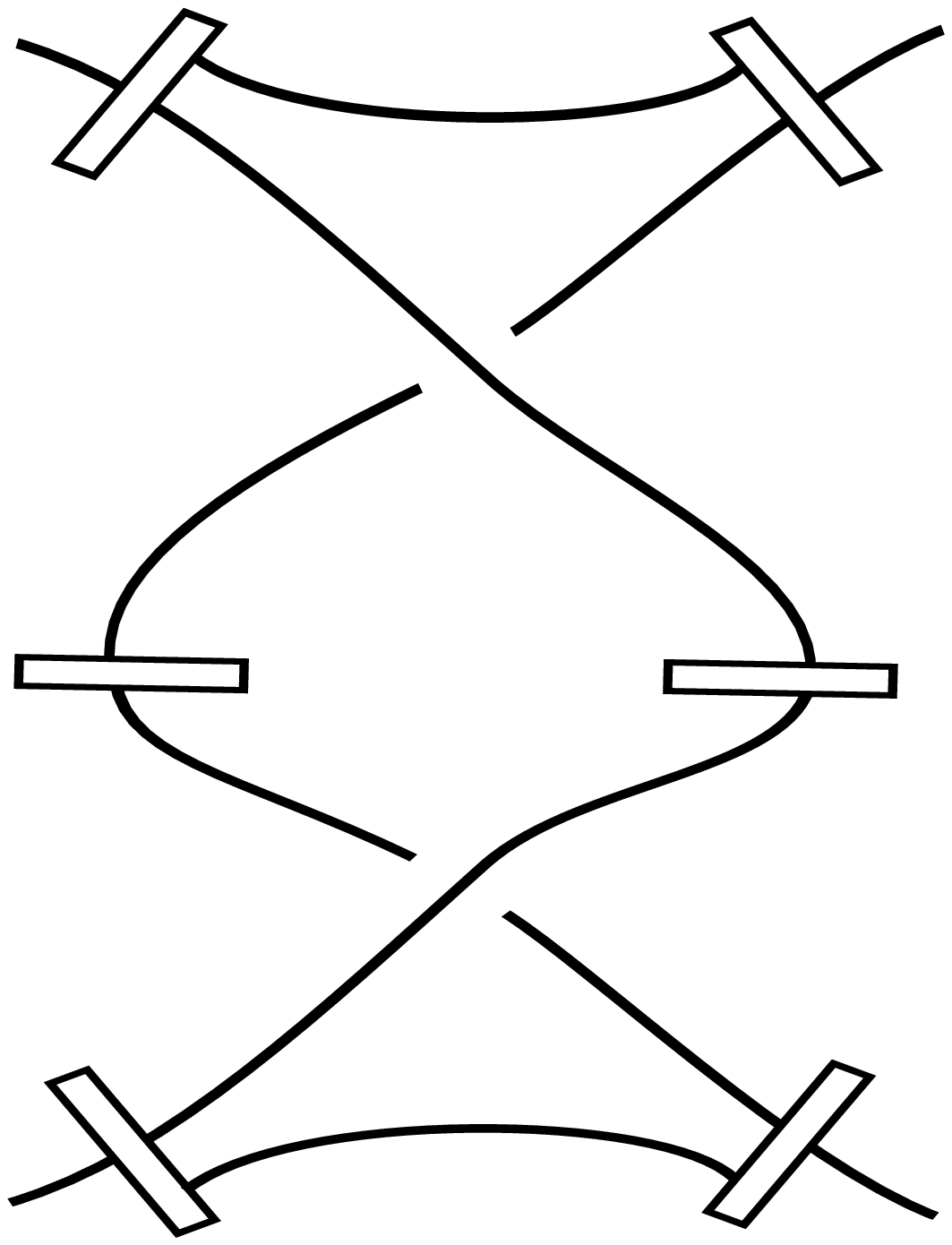}}
        \scriptsize{
          \put(-58,-5){$2n$}
          \put(-1,-5){$2n$}
           \put(-58,73){$2n$}
          \put(-1,73){$2n$}
           }
          \end{minipage}=\begin{minipage}[h]{0.08\linewidth}
        \vspace{0pt}
        \scalebox{0.08}{\includegraphics{singular_map}}
        \tiny{
        \put(-28,40){$n$}
        \put(-28,5){$n$}
        \put(-12,20){$n$}
        \put(-40,20){$n$}
        \put(-1,45){$2n$}
        \put(-48,45){$2n$}
        \put(-1,-3){$2n$}
        \put(-48,-3){$2n$}}
   \end{minipage}
  \end{eqnarray*}
 The last equation follows by doing a Reidemeister $II$ move on the strands. The result follows.  
\end{proof}

\begin{theorem}
\label{main}
Let $G$ be a $4$-valent graph. For an integer $n \geq 1 $, the rational function $[G]_{2n}$ defined  by the rules  

\begin{enumerate}
\item

$   \left[
   \begin{minipage}[h]{0.08\linewidth}
        \vspace{0pt}
        \scalebox{0.15}{\includegraphics{pos_crossing}}
   \end{minipage}
  \right]_{2n}  = \begin{minipage}[h]{0.08\linewidth}
        \vspace{0pt}
        \scalebox{0.08}{\includegraphics{colored_corssing}}
        \tiny{
        \put(-1,45){$2n$}
        \put(-48,45){$2n$}}
   \end{minipage}$
 \item
\vspace{3pt}  
 $ \left[
   \begin{minipage}[h]{0.08\linewidth}
        \vspace{0pt}
        \scalebox{0.08}{\includegraphics{singular}}
   \end{minipage}
  \right]_{2n}=\begin{minipage}[h]{0.08\linewidth}
        \vspace{0pt}
        \scalebox{0.08}{\includegraphics{singular_map}}
        \tiny{
        \put(-28,40){$n$}
        \put(-28,5){$n$}
        \put(-12,20){$n$}
        \put(-40,20){$n$}
        \put(-1,45){$2n$}
        \put(-48,45){$2n$}
        \put(-1,-3){$2n$}
        \put(-48,-3){$2n$}}
   \end{minipage}$ 
 \item
\vspace{8pt}  
 $ \left[
   \begin{minipage}[h]{0.08\linewidth}
        \vspace{0pt}
        \scalebox{0.08}{\includegraphics{strand}}
   \end{minipage}
  \right]_{2n}=\begin{minipage}[h]{0.08\linewidth}
        \vspace{0pt}
        \scalebox{0.1}{\includegraphics{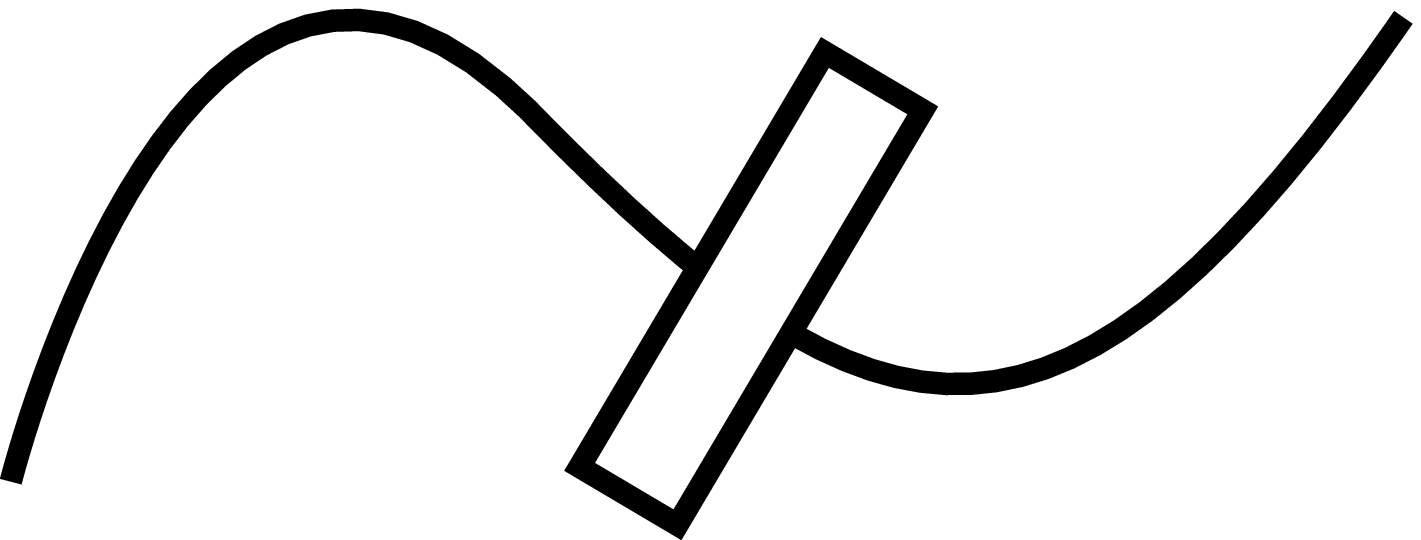}}
        \tiny{ \put(-3,5){$2n$}}
   \end{minipage}$     
  
\end{enumerate}		
is a regular isotopy invariant for rigid $4$-valent graphs.

\end{theorem}
\begin{proof}
The moves shown in Figure \ref{CRmoves} are  a finite sequence of the usual Reidemeister moves $II$ and $III$ applied on each single strand and summand of the idempotents. Hence $[.]_{2n}$ is invariant under Reidemeister moves $II$ and $III$. The same argument holds for the two diagrams in Figure \ref{forthmove} and hence $[.]_{2n}$ is invariant under Reidemeister $IV$. Finally, the invariance under move $V$ follows from Lemma \ref{move V}.

\begin{figure}[H]
\label{forthmove}
  \centering
   {\includegraphics[scale=0.12]{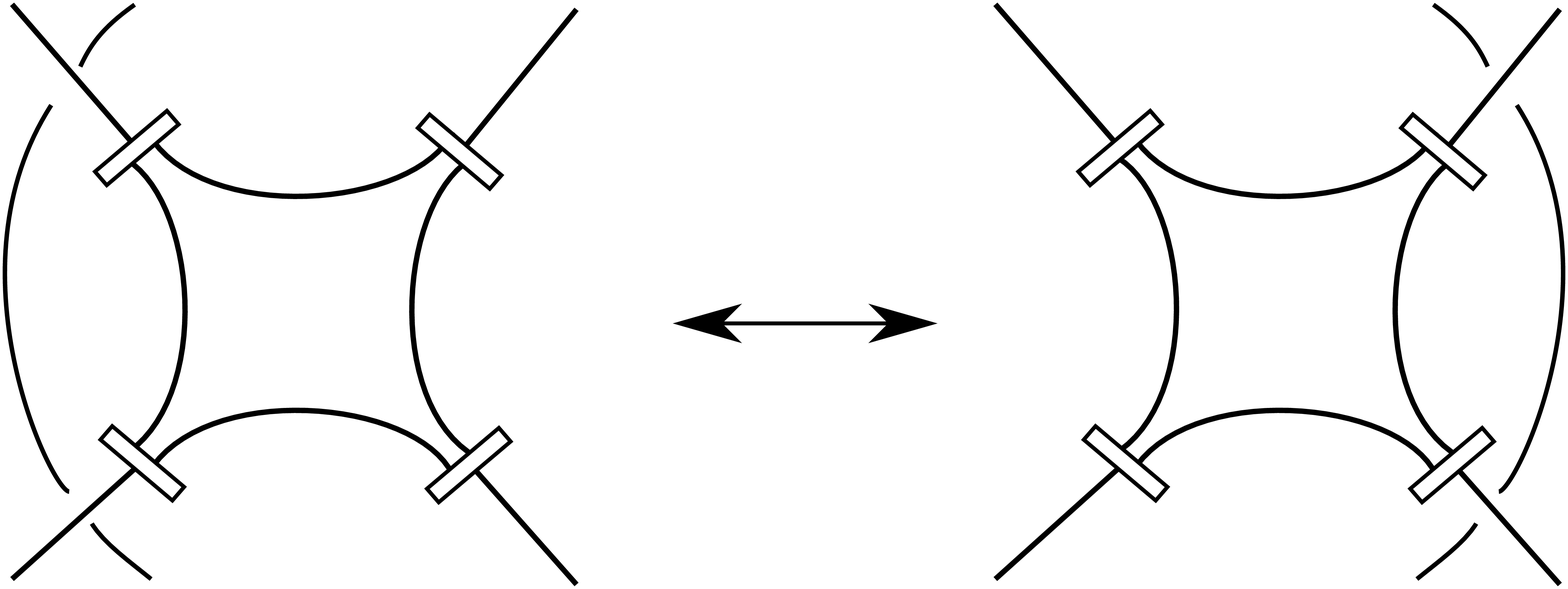}
     \caption{}
  \label{singular move}}
\end{figure}
\end{proof}
\begin{remark}
The invariant $[.]_{2n}$ can be seen to be an extension for the unreduced colored Jones polynomial $\tilde{J}(.,2n)$ for links in $S^3$. Namely, for a zero-framed knot $K$ in $S^3$ we have $\tilde{J}(K,2n)=[K]_{2n}$.
\end{remark}
\subsection{Examples}
In this sub-section we give some computational  examples of our invariants. Before we compute some examples we give some identities that we will use in our computations. 
Recall that the $q$-Pochhammer is defined as 
\begin{equation*}
(a;q)_n=\prod\limits_{j=0}^{n-1}(1-aq^j).
\end{equation*}   
We will need the following fact from \cite{Hajij3} :
\begin{eqnarray}
\label{greatness}
   \begin{minipage}[h]{0.15\linewidth}
        \vspace{0 pt}
        \scalebox{.4}{\includegraphics{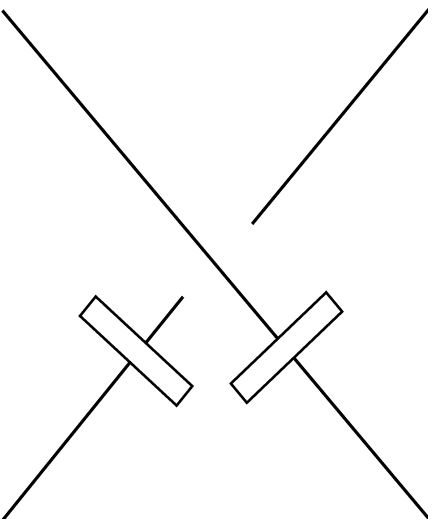}}
        \scriptsize{
         \put(-55,+52){$n$}
          \put(-1,+52){$n$}}

   \end{minipage}
   =
     \displaystyle\sum\limits_{i=0}^{n}C_{n,i}
  \begin{minipage}[h]{0.15\linewidth}
        \vspace{0pt}
        \scalebox{.4}{\includegraphics{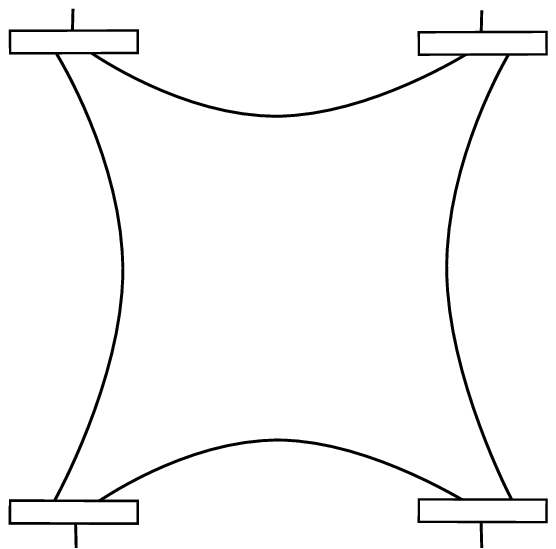}}
        \scriptsize{
        \put(-60,+64){$n$}
          \put(-5,+64){$n$}
          \put(-58,30){$i$}
          \put(-7,30){$i$}
         \put(-43,+56){$n-i$}
           \put(-43,+16){$n-i$}
           }
          \end{minipage}
  \end{eqnarray}
  where
  \begin{equation}
  \label{my fav}
  C_{n,i}=A^{n^2+2i^2-4in}\frac{(A^4,A^4)_n}{(A^4,A^4)_i(A^4,A^4)_{n-i}}.
  \end{equation}
We will also need the following identity from \cite{Masbaum}

\begin{eqnarray}
\label{masbum formula}
   \begin{minipage}[h]{0.15\linewidth}
        \vspace{0 pt}
        \scalebox{.4}{\includegraphics{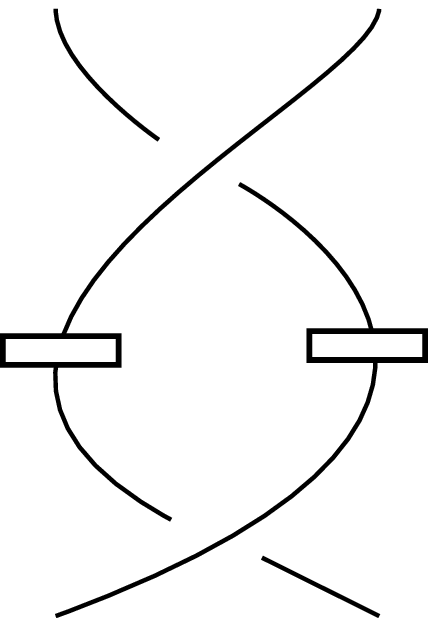}}
        \scriptsize{
         \put(-55,+52){$n$}
          \put(-1,+52){$n$}}

   \end{minipage}
   =
     \displaystyle\sum\limits_{i=0}^{n}D_{n,i}
  \begin{minipage}[h]{0.17\linewidth}
        \vspace{0pt}
        \scalebox{.4}{\includegraphics{goodbasis}}
        \scriptsize{
        \put(-60,+64){$n$}
          \put(-5,+64){$n$}
         \put(-33,+56){$i$}
           \put(-33,+16){$i$}
           }
          \end{minipage}
  \end{eqnarray}  
where

\begin{equation*}
D_{n,i}=A^{2 i^2 - 4 i n + 2 n^2}\frac{(A^4,A^4)_n}{(A^4,A^4)_i(A^4,A^4)_{n-i}} \prod_{j=n-i+1}^n (1-A^{-4j}) 
\end{equation*} 

\begin{example}\label{example1}
We compute the invariant $[G]_{2n}$ for the graph given in the following Figure \ref{Example1}.
\begin{figure}[H]
\label{IV}

   {\includegraphics[scale=0.2]{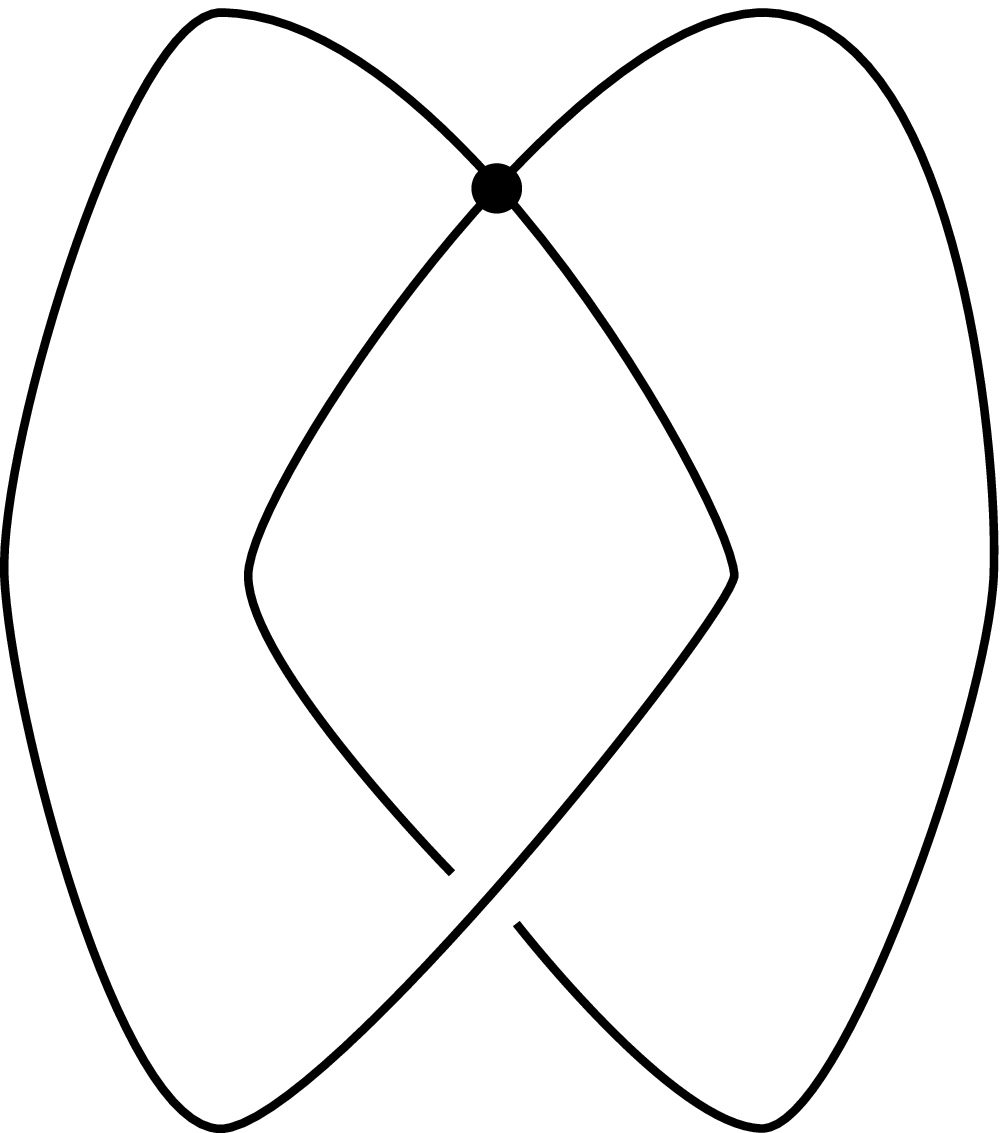}
     \caption{}
  \label{Example1}}
\end{figure}
Lemma \ref{technical lemma} implies that:
\begin{eqnarray*}
   \begin{minipage}[h]{0.12\linewidth}
        \vspace{0pt}
        \scalebox{0.08}{\includegraphics{singular_map}}
        \tiny{
        \put(-1,45){$2n$}
        \put(-48,45){$2n$}
        \put(-28,40){$n$}
        \put(-28,5){$n$}
        \put(-12,20){$n$}
        \put(-40,20){$n$}}
   \end{minipage}\times 
   \begin{minipage}[h]{0.1\linewidth}
        \vspace{0pt}
        \scalebox{0.08}{\includegraphics{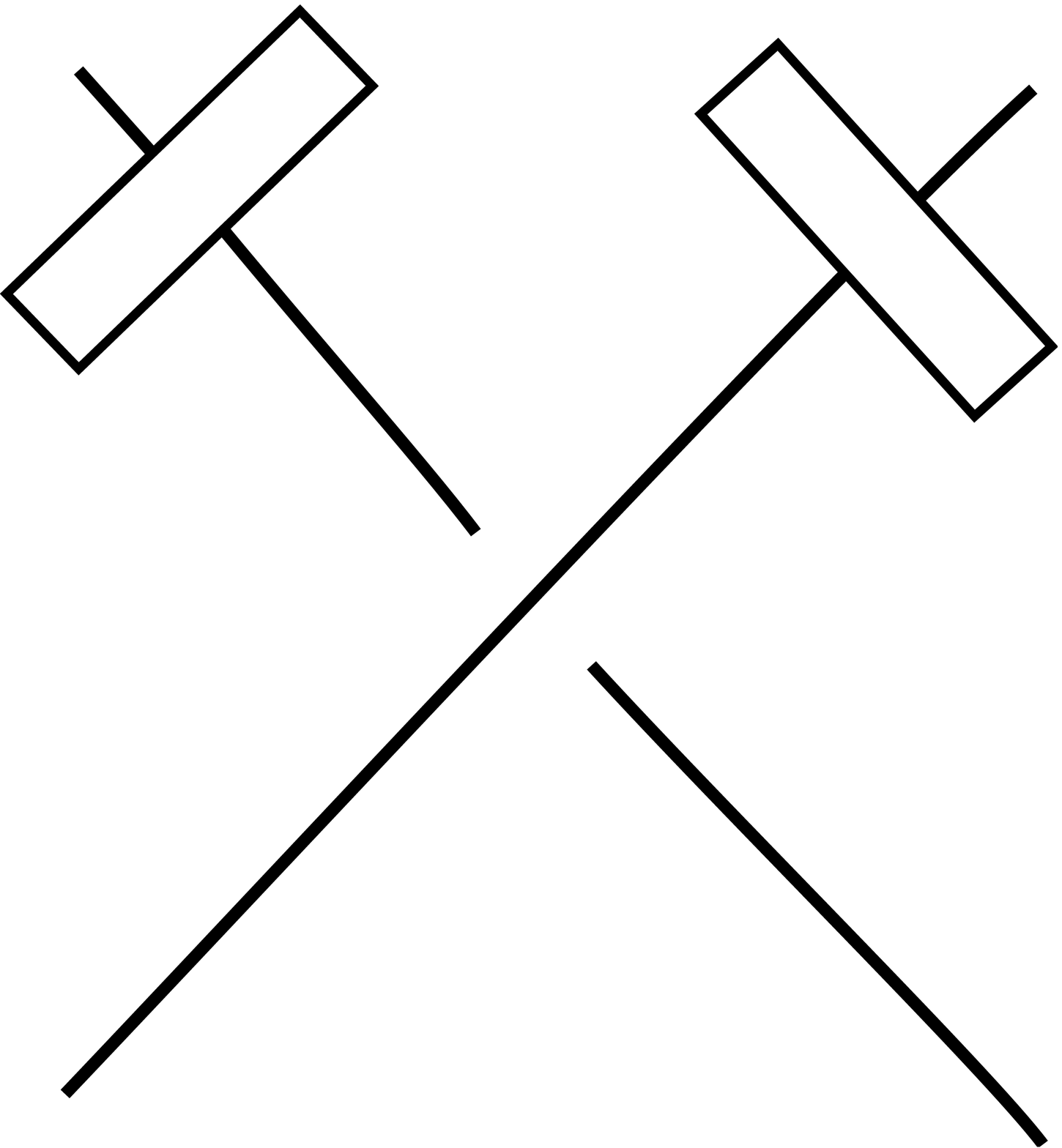}}
        \tiny{
        \put(-1,-3){$2n$}
        \put(-48,-3){$2n$}}
   \end{minipage}
   = (-1)^{-n} A^{-3n^2-2n}\quad\begin{minipage}[h]{0.12\linewidth}
        \vspace{18pt}
        \scalebox{.13}{\includegraphics{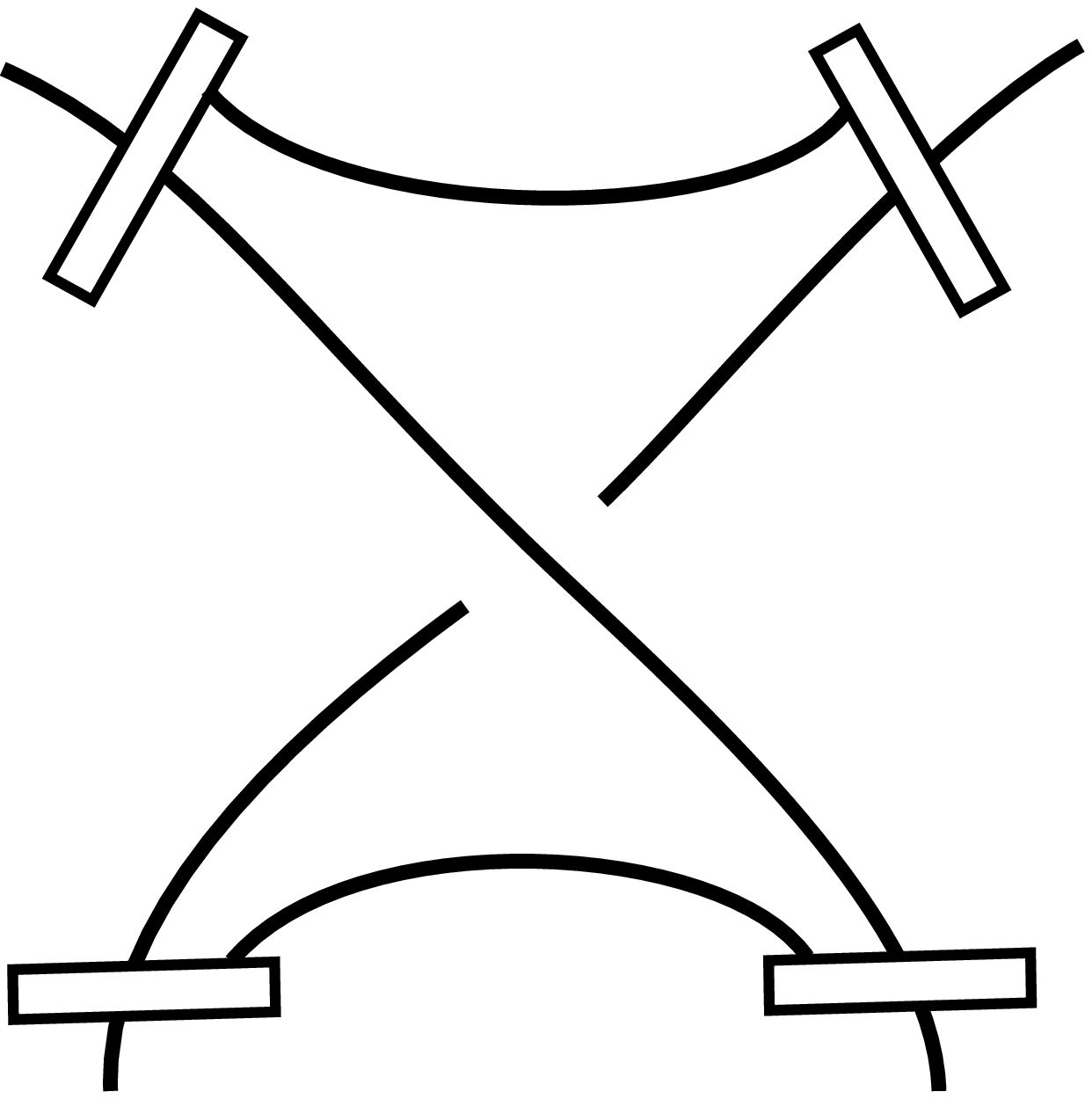}}
        \scriptsize{
          \put(-56,45){$2n$}
          \put(-5,45){$2n$}
           }
          \end{minipage}
  \end{eqnarray*}

Hence we obtain,  
\begin{eqnarray*}
   \begin{minipage}[h]{0.12\linewidth}
        \vspace{0pt}
        \scalebox{0.08}{\includegraphics{singular_map}}
        \tiny{
        \put(-1,45){$2n$}
        \put(-48,45){$2n$}
        \put(-28,40){$n$}
        \put(-28,5){$n$}
        \put(-12,20){$n$}
        \put(-40,20){$n$}}
   \end{minipage}\times 
   \begin{minipage}[h]{0.1\linewidth}
        \vspace{0pt}
        \scalebox{0.08}{\includegraphics{colored_corssing_2}}
        \tiny{
        \put(-1,-3){$2n$}
        \put(-48,-3){$2n$}}
   \end{minipage}
   =(-1)^{-n} A^{-3n^2-2n}\sum_{i=0}^n C_{n,i} \begin{minipage}[h]{0.13\linewidth}
        \vspace{0pt}
        \scalebox{0.08}{\includegraphics{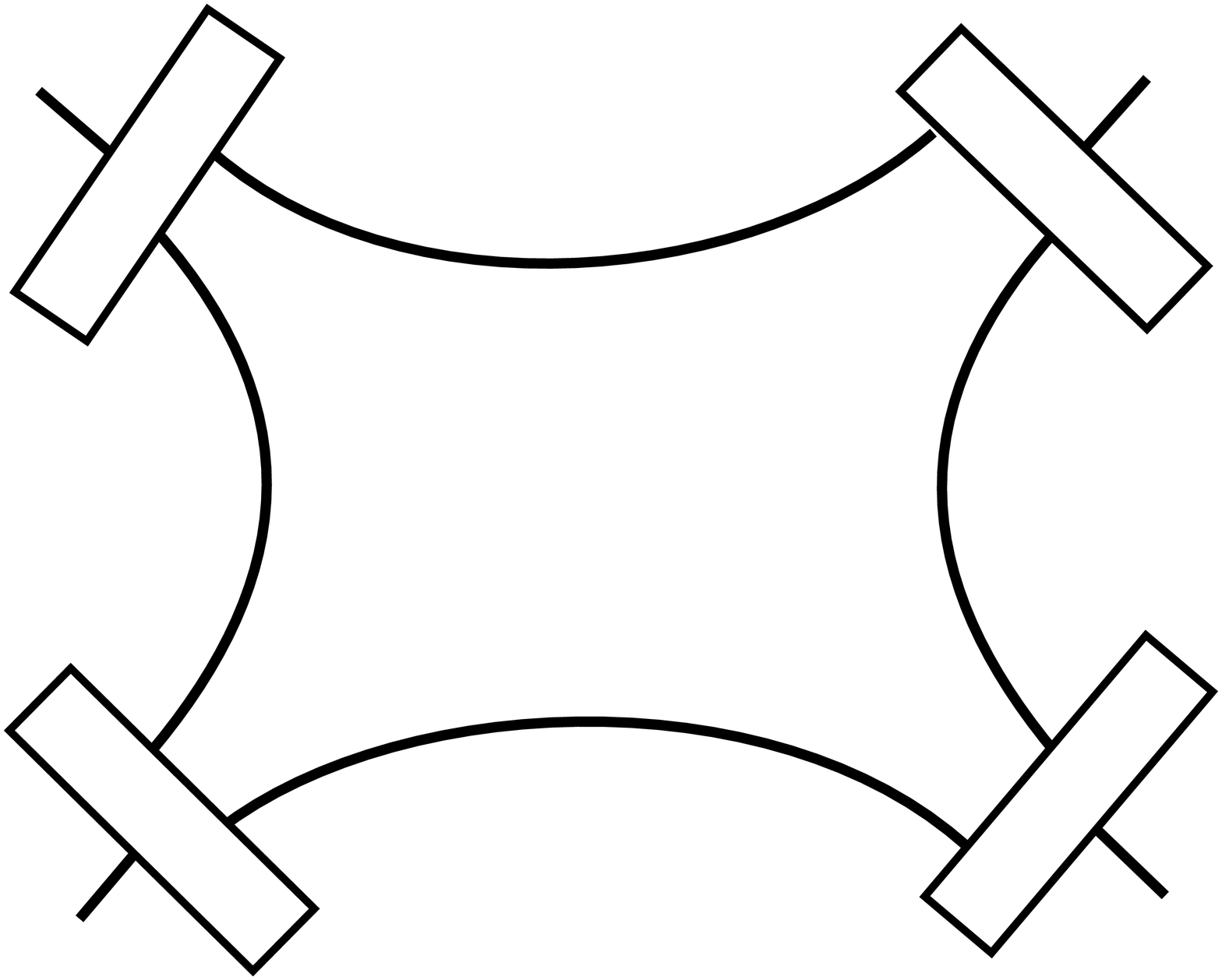}}
        \tiny{
        \put(-1,47){$2n$}
        \put(-63,47){$2n$}
        \put(-42,40){$2n-i$}
        \put(-42,5){$2n-i$}
        \put(-10,20){$i$}
        \put(-55,20){$i$}}
   \end{minipage}
  \end{eqnarray*}
We then conclude that,
 \begin{eqnarray*}
  \left[
   \begin{minipage}[h]{0.1\linewidth}
        \vspace{0pt}
        \scalebox{0.15}{\includegraphics{singular_knot_example_1}}
   \end{minipage}
  \right]_{2n}  &=&(-1)^{-n} A^{-3n^2-2n}\sum_{i=0}^n C_{n,i} \begin{minipage}[h]{0.1\linewidth}
        \vspace{0pt}
        \scalebox{0.15}{\includegraphics{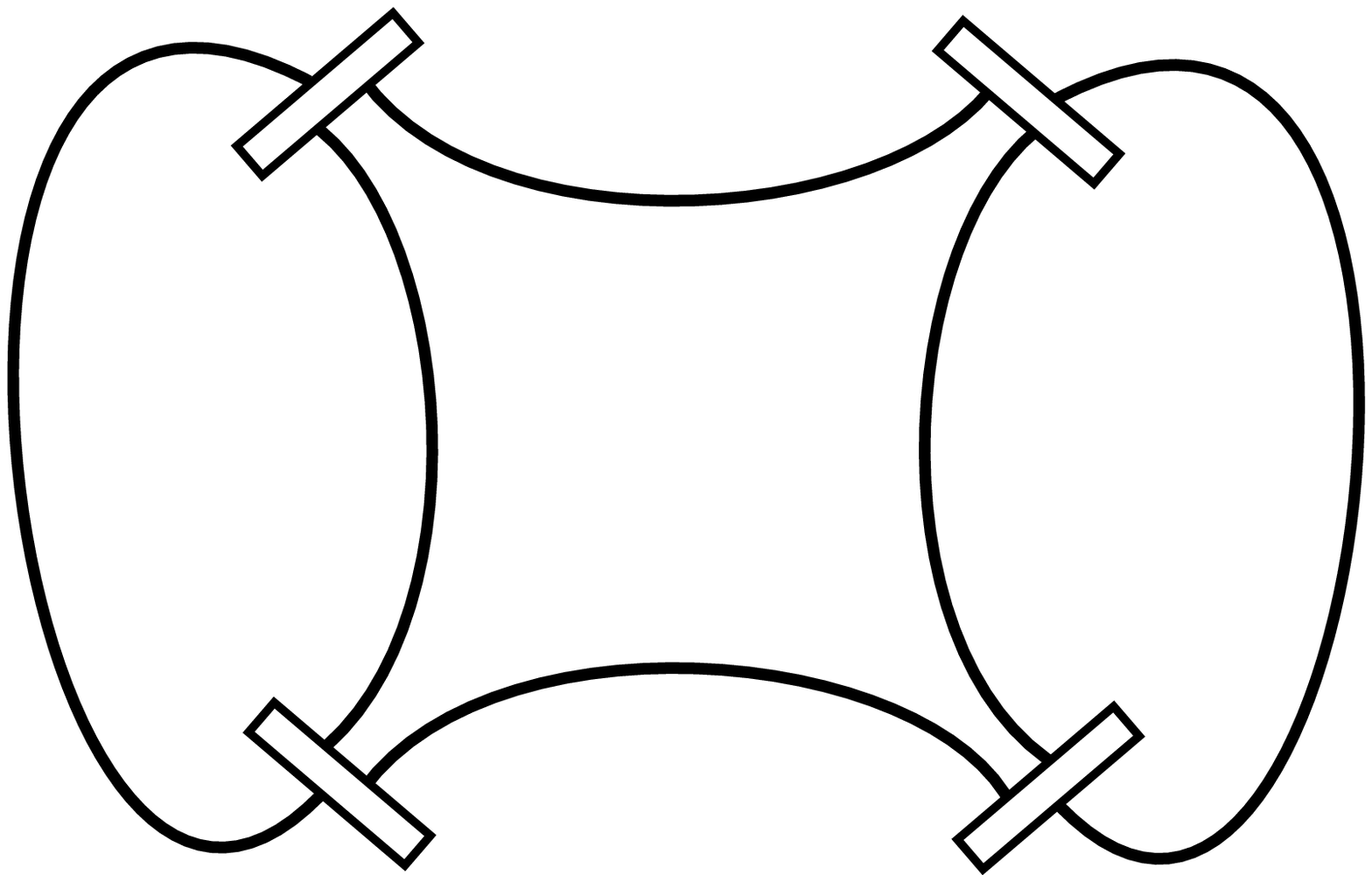}}
        \tiny{
        \put(-1,40){$2n$}
        \put(-46,28){$2n-i$}}
   \end{minipage}\\&=&(-1)^{-n} A^{-3n^2-2n}\sum_{i=0}^n C_{n,i} \frac{(\Delta_{2n})^2}{\Delta_{2n-i}}.
 \end{eqnarray*}
\end{example}

\begin{example}
We compute our invariant for the graph given in Figure \ref{Example2}.
\begin{figure}[H]
\label{IV}

   {\includegraphics[scale=0.2]{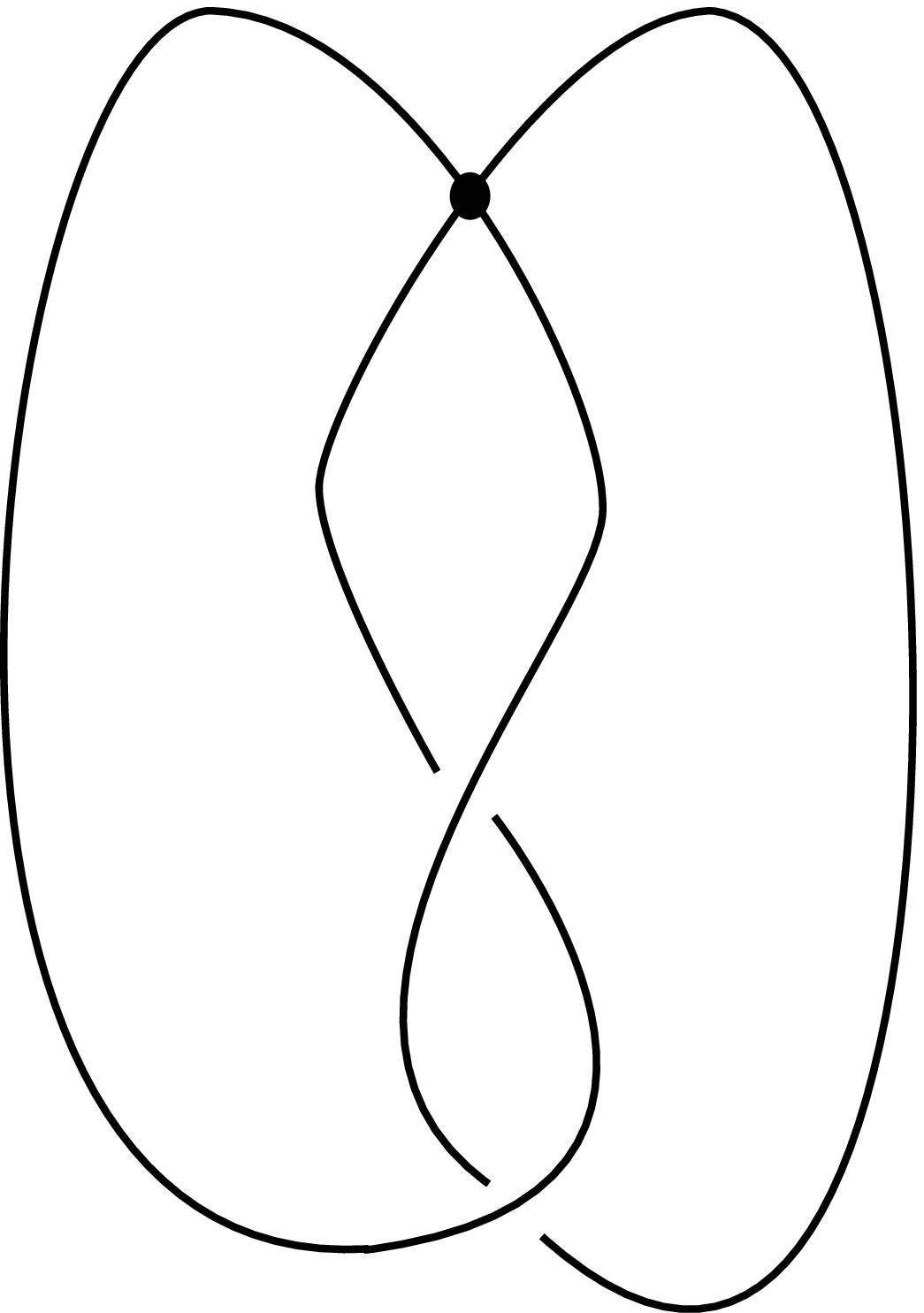}
     \caption{}
  \label{Example2}}
\end{figure}
Using Lemma \ref{technical lemma} we obtain:
\begin{eqnarray*}
   \begin{minipage}[h]{0.12\linewidth}
        \vspace{0pt}
        \scalebox{0.08}{\includegraphics{singular_map}}
        \tiny{
        \put(-1,45){$2n$}
        \put(-48,45){$2n$}
        \put(-28,40){$n$}
        \put(-28,5){$n$}
        \put(-12,20){$n$}
        \put(-40,20){$n$}}
   \end{minipage}\times 
   \begin{minipage}[h]{0.1\linewidth}
        \vspace{0pt}
        \scalebox{0.4}{\includegraphics{double_crossing}}
        \tiny{
        \put(-1,-3){$2n$}
        \put(-48,-3){$2n$}}
   \end{minipage}
   = (-1)^{-n} A^{-3n^2-2n}\quad\begin{minipage}[h]{0.12\linewidth}
        \vspace{18pt}
        \scalebox{.13}{\includegraphics{product_result.eps}}
        \scriptsize{
          \put(-56,45){$2n$}
          \put(-5,45){$2n$}
           }
          \end{minipage}\times 
   \begin{minipage}[h]{0.1\linewidth}
        \vspace{0pt}
        \scalebox{0.08}{\includegraphics{colored_corssing_2}}
        \tiny{
        \put(-1,-3){$2n$}
        \put(-48,-3){$2n$}}
   \end{minipage}
  \end{eqnarray*}

Hence,  
\begin{eqnarray*}
   \begin{minipage}[h]{0.12\linewidth}
        \vspace{0pt}
        \scalebox{0.08}{\includegraphics{singular_map}}
        \tiny{
        \put(-1,45){$2n$}
        \put(-48,45){$2n$}
        \put(-28,40){$n$}
        \put(-28,5){$n$}
        \put(-12,20){$n$}
        \put(-40,20){$n$}}
   \end{minipage}\times 
   \begin{minipage}[h]{0.1\linewidth}
        \vspace{0pt}
        \scalebox{0.4}{\includegraphics{double_crossing}}
        \tiny{
        \put(-1,-3){$2n$}
        \put(-48,-3){$2n$}}
   \end{minipage}
   =A^{-6n^2-4n}\sum_{i=0}^n D_{n,i} \begin{minipage}[h]{0.13\linewidth}
        \vspace{0pt}
        \scalebox{0.08}{\includegraphics{singular_map_2}}
        \tiny{
        \put(-1,47){$2n$}
        \put(-63,47){$2n$}
        \put(-42,40){$n+i$}
        \put(-42,5){$n+i$}
        \put(-10,20){$n-i$}
        \put(-55,20){$n-i$}}
   \end{minipage}
  \end{eqnarray*}
Thus,
 \begin{eqnarray*}
  \left[
   \begin{minipage}[h]{0.1\linewidth}
        \vspace{0pt}
        \scalebox{0.15}{\includegraphics{singular_knot_example_2}}
   \end{minipage}
  \right]_{2n}  &=& A^{-6n^2-4n}\sum_{i=0}^n D_{n,i} \frac{(\Delta_{2n})^2}{\Delta_{n+i}}.
 \end{eqnarray*}
 \end{example}
 
\begin{example}\textit{ Connected sums.} Let $K$ and $K^{\prime}$ be oriented knots. We claim that $[K]_{2n}[K^{\prime}]_{2n}\Delta_{2n}=[K\#K{^{\prime}}]_{2n}$ where $K \#K^{\prime}$ is the connected sum of $K$ and $K^{\prime}$. Using the basic properties of the Jones Wenzl idempotent, we can write $[K]_{2n}=R_1(A)\Delta_{2n}$ and  $[K^{\prime}]_{2n}=R_2(A)\Delta_{2n}$ where $R_1(A)$ and $R_2(A)$ are rational functions. Similarly, the skein element on the bottom of Figure \ref{connected sum} is equal to $[K]_{2n}[K^{\prime}]_{2n}\Delta_{2n}$.

\begin{figure}[H]
\label{connected sum}
   \centering
    \scalebox{0.27} {\includegraphics{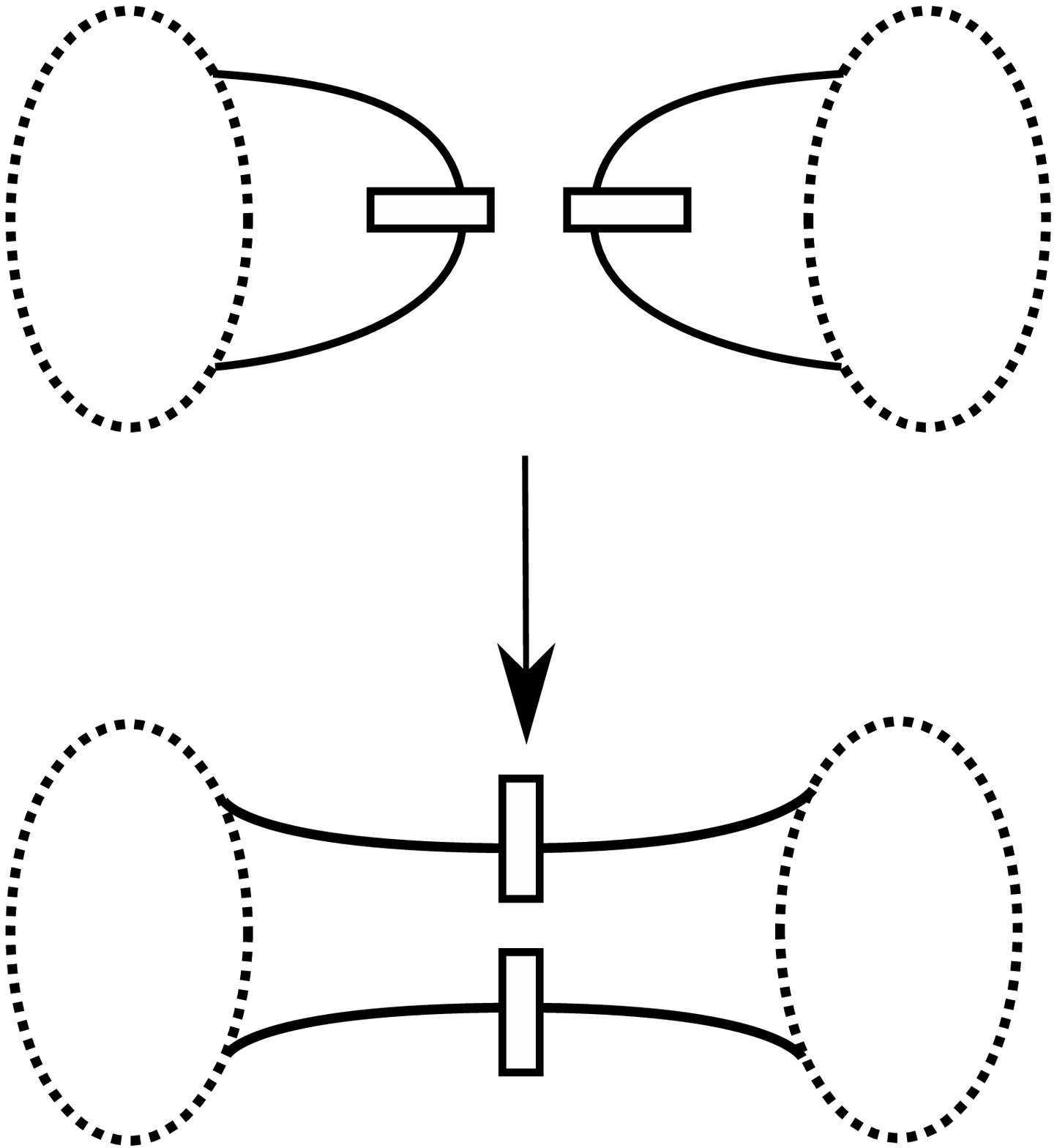}}
    \put(-20,108){$K^{\prime}$}
    \put(-110,108){$K$}
    \put(-25,22){$K^{\prime}$}
    \put(-110,22){$K$}
      \put(-47,121){$2n$}
         \put(-85,121){$2n$}
            \put(-76,38){$2n$}
        \put(-76,4){$2n$} 
   \caption{}
\end{figure}

\end{example}

\section{Singular Braid Monoid Representations}
\label{sec5}
The singular braid monoid was introduced in \cite{Baez, Birman} as a singularization of the braid group and in relation to perturbative Chern-Simons theory. In this section we use the invariant that we defined in the previous sections to give representations of the singular braid monoid. We start with the algebraic definition of the singular braid monoid  \cite{Baez,Birman}.
\begin{definition}
The singular braid monoid $SB_n$ on $n$ strands is the monoid generated by 
\begin{equation}
\sigma_1,...,\sigma_{n-1},\sigma_1^{-1},...,\sigma_{n-1}^{-1}, \tau_1,...,\tau_{n-1}
\end{equation}
subject to the relations
\begin{enumerate}
\item For all $1\leq  i < n$:   $\sigma_i \sigma_i^{-1} = e=\sigma_i^{-1} \sigma_i $.
\item For $|i-j|>1$:
\begin{enumerate}
\item $\sigma_i \sigma_j= \sigma _j \sigma _i$. 
\item $\sigma_i \tau_j= \tau _j \sigma _i$. 
\item $\tau_i \tau_j= \tau_j \tau_i$. 
\end{enumerate}
\item For all $1\leq  i < n:\tau_i \sigma_i=\sigma_i \tau_i$.
\item For all $i<n-1$: 
\begin{enumerate}
\item $\sigma_i \sigma_{i+1} \sigma_i = \sigma_{i+1}\sigma_i \sigma_{i+1}$
\item $\tau_i \sigma_{i+1} \sigma_i = \sigma_{i+1}\sigma_i \tau_{i+1}$
\item $\tau_{i+1} \sigma_{i} \sigma_{i+1} = \sigma_{i}\sigma_{i+1} \tau_{i}$
\end{enumerate}
\end{enumerate}
\end{definition}
 Now we will consider a sequence of representations of the monoid $SB_n$ into the colored Temperley-Lieb algebra $TL_{n}^{2m}$.

\begin{theorem}
For every integers $m,n \geq 1$, the map $\hat{\rho}_{m,n}$ given on the generators $\sigma_i$ and $\tau_i$ in the diagrammatic below gives a representation of $SB_n$ into $TL_{n}^{2m}$.
\begin{eqnarray}
\label{rep}
 \hspace{-299 pt}\begin{minipage}[h]{0.0\linewidth}
        \scalebox{0.12}{\includegraphics{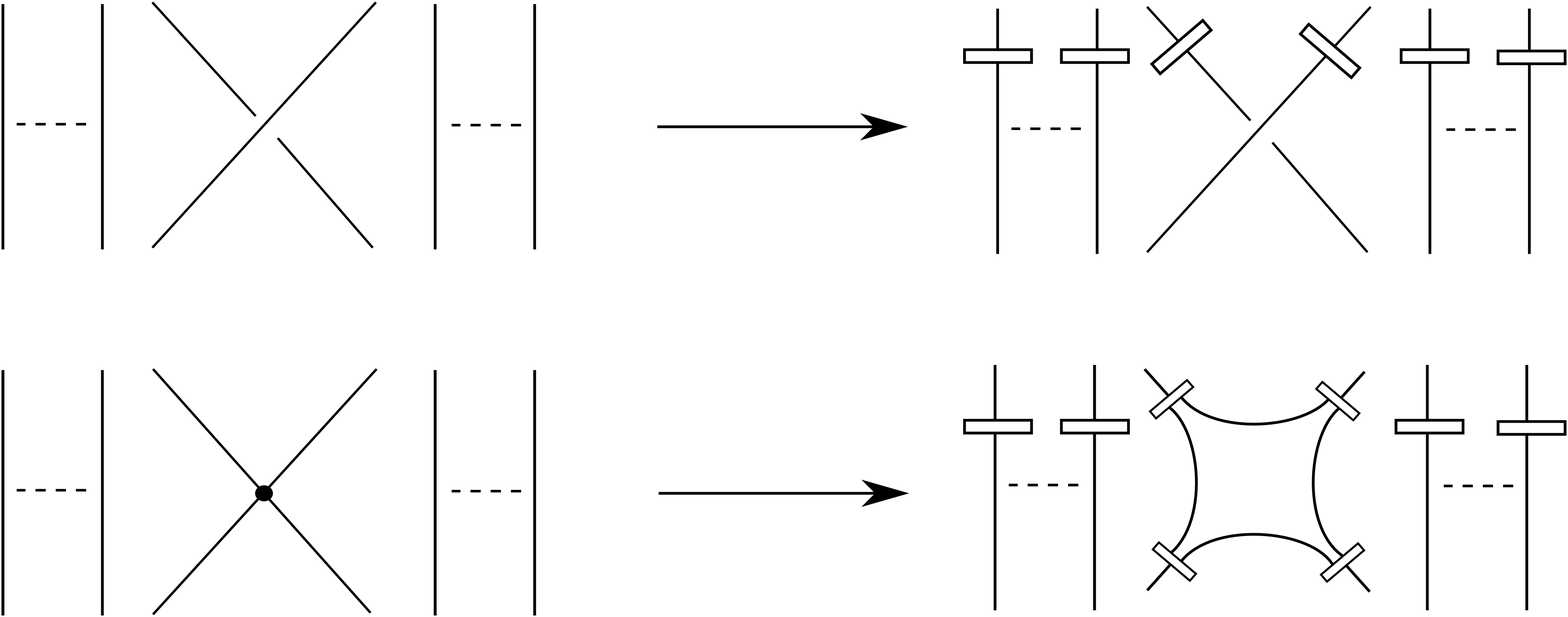}}
         \put(-180,120){$ \hat{\rho}_{m,n}$}
         \put(-180,35){$ \hat{\rho}_{m,n}$}
         \put(-360,100){$ \sigma_i=$}
         \put(-360,25){$ \tau_i=$}
         \tiny{
         \put(-134,50){$2m$}
         \put(-113,50){$2m$}
         \put(-85,53){$2m$}
         \put(-55,53){$2m$}
         \put(-28,50){$2m$}
         \put(-7,50){$2m$}%
         \put(-134,130){$2m$}
         \put(-113,130){$2m$}
         \put(-85,133){$2m$}
         \put(-55,133){$2m$}
         \put(-28,130){$2m$}
         \put(-7,130){$2m$}     
         }
   \end{minipage}  
\end{eqnarray} 
\end{theorem}

\begin{proof}
Using Theorem \ref{main}, it is straightforward to see that the images by $\hat{\rho}_{m,n}$ of the relations of the singular braid hold in $TL_{n}^m$ giving a representation of $SB_n$ into $TL_{n}^m$.
\end{proof}	
Note that the restriction of the map $\hat{\rho}_{m,n}$ to $B_n$ is the map ${\rho}_{m,n}$ given in section \ref{sec2}.  
\section{Integrality of the Invariant and Open Questions}
The invariant $[.]_{2n}$ takes values in $\mathbb{Q}(A)$. However, our computations show that it can be made into an element in $\mathbb{Z}[A,A^{-1}]$ by multiplying by a certain Laurent polynomial. More precisely, let $L$ be a singular link with $k$ singular crossings, then we conjecture that multiplying $C_{2n,n}^k$ with $[L]_{2n}$ makes $C_{2n,n}^k[L]_{2n}$ an element of $\mathbb{Z}[A,A^{-1}]$ where $C_{2n,n}$ is defined in (\ref{my fav}). Now we give an illustration that this conjecture cannot be proven using a local argument. To show this, suppose that $L$ is a singular link with only one singular crossing. We use identity~\ref{identity4} from Definition \ref{defG2} and the definition of the Jones-Wenzl idempotent to expand the singular crossing in $L$ as follows:    

\begin{eqnarray*}
	\begin{minipage}[h]{0.1\linewidth}
		\vspace{0pt}
		\scalebox{0.08}{\includegraphics{singular_map}}
		\tiny{
			\put(-1,45){$2$}
			\put(-48,45){$2$}
			\put(-28,40){$1$}
			\put(-28,5){$1$}
			\put(-12,20){$1$}
			\put(-40,20){$1$}}
	\end{minipage}
	&=& \begin{minipage}[h]{0.13\linewidth}
		\vspace{0pt}
		\scalebox{0.2}{\includegraphics{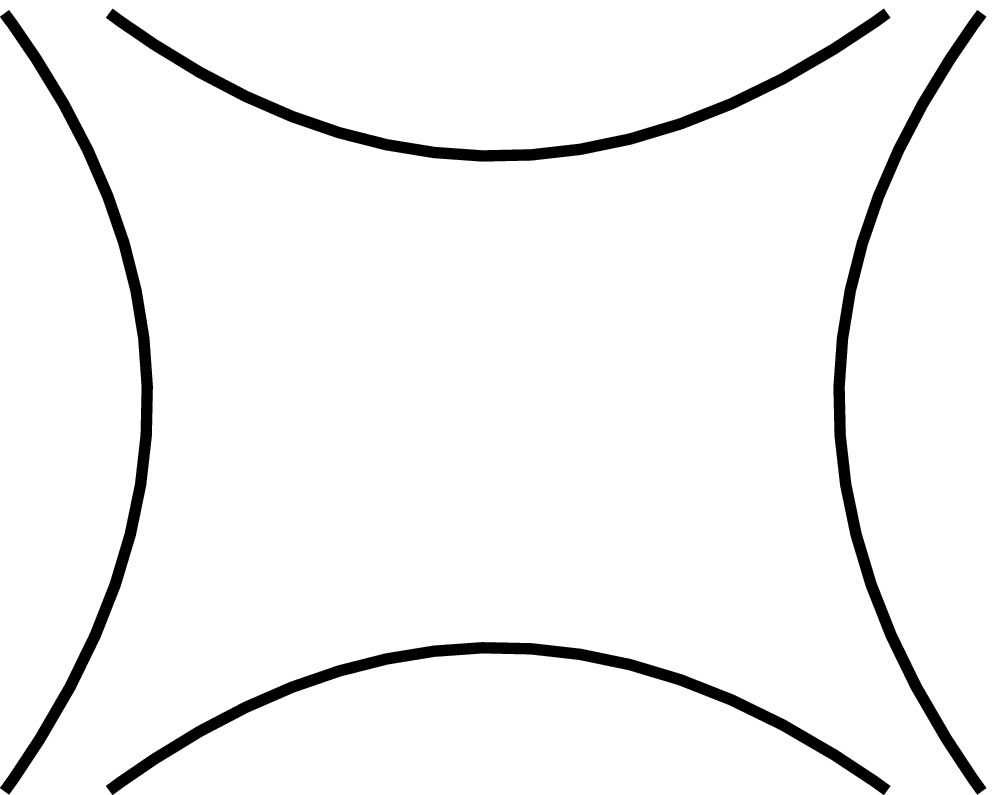}}
	\end{minipage}-\frac{1}{d} \Bigg(\quad\begin{minipage}[h]{0.14\linewidth}
	\vspace{0pt}
	\scalebox{0.2}{\includegraphics{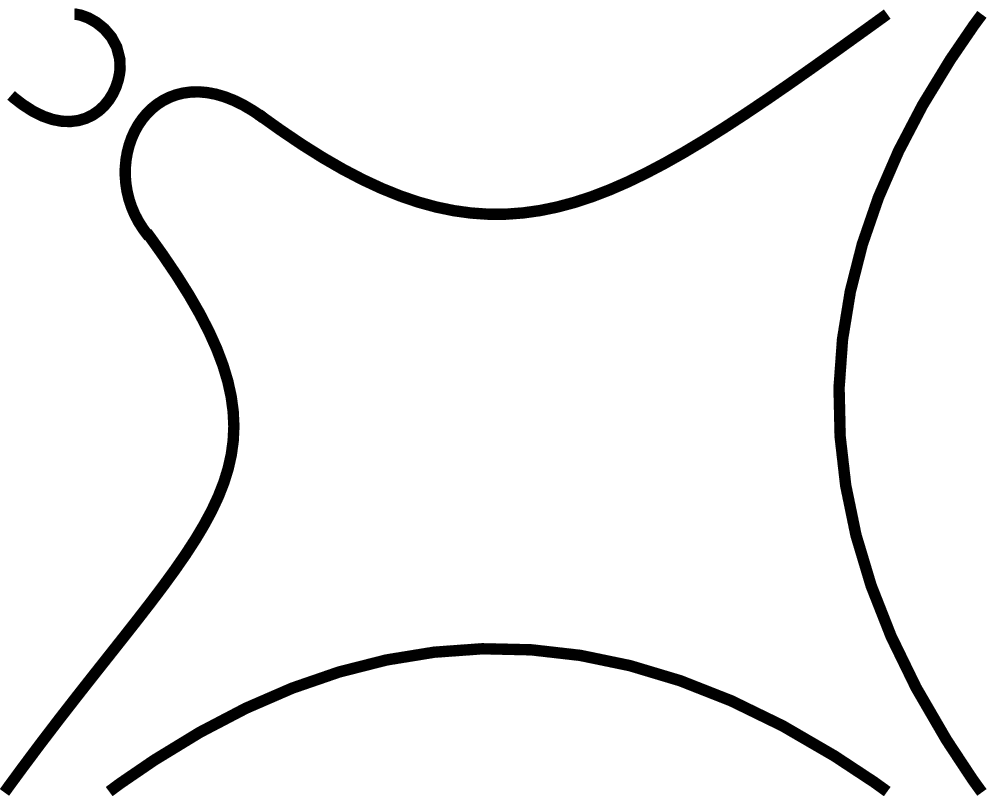}}
\end{minipage}+\begin{minipage}[h]{0.14\linewidth}
\vspace{0pt}
\scalebox{0.2}{\includegraphics{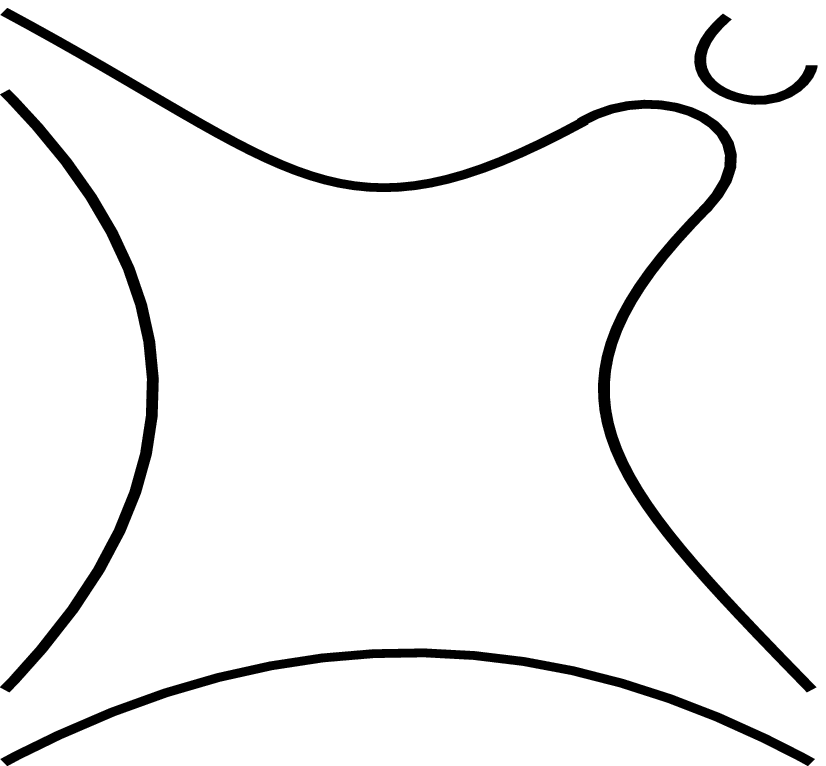}}
\end{minipage}+\begin{minipage}[h]{0.13\linewidth}
\vspace{0pt}
\scalebox{0.2}{\includegraphics{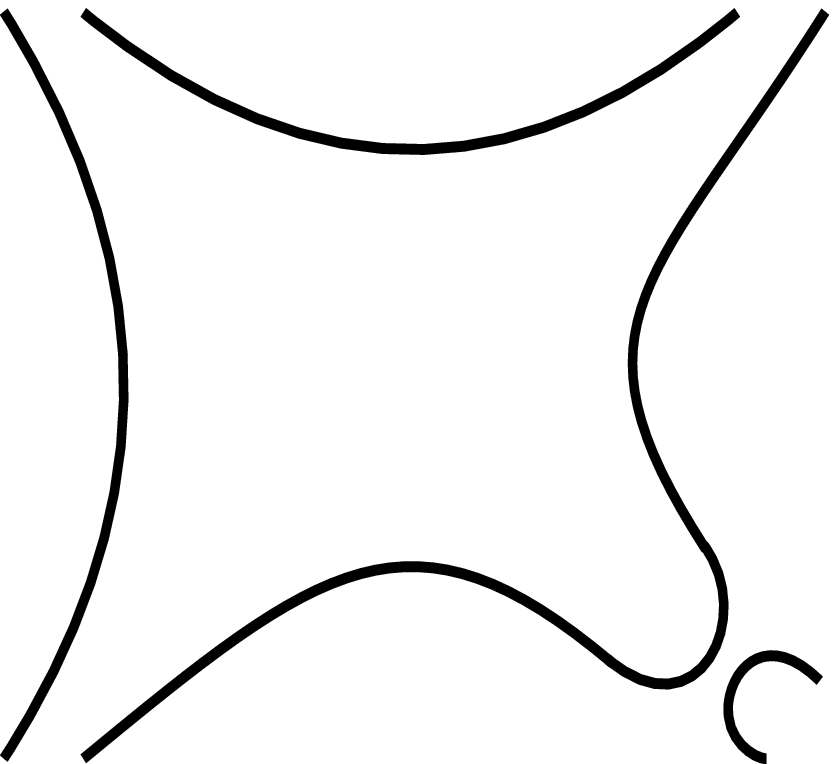}}
\end{minipage}+\begin{minipage}[h]{0.11\linewidth}
\vspace{0pt}
\scalebox{0.2}{\includegraphics{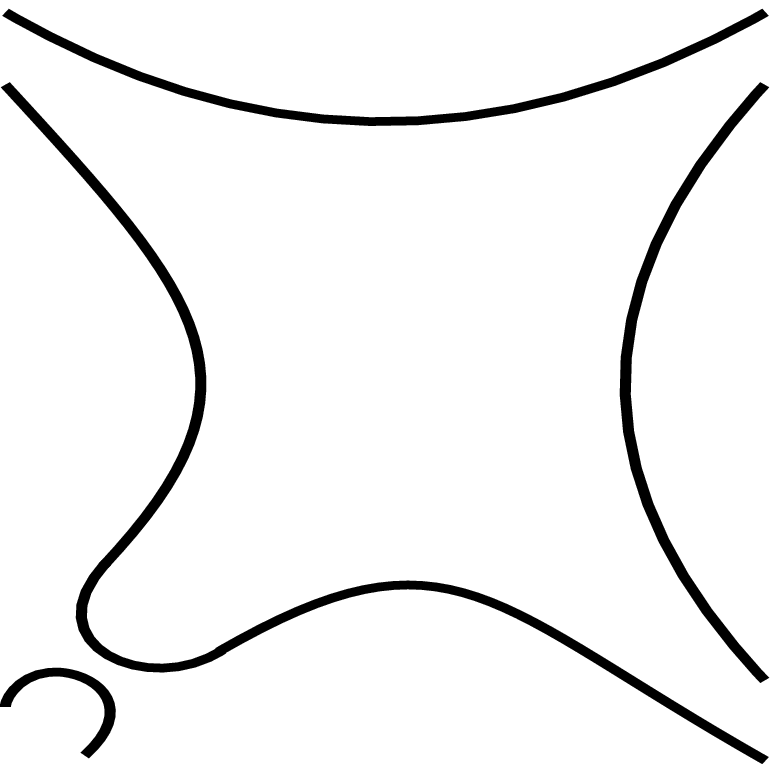}}
\end{minipage}\Bigg)\\&+&\frac{1}{d^2} \Bigg(\quad \begin{minipage}[h]{0.13\linewidth}
\vspace{0pt}
\scalebox{0.2}{\includegraphics{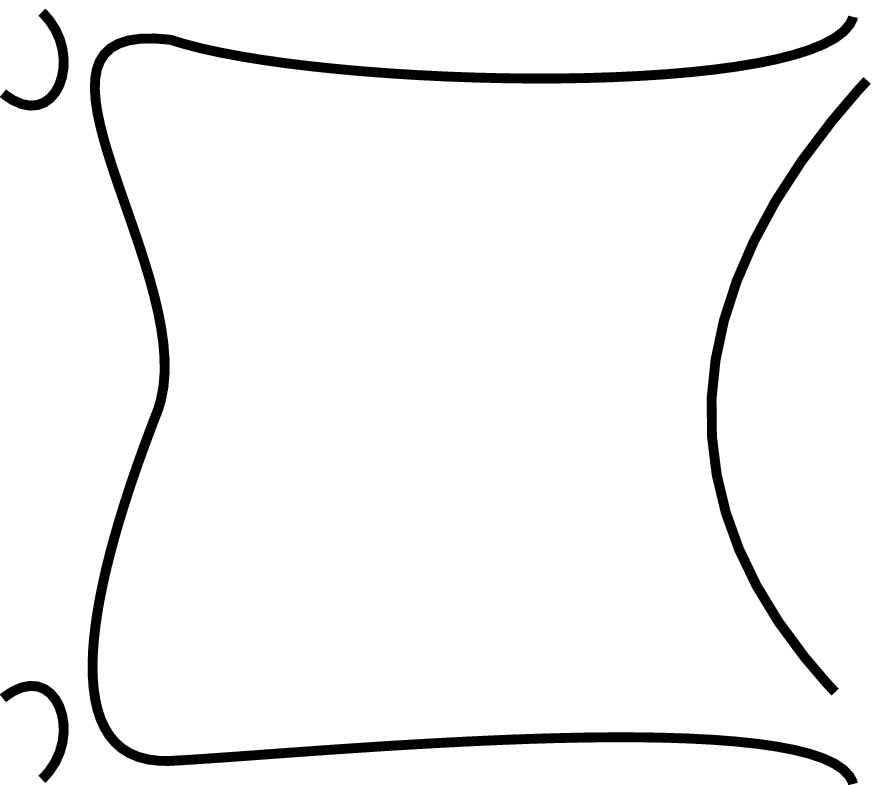}}
\end{minipage}+\begin{minipage}[h]{0.14\linewidth}
\vspace{0pt}
\scalebox{0.2}{\includegraphics{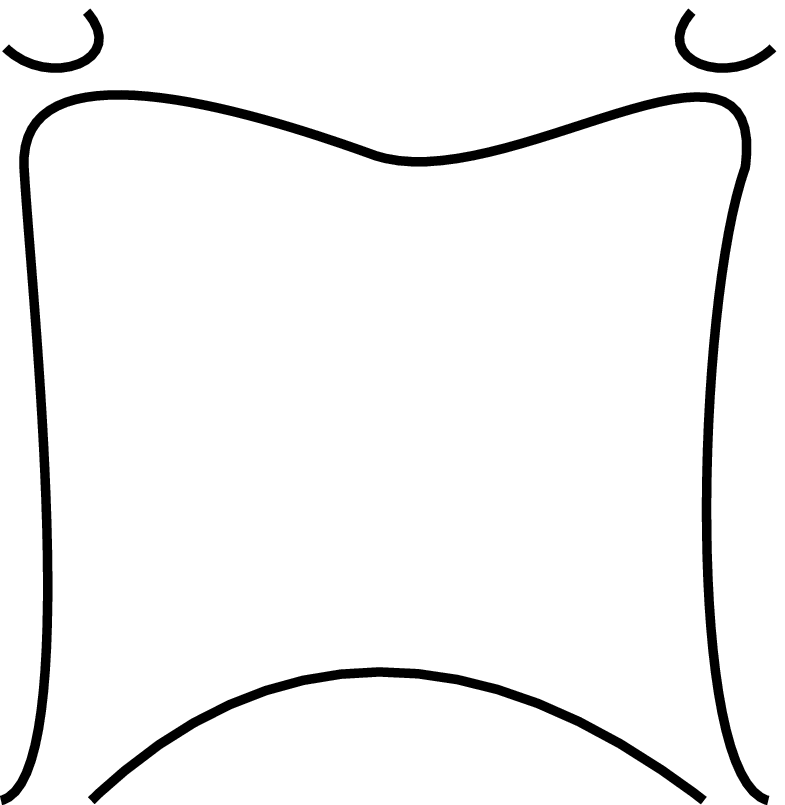}}
\end{minipage}+\begin{minipage}[h]{0.14\linewidth}
\vspace{0pt}
\scalebox{0.2}{\includegraphics{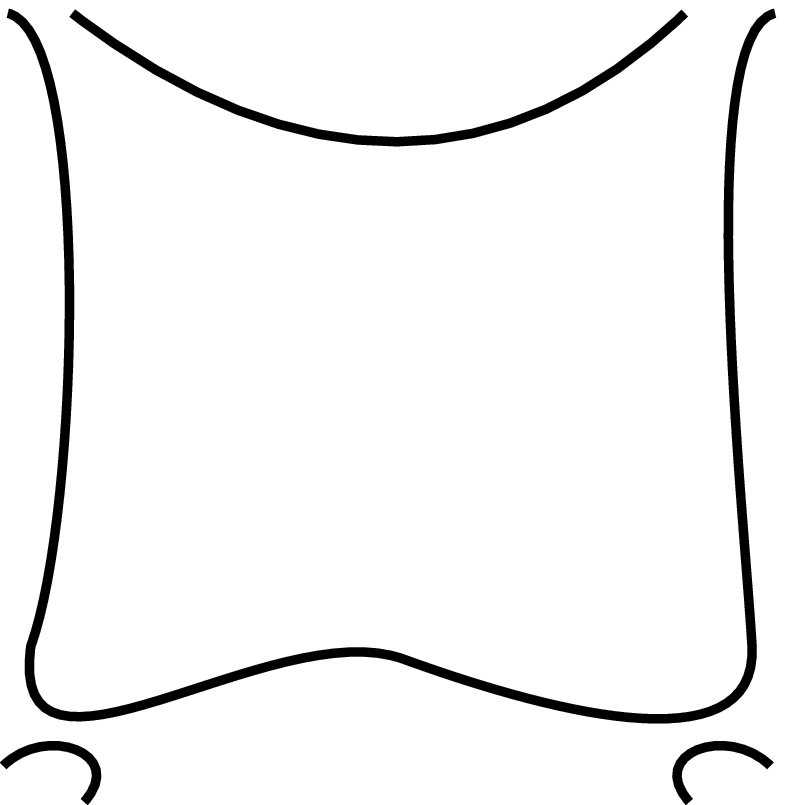}}
\end{minipage}+\begin{minipage}[h]{0.11\linewidth}
\vspace{0pt}
\scalebox{0.2}{\includegraphics{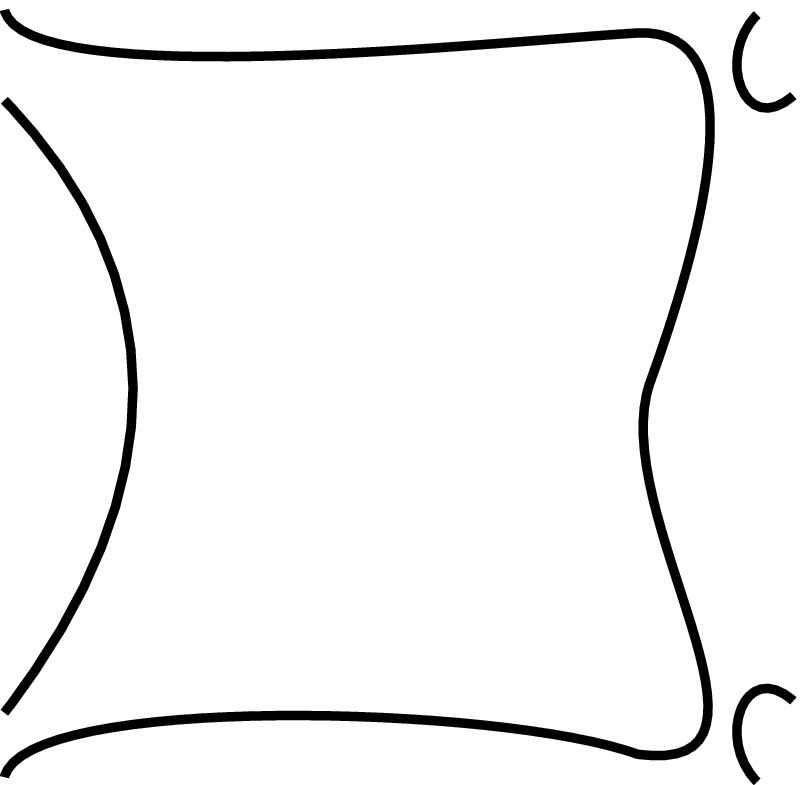}}
\end{minipage}\Bigg)\\&-&\frac{1}{d^3} \Bigg(\quad\begin{minipage}[h]{0.13\linewidth}
\vspace{0pt}
\scalebox{0.2}{\includegraphics{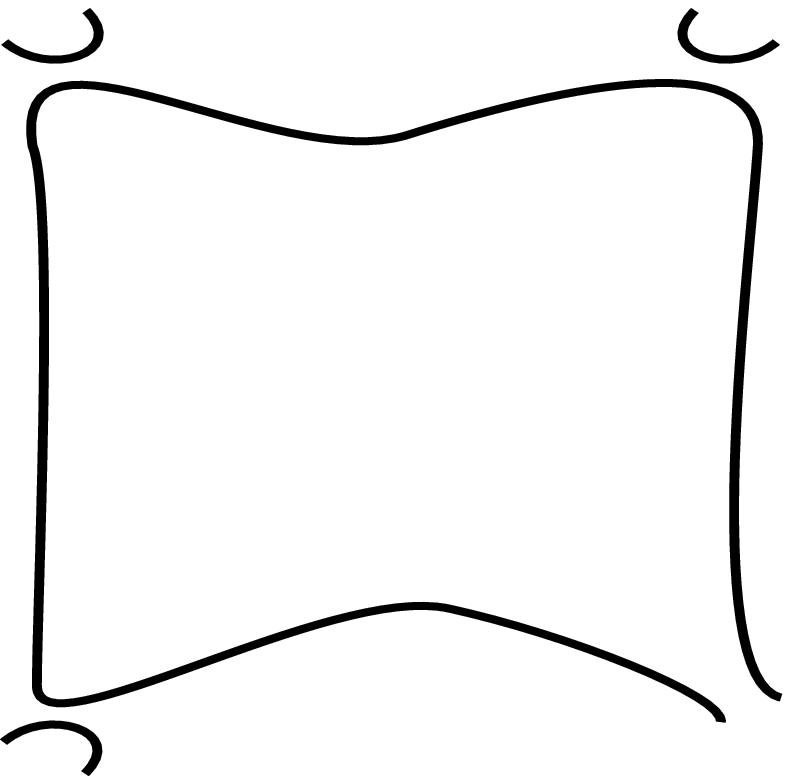}}
\end{minipage}+\begin{minipage}[h]{0.14\linewidth}
\vspace{0pt}
\scalebox{0.2}{\includegraphics{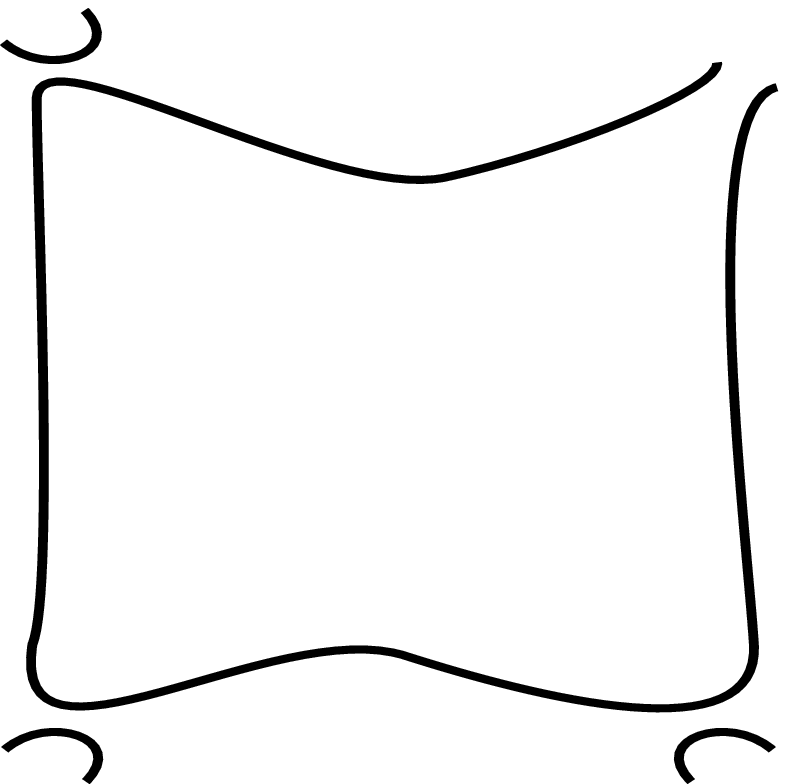}}
\end{minipage}+\begin{minipage}[h]{0.14\linewidth}
\vspace{0pt}
\scalebox{0.2}{\includegraphics{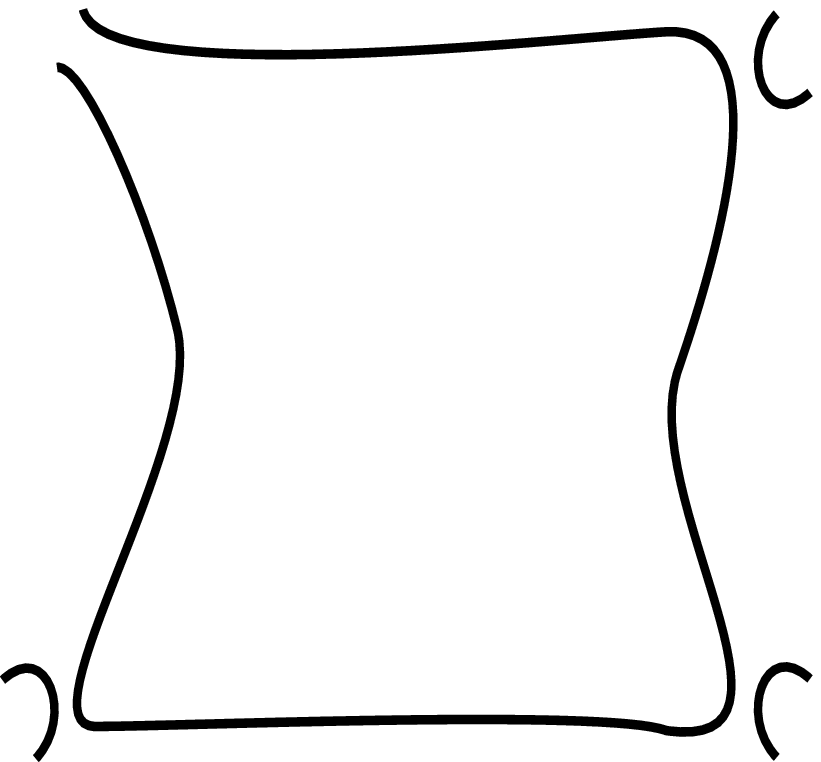}}
\end{minipage}+\begin{minipage}[h]{0.11\linewidth}
\vspace{0pt}
\scalebox{0.2}{\includegraphics{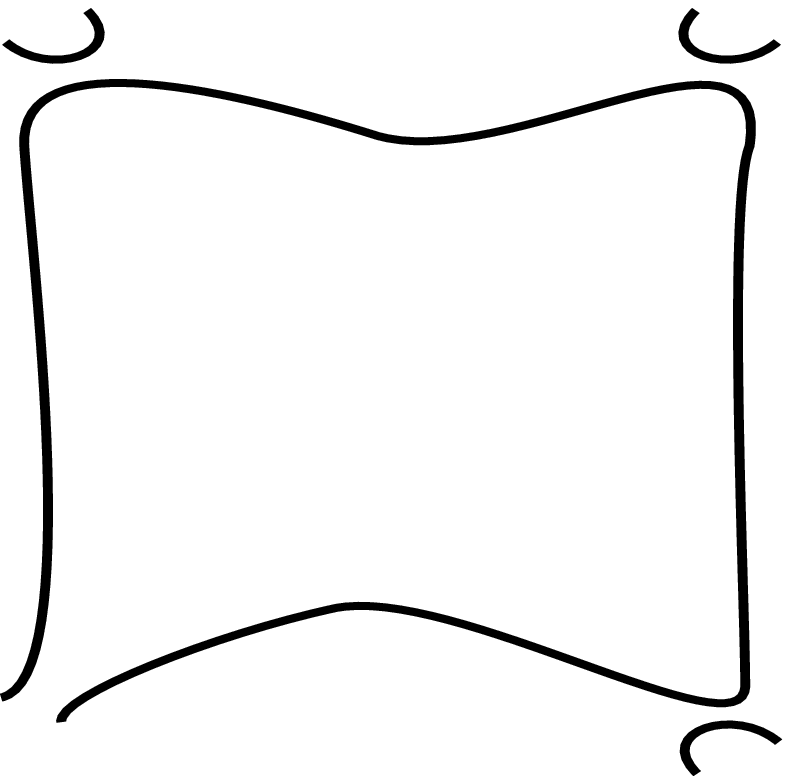}}
\end{minipage}\Bigg)\\&+&\frac{1}{d^3} \quad \begin{minipage}[h]{0.14\linewidth}
\vspace{0pt}
\scalebox{0.2}{\includegraphics{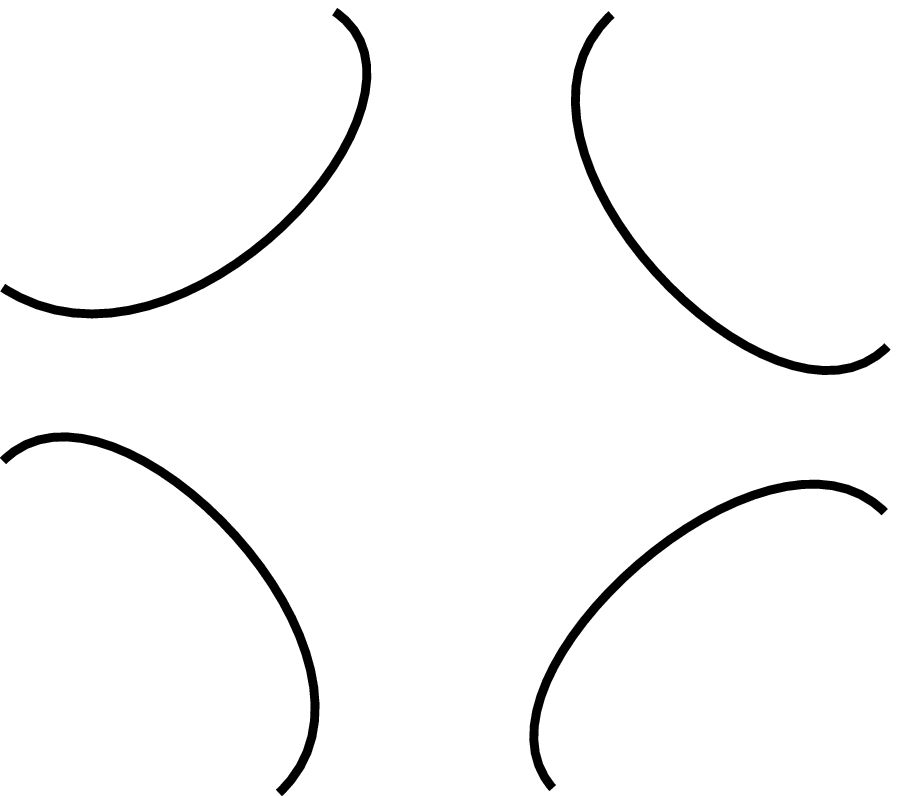}},
\end{minipage}
\end{eqnarray*}

where $d=-A^2-A^{-2}$. One can see that $\frac{C_{2,1}}{d^2} \notin \mathbb{Z}[A,A{-1}]$ . Hence the poles that occur in $[L]_2$ cannot be removed using this simple local argument.  This conjecture is in fact true for $n=1$ as can be seen from equation \ref{n=1}:

\begin{equation*}C_{2,1} \quad 
\begin{minipage}[h]{0.12\linewidth}
\vspace{0pt}
\scalebox{0.08}{\includegraphics{singular_map}}
\tiny{
	\put(-28,40){$1$}
	\put(-28,5){$1$}
	\put(-12,20){$1$}
	\put(-40,20){$1$}
	\put(-1,45){$2$}
	\put(-48,45){$2$}
	\put(-1,0){$2$}
	\put(-48,0){$2$}}
\end{minipage}=(A^2+A^{-2})  
\begin{minipage}[h]{0.12\linewidth}
\vspace{0pt}
\scalebox{0.08}{\includegraphics{singular_map}}
\tiny{
	\put(-28,40){$1$}
	\put(-28,5){$1$}
	\put(-12,20){$1$}
	\put(-40,20){$1$}
	\put(-1,45){$2$}
	\put(-48,45){$2$}
	\put(-1,0){$2$}
	\put(-48,0){$2$}}
\end{minipage}=  \begin{minipage}[h]{0.08\linewidth}
\vspace{0pt}
\scalebox{0.08}{\includegraphics{colored_corssing}}
\tiny{
	\put(-1,45){$2$}
	\put(-48,45){$2$}}
\end{minipage}  
- A^4 
\begin{minipage}[h]{0.1\linewidth}
\vspace{0pt}
\scalebox{0.19}{\includegraphics{one_smoothing_color_1}}
\end{minipage}
- A^{-4}
\begin{minipage}[h]{0.1\linewidth}
\vspace{0pt}
\scalebox{0.19}{\includegraphics{one_smoothing_color}}
\end{minipage}
\end{equation*}
Each one of the three skein elements appearing on the right hand side of the previous equation is a link colored with the Jones-Wenzl projector and hence its evaluation is in $\mathbb{Z}[A,A^{-1}]$. This implies that the evaluation of the term on the left hand side is also in $\mathbb{Z}[A,A^{-1}]$.\\ 

In fact, more can be said here in regard of the integrality. Let $L$ be a link. Use (\ref{greatness}) to write the colored Jones polynomial of $L$ as 
\begin{eqnarray}
\label{new}
\tilde{J}(L,n) =
     \displaystyle\sum\limits_{i=0}^{n}C_{n,i}
  S_{n,i}
  \end{eqnarray}
where $S_{n,i}$ is the skein element shown on the right hand side of equation (\ref{greatness}). One can see that the skein elements $S_{n,0}$ and $S_{n,n}$ are links cabled with the $n^{th}$ Jones-Wenzl projector and hence their evaluations in the Kauffman bracket skein module give an element in $\mathbb{Z}[A,A^{-1}]$. In general this is not true for $S_{n,i}$ when $0< i < n$. However, we conjecture that $C_{n,i}S_{n,i} \in \mathbb{Z}[A,A^{-1}] $ for $0\leq i \leq n$.

\end{document}